\documentclass[10pt]{article}
\usepackage{amsmath,amssymb}
\usepackage{color}
\usepackage{mathrsfs}
\usepackage{mathptmx}
\usepackage{verbatim}
\usepackage{epsfig}
\usepackage{CJK}
\usepackage{bm}
\usepackage{amsfonts}
\usepackage{amsthm}
\input{epsf.sty}
\usepackage{epsfig}
\def\be{\begin{equation}}
\def\ee{\end{equation}}
\def\ba{\begin{array}}
\def\ea{\end{array}}

\newtheorem{thm}{Theorem}[section]

\newtheorem{lem}[thm]{Lemma}

\numberwithin{equation}{section}

\newcommand{\mi}{\mathbf{i}}

\def\be{\begin{equation}}
\def\ee{\end{equation}}
\def\br{\begin{eqnarray}}
\def\er{\end{eqnarray}}

\title{Invariant Cantor manifolds of quasi-periodic solutions for the derivative nonlinear Schr$\ddot{\mbox{o}}$dinger equation
}
\author{
 $\mbox{Meina \ Gao}$\\
 $\mbox{College of Arts and Sciences, Shanghai Polytechnic University}$\\ $\mbox{Shanghai 201209,
P R China,}$\\
 $\mbox{mngao@sspu.edu.cn}$\\
        $\mbox{Jianjun \ Liu}$
        \\$\mbox{School of Mathematics, Sichuan University}$\\  $\mbox{ Chengdu 610065, P R
China}$\\ $\mbox{liujj@fudan.edu.cn}$}

\date{}
\begin{document}\maketitle
\textbf{Abstract}: This paper is concerned with the derivative nonlinear Schr$\ddot{\mbox{o}}$dinger equation with periodic boundary conditions
$$\mathbf{i}u_t+u_{xx}+\mathbf{i}\Big(f(x,u,\bar{u}
)\Big)_x=0,\quad x\in\mathbb{T}:=\mathbb{R}/2\pi\mathbb{Z},$$
where $f$ is an analytic function of the form
$$f(x,u,\bar{u})=\mu|u|^2u+f_{\geq4}(x,u,\bar{u}),\quad 0\neq\mu\in\mathbb{R},$$
and $f_{\geq4}(x,u,\bar{u})$ denotes terms of order at least four in $u,\bar{u}$. We show the above equation possesses Cantor families of smooth quasi-periodic solutions of small amplitude. The proof is based on an infinite dimensional KAM theorem for unbounded perturbation vector fields.

\section{Introduction and Main Results}
In this paper, we consider the derivative nonlinear Schr$\ddot{\mbox{o}}$dinger equation with periodic boundary conditions
\begin{equation}\label{1.1}
\mathbf{i}u_t+u_{xx}+\mathbf{i}\Big(f(x,u,\bar{u})
\Big)_x=0,\quad x\in\mathbb{T},\end{equation}
where $f$ is an analytic function of the form
\begin{equation}f(x,u,\bar{u})=
\mu|u|^2u+f_{\geq4}(x,u,\bar{u}),\quad 0\neq\mu\in\mathbb{R},\end{equation}
and $f_{\geq4}(x,u,\bar{u})$ denotes terms of order at least four in $u,\bar{u}$. Moreover, we require
\begin{equation}f_{\geq4}(x,u,\bar{u})=
\frac{\partial F_{\geq5}}{\partial{\bar{u}}}(x,u,\bar{u}),\end{equation}
such that \eqref{1.1} can be viewed as a Hamiltonian system, where $F_{\geq5}(x,u,\bar{u})$ is a real analytic function of order at least five in $u,\bar{u}$. As in \cite{Poschel3}, we may assume $\mu=1$. Then the equation \eqref{1.1} can be regarded as
a perturbation of the following equation:
\begin{equation}\label{1.2}\mathbf{i}u_t+u_{xx}+
\mathbf{i}(|u|^2u)_x=0,\end{equation}
which appears in various physical applications and has been widely studied in the literature.

We study \eqref{1.1} as a Hamiltonian system on some suitable phase space $\mathcal{P}$, for example, we may take $H_0^2(\mathbb{T})$, the usual Soblev space on $\mathbb{T}$ with vanishing average. Under the standard inner product on $L^2(\mathbb{T})$, \eqref{1.1} can be written in the form
\begin{equation}\label{1.3}\frac{\partial u}{\partial t}=-\frac{d}{dx}\frac{\partial H}{\partial \bar{u}}\end{equation}
with the real analytic Hamiltonian
\begin{equation}\label{1.4}H=-\mathbf{i}
\int_{\mathbb{T}}u_x\bar{u}dx+
\frac12\int_{\mathbb{T}}|u|^4dx+
\int_{\mathbb{T}}F_{\geq5}(x,u,\bar{u})dx.\end{equation}

 We will construct Cantor families of time quasi-periodic solutions of small amplitude.
 A result of similar form was firstly obtained in
\cite{Poschel3} by Kuksin and P\"{o}schel for the nonlinear
Schr\"{o}dinger equation with Dirichlet boundary conditions
\begin{equation*}
{\mi}u_t=u_{xx}-mu-f(|u|^2)u,\hspace{12pt}u(t,0)=0=u(t,\pi),
\end{equation*}
where $m$ is real, $f$ is real analytic in some neighborhood of the
origin in $\mathbb{C}$, $f(0)=0$, and $f'(0)\neq0$. For convenience,
we keep fidelity with the notation and terminology from \cite{Poschel3}. Let
$$\phi_j(x)=\frac{1}{\sqrt{2\pi}}e^{{\mi}jx},
\hspace{12pt}j\in\bar{\mathbb{Z}}:
=\mathbb{Z}\setminus\{0\}$$
be the basic modes. For every index set
$$J=\{j_1<j_2<\cdots<j_n\}\subset\bar{\mathbb{Z}},$$
denote by $E_J$ the linear subspace of complex dimension $n$
which is completely foliated into rotational tori
$$E_J=\{u=q_1\phi_{j_1}+\cdots+q_n\phi_{j_n}
:q\in\mathbb{C}^n\}=\bigcup_{I\in\overline{\mathbb{P}^n}}
\mathcal{T}_J(I),$$
where $\mathbb{P}^n=\{I\in\mathbb{R}^n:I_b>0,\ 1\leq b\leq n\}$ is
the positive quadrant in $\mathbb{R}^n$ and
$$\mathcal{T}_J(I)=\{u=q_1\phi_{j_1}+\cdots+q_n\phi_{j_n}
:|q_b|^2=I_b,\ 1\leq b\leq n\}.$$
 The
following is our result for (\ref{1.1}):

\begin{thm}\label{1.1}
For any integer $n\geq2$ and index set
$J=\{j_1<j_2<\cdots<j_n\}\subset\bar{\mathbb{Z}}
$ satisfying
 \begin{equation}\label{18.3.31.2}
 2n-1\nmid\sum_{b=1}^nj_b\end{equation}
 and additionally $j_1j_2<0$ if $n=2$,
 there exist:
\begin{itemize}
\item[{(1)}] a Cantor set $\mathcal{C}\subset\mathbb{P}^n$ with full
density at the origin;
\item[{(2)}] a Lipschitz embedding
$\Psi:\mathcal{T}_J[\mathcal{C}]\hookrightarrow\mathcal{P}$, which
is a higher order perturbation of the inclusion mapping
$\Psi_0:E_J
\hookrightarrow\mathcal{P}$ restricted to
$\mathcal{T}_J[\mathcal{C}]$,
where
$\mathcal{T}_J[\mathcal{C}]:=\bigcup_{I\in
\mathcal{C}}\mathcal{T}_J(I)\subset{E_J}$
is a family of $n$-tori over $\mathcal{C}$,
\end{itemize}
such that the image
$\mathcal{E}_J:=\Psi\big(\mathcal{T}_J[\mathcal{C}]\big)$ is a
Cantor manifold of diophantine $n$-tori for the derivative nonlinear
Schr\"{o}dinger equation (\ref{1.1}). Moreover, the restriction
of $\Psi$ to each torus $\mathcal{T}_J(I)$, $I\in\mathcal{C}$ is
smooth, and $\mathcal{E}_J$ has a tangent space at the origin equal
to $E_J$.
\end{thm}

The above theorem is proved by using infinite dimensional KAM theory.
Historically, KAM theory for partial differential equations was originated by Kuksin \cite{Kuksin4} and Wayne \cite{Wayne}, where one dimensional nonlinear wave and Schr\"{o}dinger equations under Dirichlet boundary conditions were studied. Owing to the special boundary conditions and the absence of
spatial derivatives in the nonlinearity, the corresponding infinite dimensional Hamiltonian systems have simple normal frequencies and bounded perturbations. Then infinite dimensional KAM theory for bounded perturbations was deeply investigated, including both simple and multiple normal frequencies. See
\cite{Berti7,Berti4,Bourgain1,Bourgain2,
CYou,Criag-Wayne,Eliasson1,Eliasson,
GY,GY1,Gre1,Gre,K4,K5,Poschel3,procesi,P,P2,Y,Y3} for example.
(We can not list all papers in this field.)

Since the nonlinearity
of (\ref{1.1}) contains a spatial derivative, a suitable KAM theorem for unbounded perturbation vector field is required for our result. The first KAM
theorem for unbounded perturbations is due to Kuksin \cite{Kuksin3,Kuksin1}. Under the assumption $0<\delta<d-1$ ($d$ is the order
of the linear vector field and $\delta$ is the order of the perturbation vector field), a suitable estimate, which is now called Kuksin's lemma, is proved for the small-denominator equation with large variable coefficient. By using this estimate, a
KAM theorem is established to prove the persistence of the finite-gap solutions of the KdV equation. For the case $0<\delta<d-1$, also see \cite{KP} for perturbed KdV equations by Kappeler and P\"{o}schel, and \cite{Bambusi} for a class of time dependent Schr$\ddot{\mbox{o}}$dinger operators by Bambusi and Graffi. More recently, KAM theory for unbounded
perturbations has been extended by Liu-Yuan \cite{L-Y2,L-Y1}. The small-denominator equation
with large variable coefficient is suitably estimated for the limiting case $0<\delta=d-1$, and consequently
the corresponding KAM theorems are established with applications to quantum Duffing oscillator, derivative nonlinear Schr$\ddot{\mbox{o}}$dinger and Benjamin-Ono equations. Also see Zhang-Gao-Yuan \cite{Zhang} for reversible derivative nonlinear Schr$\ddot{\mbox{o}}$dinger equations.
For the case $\delta>d-1$, both Kuksin's Lemma and Liu-Yuan's estimate are invalid, and thus there is no general KAM theorem in which the perturbation only satisfies smallness condition.

However, this does not mean that nothing is known for partial differential equations with their nonlinearity containing higher order space derivative. Actually, a great progress has been made for quasi-linear and even fully nonlinear partial differential equations recently. See Baldi-Berti-Montalto \cite{Baldi2,Baldi1,Baldi3} for KdV and mKdV equations, Berti-Montalto \cite{Berti5} for water wave equations, Feola-Procesi \cite{Feola} for reversible derivative nonlinear Schr$\ddot{\mbox{o}}$dinger equations, and Montalto \cite{Montalto} for Kirchoff equation.
 The idea of the proof is to use pseudo-differential calculus in order to conjugate the original system to a system with a smoothing perturbation and then to apply KAM theory.
Also see Bambusi \cite{Bambusi1,Bambusi2}
for the reducibility of Schr\"{o}dinger operators.

Moreover, derivative nonlinear wave equations were studied by Berti-Biasco-Procesi \cite{Berti2,Berti1}. The idea is to prove
first order asymptotic expansions of the perturbed normal frequencies by exploiting the quasi-T$\ddot{\mbox{o}}$plitz property and penalize high-momentum terms by introducing weighted norms.

On the other hand, both Kuksin's Lemma and Liu-Yuan's estimate are only valid for scalar homological equations, which requires the normal frequencies to be simple, that is, $\Omega_j^\sharp=1$. This implies the range of applications of the
previous KAM theorems for unbounded perturbations pertain to those
PDEs with simple frequencies. This precludes the derivative
nonlinear Schr$\ddot{\mbox{o}}$dinger equation (\ref{1.1}) with periodic
boundary conditions, since the multiplicity $\Omega_j^\sharp=2$.

For \eqref{1.1} with its nonlinearity being gauge invariant and not containing $x$ explicitly, see Liu-Yuan \cite{L-Y}. The above difficulty was avoided by using momentum conservation, which guarantees the indices of the monomials
\begin{equation*}\label{18.3.28.3}e^{\mathbf{i}
k_1x_1+\cdots+\mathbf{i}k_nx_n}y_1^{m_1}\cdots y_n^{m_n}\prod_{j\in\bar{Z}\setminus J}
z_j^{l_j}\bar{z}_j^{\bar{l}_j}\end{equation*}
satisfying
\begin{equation}\label{18.3.28.2}\sum_{1\leq b\leq n}k_bj_b
+\sum_{j\in\bar{Z}\setminus J}(l_j-\bar{l}_j)j=0.\end{equation}
  Take the most difficult terms in the KAM iteration scheme: $e^{\mathbf{i}k\cdot x}z_i\bar{z}_j,k\in\mathbb{Z}^n,i,j\in\bar{\mathbb{Z}}
 \setminus J$, where $k\cdot x=k_1x_1+\cdots+k_nx_n$.
Following Bourgain's observation in \cite{Bourgain3},
the restriction \eqref{18.3.28.2} means that $|i|+|j|$ is controlled by $|k|$ unless $i=j$. Hence, for a fixed $k$, all the nearly resonant terms except $e^{\mathbf{i}k\cdot x}z_j\bar{z}_j$ can be eliminated. As a result, only $e^{\mathbf{i}k\cdot x}z_j\bar{z}_j$ are left as normal form terms. The homological equations are then scalar, and the estimate in \cite{L-Y2} for small-denominator equation with large variable coefficients still works.

In the present paper, the nonlinearity contains the space variables $x$ explicitly so that \eqref{18.3.28.2} is not true, and thus  $e^{\mathbf{i}k\cdot x}z_{-j}\bar{z}_{j}$ are difficult to handle. Luckily, a key observation for \eqref{1.1} provides a chance to solve this problem.
After introducing action-angle coordinates to Birkhoff normal form of order four and choosing parameters properly, the normal frequencies $\Omega_j$ take the form (see \eqref{17.10.19.2})
\begin{equation}\label{18.3.28.4}\Omega_j=j^2+cj,\quad j\in\bar{\mathbb{Z}}\setminus J,\end{equation}
where
\begin{equation}\label{18.4.1.1}
c=\frac{2\sum_{b=1}^n\xi_b}{2n-1}.
\end{equation}
This formally indicates that $\Omega_{-j}$ and $\Omega_j$ do not coincide, and even $|\Omega_{-j}-\Omega_j|\rightarrow\infty$ as $j\rightarrow\infty$. Also see \cite{Liu}
for this observation.

 However, the problem is more complicated than it appears because of these two characters of $c$: first, in view of \eqref{18.4.1.1} \eqref{18.3.31.1} \eqref{18.4.1.2},
 \begin{equation}c\approx\epsilon^{\frac67},\end{equation}
 which indicates that $c$ is rather small;
 second, $c$ depends on the parameters with Lipschitz semi-norm
 \begin{equation}|c|^{lip}=\frac{2n}{2n-1},\end{equation}
 which is not small. In the following, for these two characters respectively, we will introduce the corresponding difficulties and the methods to overcome them.

In order to eliminate $e^{\mathbf{i}k\cdot x}z_{-j}\bar{z}_{j}$, $k\in\mathbb{Z}^n,\pm j\in\bar{\mathbb{Z}}\setminus J$, the corresponding small-divisors are
\begin{equation}\label{18.4.1.3}\langle k,\omega\rangle+\Omega_{-j}-\Omega_j.\end{equation}
By the usual method of measure estimate, roughly determined by
\begin{equation}|\Omega_{-j}-
\Omega_j|<|k||\omega|_{\mathcal{O}},\end{equation}
for a fixed $k$, the number of small-divisors is
\begin{equation}\frac{|k||\omega|_{\mathcal{O}}}{2c},
\end{equation}
which is out of control for $c$ too small. Thus, an important fact in Theorem \ref{thm7.12.1} is that,
  $|\Omega_{-j}-\Omega_j|$ is not required as ``frequency asymptotic'' by assumption (A) as usual, but required as ``small-divisor'' by assumption (C).
In this sense, Theorem \ref{thm7.12.1} can not be viewed as a usual unbounded KAM theorem with simple normal frequencies.

In order to solve the above problem, i.e., controlling the number of small-divisors \eqref{18.4.1.3}, we introduce momentum majorant norm (see \eqref{2.1.16}). At the $\nu$-th KAM step, we only eliminate $e^{\mathbf{i}k\cdot x}z_{-j}\bar{z}_{j}$ with lower momentum, roughly that is,
\begin{equation}|\sum_{1\leq b\leq n}k_bj_b
-2j|<|\ln\epsilon_{\nu+1}|.\end{equation}
Therefore, for a fixed $k$, the number of small-divisors is
\begin{equation}\label{18.4.1.4}|k|\max_{1\leq b\leq n}|j_b|+|\ln\epsilon_{\nu+1}|,\end{equation}
which is accepted in the KAM iteration. Consequently, $e^{\mathbf{i}k\cdot x}z_{-j}\bar{z}_{j}$ with lower momentum can be eliminated. On the other hand,  $e^{\mathbf{i}k\cdot x}z_{-j}\bar{z}_{j}$ with higher momentum are put into the perturbation. As a result, only $e^{\mathbf{i}k\cdot x}z_{j}\bar{z}_{j}$ are left as normal forms and the homological equations are  scalar.

Nevertheless, owing to the momentum majorant norm, the estimate for small-denominator equation with large coefficient in \cite{L-Y2} does not work directly.
By properly designing momentum weight and analyticity width, the momentum majorant norm and the sup-norm of a function can be controlled by each other
(see \eqref{17.11.3.2} \eqref{17.11.4.2}). So we obtain a lemma which is Theorem 1.4 in \cite{L-Y2} with the momentum majorant norm estimate instead of the sup-norm estimate (see Lemma \ref{lem17.7.31.1}).

Recall the second character of $c$ mentioned above, that is, $c$ depends on the parameters with  $|c|^{lip}$ not small. Then the Lipschitz semi-norm $|\Omega|^{lip}_{-1}$ is not small, and thus we are not able to get the twist of $\langle k,\omega(\xi)\rangle+\langle l,\Omega(\xi)\rangle$ by the Lipschitz continuity of $\omega$ in both directions as usual. To overcome this problem, we add assumption (C) in the KAM theorem.
Therefore, we must verify: assumption (C) is preserved under KAM iteration; assumption (C) is satisfied for the derivative nonlinear Schr$\ddot{\mbox{o}}$dinger equation \eqref{1.1}.
The former is relatively trivial; while the latter is rather complicated, and some restrictions (see \eqref{18.3.31.2} and additionally $j_1j_2<0$ if $n=2$) on the index set $J$ are necessary. As a counterexample, taking $n=2$, $j_1=-1,j_2=1$, then the small-divisor
\begin{equation}4\omega_1-4\omega_2+\Omega_3-\Omega_{-3}
\equiv0.
\end{equation}
Of course, the above restrictions on $J$ can be more flexible by more discussions.

We now lay out an outline of the present paper:

  In Section 2, we introduce the momentum majorant norm of the vector field and then formulate the KAM theorem. In our theorem,
 the Lipschitz semi-norm of $\omega^{-1}$ is not required as usual; moreover, assumption (C) is added, which means that for any fixed $k\in\mathbb{Z}^n,|l|\leq2$ the small-divisor
 $\langle k,\omega\rangle+\langle l,\Omega\rangle$ is big by itself or has a big twist.

 Section 3 contains the proof of Theorem \ref{1.1}.
We transform the Hamiltonian into a partial Birkhoff normal form up to order four with estimates of momentum majorant norm; then introduce action-angle coordinates for tangential variables and extract parameters by amplitude-frequency modulation; finally Theorem \ref{1.1} is achieved by using the KAM theorem. Therein
  many efforts are paid to verify assumption (C) in the KAM theorem, seeing the proof of Lemma \ref{lem18.3.27.1}, where the condition
  $j_1j_2<0$ for $n=2$ is used in subcase 2.2
   and the condition $2n-1\nmid\sum_{b=1}^nj_b$ is used in subcase 2.3.

    In Section 4-7, the KAM theorem is proved. In Section 4, we derive the homological equations and prove two lemmas (Lemma \ref{lem17.7.31.1}, Lemma \ref{lem17.9.22.1}) for solving them.
   For $k\in\mathbb{Z}^n,i\neq \pm j$, the homological equations are solved in the same way as \cite{L-Y1};
    for $
    k\in\mathbb{Z}^n,i=-j$, the lower momentum terms are eliminated by Lemma \ref{lem17.7.31.1} and the higher momentum terms are left as perturbation. In Section 5, the new Hamiltonian including normal form and perturbation are estimated. In Section 6, by choosing iterative parameters properly, we prove the iterative lemma and the convergence. Therein the transformation and its derivative are estimated by sup-norm as usual.
     In Section 7, the measure of excluded parameters   is estimated by Lemma
     \ref{lem18.4.1.2} for $k\in\mathbb{Z}^n,i=-j$ and Lemma \ref{lem18.4.1.3} for the others. Therein we design $\alpha_{2,\nu}\rightarrow0$ because of \eqref{18.4.1.4}$\rightarrow\infty$ as $\nu\rightarrow\infty$.

  Section 8 contains several lemmas: Lemma \ref{lem18.4.1.1} provides three elementary inequalities being frequently used in this paper; Lemma \ref{lem18.3.19.1} shows that the $|\cdot|_{s,\tau+1}$ norm of a function is controlled by the momentum majorant norm;
   Lemma \ref{lem17.9.7.1} establishes a bridge of the momentum majorant norm between a vector field and its elements; Lemma \ref{lem17.12.22.1}
    is an estimate for the momentum majorant norm of
the commutator of two vector fields;
   Lemma \ref{lem17.12.23.1} is an estimate for the momentum majorant norm of the transformed Hamiltonian vector field.

Finally we remark that the higher order nonlinearity $f_{\geq4}(x,u,\bar{u})$
only contributes to perturbation. Actually, by using the same KAM theorem, Cantor families of quasi-periodic solutions can be obtained for \eqref{1.2} with more general perturbations of the form $\mathbf{i}\frac{d}{dx}\frac{\partial K}{\partial{\bar{u}}}$, where $K$ is a real analytic Hamiltonian with $\frac{\partial K}{\partial{\bar{u}}}$ being bounded and some other conditions. Of course, the above procedure is invalid for \eqref{1.2} with quasi-linear or fully nonlinear perturbations.
We hope to decrease the order of derivatives with the help of the ideas in \cite{Baldi1}, such that the above procedure is valid with some modifications.

\section{A KAM Theorem}

Recall the index set $J=\{j_1<\cdots<j_n\}\subset\bar{\mathbb{Z}}$,
denote $\mathbb{Z}_*:=\bar{\mathbb{Z}}\setminus J$. For $a,p\in\mathbb{R}$, we define the Hilbert space $\ell^{a,p}_J$ of all complex sequences $z=(z_j)_{j\in\mathbb{Z}_*}$ with $$||z||_{a,p}^2=\sum_{j\in\mathbb{Z}_*}e^{2a|j|}|j|^{2p}|z_j|^2<\infty.$$
We consider the direct product
\begin{equation}\label{17.9.18.11}\mathcal{P}^{a,p}:=\mathbb{C}^n\times\mathbb{C}^n\times
\ell^{a,p}_J\times\ell^{a,p}_J\end{equation}
 endowed with $(s,r)$-weighted norm
\begin{equation}\label{17.9.18.1}v=(x,y,z,\bar{z})\in\mathcal{P}^{a,p},\quad ||v||_{s,r,p}=\frac{|x|}{s}+\frac{|y|_1}{r^2}+
\frac{||z||_{a,p}}{r}+
\frac{||\bar{z}||_{a,p}}{r},\end{equation}
where $0<s,r<1$, and $|x|:=\max_{1\leq h\leq n}|x_h|$, $|y|_1:=\sum_{h=1}^n|y_h|$. In the whole of this paper the parameter $a$ is fixed and thus we drop it in the notation $||\cdot||_{s,r,p}$. Note that $z$ and $\bar{z}$ are independent variables. As phase space, we consider the toroidal domain
\begin{equation}\label{17.9.19.1}D(s,r):=\mathbb{T}^n_s\times D(r):=\mathbb{T}^n_s\times B_{r^2}\times B_r\times B_r\subset\mathcal{P}^{a,p},\end{equation}
where
$\mathbb{T}^n_s:=\{x\in\mathbb{C}^n:\mbox{Re}x
\in\mathbb{T}^n:=\mathbb{R}^n/2\pi\mathbb{Z}^n,
\max_{1\leq h\leq n }|\mbox{Im}\ x_h|<s\}, B_{r^2}:=\{y\in\mathbb{C}^n:
|y|_1<r^2\}$
and $B_r\subset\ell_J^{a,p}$ is the open ball of radius $r$ centered at zero.

For $q\in\mathbb{R}$, we consider vector fields $X:D(s,r)\rightarrow\mathcal{P}^{a,q}$ of the form
\begin{equation}\label{2.1.2}X(v)=(X^{(\tilde{x})}(v),X^{(\tilde{y})}(v),
X^{(\tilde{z})}(v),X^{(\bar{{\tilde z}})}(v))\in \mathcal{P}^{a,q},\end{equation}
where $v\in D(s,r)$ and $X^{(\tilde{x})}(v), X^{(\tilde{y})}(v)\in\mathbb{C}^n$, $X^{(\tilde{z})}(v),X^{(\bar{{\tilde z}})}(v)\in\ell_J^{a,q}$. Denote
\begin{equation}V:=\{\tilde{x}_1,\cdots,\tilde{x}_n,\tilde{y}_1,\cdots, \tilde{y}_n,
\cdots,\tilde{z}_j,\cdots,\bar{{\tilde z}}_j,\cdots\},\quad j\in\mathbb{Z}_*.\end{equation}
We can write $X(v)=(X^{\mathbf{v}}(v))_{\mathbf{v}\in V}$, where each component is a formal scalar power series
\begin{equation}\label{17.7.21.17}X^{(\mathbf{v})}(v)=\sum_{(k,i,\alpha,\beta)\in\mathbb{I}}
X^{(\mathbf{v})}_{k,i,\alpha,\beta}e^{\mathbf{i}k\cdot x}y^{i}z^{\alpha}\bar{z}^{\beta}\end{equation}
with coefficients $X^{(\mathbf{v})}_{k,i,\alpha,\beta}\in\mathbb{C}$ and multi-indices in
\begin{equation}\label{2.1.6}\mathbb{I}:=\mathbb{Z}^n
\times\mathbb{N}^n\times
\mathbb{N}^{\mathbb{Z}_*}\times\mathbb{N}^{\mathbb{Z}_*},\end{equation}
where $\mathbb{N}^{\mathbb{Z}_*}:=\{\alpha=(\alpha_j)_{j\in\mathbb{Z}_*}\in
\mathbb{N}^{\mathbb{Z}_*}\mbox{with}\ |\alpha|=\sum_{j\in\mathbb{Z}_*}\alpha_j<\infty\}$. In \eqref{2.1.6}, we use the standard multi-indices notation $z^{\alpha}\bar{z}^{\beta}:=\prod_{j\in\mathbb{Z}_*}z_j^{\alpha_j}
\bar{z}_j^{\beta_j}$. The formal vector field $X$ is absolutely convergent in $\mathcal{P}^{a,q}$ (with norm \eqref{17.9.18.1} for $q$ instead of $p$) at $v\in D(s,r)$ if every component $X^{(\mathbf{v})}(v),\mathbf{v}\in V$ is absolutely convergent and $||(X^{(\mathbf{v})}(v))_{\mathbf{v}\in V}||_{s,r,q}<+\infty$.

We also use the differential geometry notation
\begin{equation}\label{17.9.18.2}X(v)=\sum_{\mathbf{v}\in
V}X^{(\mathbf{v})}\partial_{\mathbf{v}}=\sum_{\mathbf{v}\in V}\sum_{(k,i,\alpha,\beta)\in\mathbb{I}}X^{(\mathbf{v})}_{k,i,\alpha,\beta}
e^{\mathbf{i}k\cdot x}
y^iz^{\alpha}\bar{z}^{\beta}\partial_{\mathbf{v}}.\end{equation}
For a scalar monomial $e^{\mathbf{i}k\cdot x}y^iz^{\alpha}\bar{z}^{\beta}$, we define its momentum as
\begin{equation}\label{2.1.14}\pi(k,\alpha,\beta)=\sum_{b=1}^nk_bj_b+
\sum_{j\in\mathbb{Z}_*}(\alpha_j-\beta_j)j,\end{equation}
 and for a vector field monomial $e^{\mathbf{i}k\cdot x}
y^iz^{\alpha}\bar{z}^{\beta}\partial_{\mathbf{v}}$,
we define its momentum as
\begin{equation}\label{2.1.13}\pi(k,\alpha,\beta;\mathbf{v}):=\begin{cases}
\pi(k,\alpha,\beta),&\mbox{if}\ \ \mathbf{v}\in\{\tilde{x}_1,\cdots,\tilde{x}_n,
\tilde{y}_1,\cdots, \tilde{y}_n\},\\ \pi(k,\alpha,\beta)- j,&
\mbox{if}\ \ \mathbf{v}=\tilde{z}_j,\\
\pi(k,\alpha,\beta)+j, & \mbox{if}\ \ \mathbf{v}=\bar{\tilde{z}}_j.\end{cases}\end{equation}
We say that a vector field $X$ satisfies momentum conservation if and only if it is a linear combination of monomial vector fields with zero momentum.

Similarly to \cite{Berti1}, for a formal vector field $X$ in \eqref{17.9.18.2}, we define its momentum majorant norm on $D(s,r)$ as
\begin{equation}\label{2.1.16}||X||_{s,r,q,\mathbf{a}}:
=\sup_{(y,z,\bar{z})\in D(r)}||(\sum_{k,i,\alpha,\beta}e^{\mathbf{a}|\pi(k,\alpha,\beta;v)|}
|X_{k,i,\alpha,\beta}^{(\mathbf{v})}|e^{|k|s}|y|^i|z^{\alpha}||
\bar{z}^{\beta}|)_{\mathbf{v}\in V}||_{s,r,q},\end{equation}
where $\mathbf{a}\geq0$ and $|k|:=|k_1|+|k_2|+\cdots+|k_n|$.
Furthermore, if $X$ depends on parameters $\xi\in\mathcal{O}\subset\mathbb{R}^n$, we define the $\lambda$-Lipschitz (momentum majorant) norm ($\lambda\geq0$):
%
%
\begin{eqnarray}\nonumber||X||^{\lambda}_{s,r,q,\mathbf{a};
\mathcal{O}}&:=&||X||_{s,r,q,\mathbf{a};\mathcal{O}}+
\lambda||X||_{s,r,q,\mathbf{a}
;\mathcal{O}}^{lip}\\
\label{17.9.19.2}&:=&\sup_{\xi\in\mathcal{O}}||X(\xi)||
_{s,r,q,\mathbf{a}}+\lambda\sup_{\xi,\zeta\in\mathcal{O}
\atop{\xi\neq\zeta}}\frac{||\Delta_{\xi\zeta}X||_{s,r,q,\mathbf{a}}}
{|\xi-\zeta|},\end{eqnarray}
where $\Delta_{\xi\zeta}X=X(\cdot;\xi)-X(\cdot;\zeta)$.

Similarly, we define the $\lambda$-Lipschitz sup-norm:
\begin{eqnarray}\nonumber||X||^{\lambda}_{s,r,q;D(s,r)\times
\mathcal{O}}&:=&||X||_{s,r,q;D(s,r)\times\mathcal{O}}+
\lambda||X||_{s,r,q
;D(s,r)\times\mathcal{O}}^{lip}\\
\label{17.12.19.1}&:=&\sup_{(v;\xi)\in D(s,r)\times\mathcal{O}}||X(v;\xi)||
_{s,r,q}+\lambda\sup_{\xi,\zeta\in\mathcal{O}
\atop{\xi\neq\zeta}}\sup_{D(s,r)}\frac{||\Delta_{\xi\zeta}X||_{s,r,q}}
{|\xi-\zeta|}.\end{eqnarray}
Obviously, we have
\begin{equation}\label{17.11.8.2}||X||_{s,r,q;D(s,r)}\leq
||X||_{s,r,q,\mathbf{a}},\end{equation}
\begin{equation}||X||^{\lambda}_{s,r,q;D(s,r)\times
\mathcal{O}}\leq||X||^{\lambda}_{s,r,q,\mathbf{a};
\mathcal{O}}.\end{equation}

Now consider small perturbations $H=N+P$ of an infinite dimensional Hamiltonian in the parameter dependent normal form
\begin{equation}\label{2.1}N=\sum_{1\leq b\leq n}\sigma_{j_b}\omega_b(\xi)y_b+\sum_{j\in\mathbb{Z}_*}\sigma_{j}
\Omega_j(\xi)z_j\bar{z}_j\end{equation}
defined on the phase space $\mathcal{P}^{a,p}$ ($a\geq0,p\geq0$)
 with the symplectic structure
 \begin{equation}\label{7.13.4}\sum_{1\leq b\leq n}\sigma_{j_b}dx_b\wedge dy_b-\mathbf{i}\sum_{j\in\mathbb{Z}_*}\sigma_jdz_j\wedge d\bar{z}_j,\end{equation}
 where $\sigma_j=1$ for $j>0$ and $\sigma_j=-1$ for $j<0$. The tangential frequencies $\omega:=(\omega_1,\cdots,\omega_n)$ and the normal frequencies $\Omega:=(\Omega_j)_{j\in\mathbb{Z}_*}$ are real vectors depending on real parameters $\xi\in\mathcal{O}\subset\mathbb{R}^n$, $\mathcal{O}$ a closed bounded
 set of positive Lebesgue measure, and roughly
  $$\Omega_j(\xi)=j^2+\cdots.$$

  The perturbation term $P$ is real analytic in the space coordinates and Lipschitz in the parameters. Moreover, for each $\xi\in\mathcal{O}$ its Hamiltonan
  vector field
  \begin{equation}\label{7.13.2} X_P=((\sigma_{j_b}P_{y_b})_{1\leq b\leq n}, -(\sigma_{j_b}P_{x_b})_{1\leq b\leq n},-\mathbf{i}(\sigma_jP_{\bar{z}_j})_{j\in\mathbb{Z}_*},\mathbf{i}(\sigma_jP_{z_j})_{j\in\mathbb{Z}_*})^T\end{equation}
  defines in the neighbourhood of  $\mathcal{T}_0=\mathbb{T}^n\times\{y=0\}\times\{z=0\}\times\{\bar{z}=0\}$, that is $D(s,r)$ in \eqref{17.9.19.1}, a real analytic map
   \begin{equation}\label{7.13.1}X_P:\mathcal{P}^{a,p}\rightarrow
   \mathcal{P}^{a,q}.\end{equation}

Similarly to the Lipschitz-norm of the vector filed in \eqref{17.9.19.2},
 the Lipschitz semi-norms of the frequencies $\omega$ and $\Omega$ are defined as

 \begin{equation}|\omega|_{\mathcal{O}}^{lip}=\sup_{\xi,\zeta\in
 \mathcal{O}\atop{\xi\neq\zeta}}\frac{|\Delta_{\xi\zeta}\omega|}{|\xi-\zeta|},
\quad |\Omega|_{-\delta,\mathcal{O}}^{lip}=\sup_{\xi,\zeta\in\mathcal{O}
\atop{\xi\neq\zeta}}
\sup_{j\in\mathbb{Z}_*}|j|^{-\delta}\frac{|\Delta_{\xi\zeta}\Omega_j|}
{|\xi-\zeta|}\end{equation}
for any real number $\delta$.

 \begin{thm} \label{thm7.12.1}Suppose the normal form $N$ in \eqref{2.1} described above satisfies the following
assumptions:
\begin{itemize}
\item[(A)] There exists a constant $m>0$ such that
\begin{equation}|\Omega_i-\Omega_j|\geq m|i^2-j^2|,\end{equation}
for all $i,j\in\mathbb{Z}_*\cup\{0\}$ uniformly on $\mathcal{O}$. Here $\Omega_0=0$;
\item[(B)]The map $\xi\mapsto\omega(\xi)$ between $\mathcal{O}$ and its image is Lipschitz continuous, i.e. there exist a positive constant $M_1$ such that $|\omega|_{\mathcal{O}}^{lip}\leq M_1$;
    there exists $\delta\leq1$ such that the functions $\xi\mapsto\frac{\Omega_j(\xi)}{j^{\delta}}$ are uniformly Lipschitz
on $\mathcal{O}$ for $j\in\mathbb{Z}_*$, i.e. there exist a positive constant $M_2$ such that $|\Omega|_{-\delta,\mathcal{O}}^{lip}\leq M_2$;

\item[(C)]
There exists a constant $M_3>0$ such that, for every $k\in\mathbb{Z}^n$ and $l\in\mathbb{Z}^{\mathbb{Z}_*}$ with $|l|:=\sum_{j\in\mathbb{Z}_*}|l_j|\leq2$, the small divisor
$D_{kl}(\xi):=\langle k,\omega(\xi)\rangle+\langle l,\Omega(\xi)\rangle$
satisfies
\begin{equation}\label{18.3.18.3}\inf_{\xi\in\mathcal{O}}|D_{kl}(\xi)|
+\inf_{\xi-\zeta// v_{kl}}\frac{|\Delta_{\xi\zeta}
    D_{kl}|}{|\xi-\zeta|}\geq M_{3}\max\{|k|,\sum_{j\in\mathbb{Z}_*}|jl_j|\},
\end{equation}
     where $v_{kl}$ is a unit vector in $\mathbb{R}^n$ depending on $k$ and $l$, and the notation ``$\xi-\zeta// v_{kl}$'' means ``for all $\xi,\zeta\in\mathcal{O}$ with $\xi-\zeta$ parallelling to $v_{kl}$''.
\end{itemize}
Set $M=M_1+M_2$. Then for every $\beta>0$, there exists a positive constant $\gamma$, depending only on $n,m$, the frequencies $\omega$, $\Omega$,
$s>0$ and $\beta$, such that for every perturbation term $P$ described above with
\begin{equation}p-q\leq1\end{equation}
and
\begin{equation}\label{17.7.21.20}\epsilon:=||X_P||_{s,r,q,\mathbf{a};
\mathcal{O}}+\frac{\alpha}{M}||X_P||_{s,r,q,\mathbf{a};
\mathcal{O}}^{lip}\leq
(\alpha\gamma)^{1+\beta}\end{equation}
for some $r>0$ and $0<\alpha<1$, there exist:
\begin{itemize}
\item[(1)] a Cantor set $\mathcal{O}_{\alpha}\subset\mathcal{O}$ with
\begin{equation}\label{17.7.21.21}|\mathcal{O}\setminus\mathcal{O}_{\alpha}|\leq c_1\rho^{n-1}\alpha,\end{equation}
where $|\cdot|$ denotes the Lebesgue measure, $\rho:=\mbox{diam}\mathcal{O}$ represents the diameter of $\mathcal{O}$, and $c_1>0$ is a constant depends on $n,J,\omega,\Omega$;
\item[(2)] a Lipschitz family of smooth torus embeddings $\Phi:\mathbb{T}^n\times\mathcal{O}_{\alpha}\rightarrow
    \mathcal{P}^{a,p}$ satisfying: for every non-negative integer multi-index $k=(k_1,\cdots,k_n),$
    \begin{equation}\label{17.7.21.22}
    ||\partial_{x}^k(\Phi-\Phi_0)||_{s,r,p;\mathbb{T}^n
    \times\mathcal{O}_{\alpha}}+
    \frac{\alpha}{M}||\partial_x^k(\Phi-\Phi_0)||_{s,r,p;
    \mathbb{T}^n\times\mathcal{O}_{\alpha}}^{lip}\leq
     c_2\epsilon^{\frac{1}{1+\beta}}/\alpha,\end{equation}
     where $\partial_x^k:=\frac{\partial^{|k|}}{\partial x_1^{k_1}\cdots\partial x_n^{k_n}}$,
     $$\Phi_0:\mathbb{T}^n\times\mathcal{O}\rightarrow\mathcal{T}_0,\quad (x,\xi)\mapsto(x,0,0,0)$$
     is the trivial embedding for each $\xi$, and $c_2$ is a positive constant which depends on $k$ and the same parameter as $\gamma$;
\item[(3)] a Lipschitz map $\phi:\mathcal{O}_{\alpha}\rightarrow\mathbb{R}^n$ with
\begin{equation}\label{17.7.21.23}|\phi-\omega|_{\mathcal{O}_{\alpha}}+\frac{\alpha}{M}|\phi-\omega|_{\mathcal{O}_{\alpha}}^{lip}\leq c_3\epsilon,\end{equation}
where $c_3$ is a positive constant which depends on the same parameter as $\gamma$,
\end{itemize}
such that for each $\xi\in\mathcal{O}_{\alpha}$ the map $\Phi$ restricted to $\mathbb{T}^n\times\{\xi\}$ is a smooth embedding of a rotational torus with frequencies $\phi(\xi)$ for the perturbed Hamiltonian $H$ at $\xi$. In other words,
$$t\mapsto\Phi(\theta+\phi(\xi),\xi),\quad t\in\mathbb{R}$$
is a smooth quasi-periodic solution for the Hamiltonian $H$ evaluated at $\xi$ for every $\theta\in\mathbb{T}^n$ and $\xi\in\mathcal{O}_{\alpha}$.
\end{thm}

%
%
%
%

\section{Proof of Theorem 1.1}
In the first subsection, we write the derivative nonlinear Schr$\ddot{\mbox{o}}$dinger equation \eqref{1.1} in Hamiltonian form of infinitely coordinates, and then transform it into a partial Birkhoff normal form up to order four. In the second subsection, we prove Theorem \ref{1.1} by using Theorem \ref{thm7.12.1}.

\subsection{Birkhoff Normal Form}

We introduce for any $a\geq0$ and $\tilde{p}>3/2$ the phase space
$$\mathcal{H}_{0}^{a,\tilde{p}}=\{u\in L^2(\mathbb{T}):\hat{u}(0)=0,||u||_{a,\tilde{p}}^2
=\sum_{j\in\bar{\mathbb{Z}}}|\hat{u}(j)|^2|j|^{2\tilde{p}}e^{2a|j|}<\infty\}$$
of complex valued functions on $\mathbb{T}$, where
$$\hat{u}(j)=\int_{0}^{2\pi}u(x)\phi_{-j}(x)dx.$$
To write \eqref{1.3} in infinitely many coordinates, we make the ansatz
\begin{equation}\label{3.1}u(t,x)=\sum_{j\in\bar{\mathbb{Z}}}
\gamma_jq_j(t)\phi_j(x),\end{equation}
where $\gamma_j=\sqrt{|j|}$. The coordinates are taken from the Hilbert space $\bar{\ell}^{a,p}$ of all complex-valued
sequences $q=(q_j)_{j\in\bar{\mathbb{Z}}}$ with finite norm
$$||q||_{a,p}^2=\sum_{j\in\bar{\mathbb{Z}}}|q_j|^2
|j|^{2p}e^{2a|j|}<\infty,$$
where $p=\tilde{p}+\frac12$. We remark that $\bar{\ell}^{a,p}$ is $\ell_J^{a,p}$ with $J=\emptyset$. In the following, for convenience the notation $\in\bar{\mathbb{Z}}$ is abbreviated as ``$\neq 0$'' or omitted. Now \eqref{1.3} can be rewritten as
\begin{equation}\label{3.2}\dot{q}_j=-\mathbf{i}\sigma_j\frac{\partial H}{\partial \bar{q}_j},\quad
\sigma_j=\begin{cases}1,\ &j\geq1\\ -1,\ &j\leq-1\end{cases}\end{equation}
with the Hamiltonian
\begin{equation}\label{3.3}H=\Lambda+G+K,\end{equation}
where \begin{equation}\label{3.4}\Lambda=
\sum_{j\neq0}\sigma_jj^2|q_j|^2,\end{equation}
\begin{equation}\label{3.5}G=\frac{1}{4\pi}
\sum_{j,k,l,m\neq0\atop{j-k+l-m=0}}
\gamma_j\gamma_k\gamma_l\gamma_mq_j
\bar{q}_kq_l\bar{q}_m,\end{equation}
\begin{equation}K=\int_{\mathbb{T}^n}
F_{\geq5}(x,\sum_{j\in\bar{\mathbb{Z}}
}\gamma_jq_j\phi_j,\sum_{j\in\bar{\mathbb{Z}}
}\gamma_j\bar{q}_j\phi_{-j})dx.\end{equation}
Now we consider the $4$-order term $G$. The normal form part of $G$ is \eqref{3.5} with $j=k$ or $j=m$, that is
\begin{equation}\label{3.7}B=\frac{1}{4\pi}\sum_{j\neq 0}j^2|q_j|^4+\frac{1}{2\pi}\sum_{j,l\atop{j\neq l}}|jl||q_j|^2|q_l|^2.\end{equation}
Fix a positive integer $N$. Define the index set
$$\Delta=\{{(j,k,l,m)}\in\mathbb{Z}^4:j,k,l,m\neq0,j-k+l-m=0,j\neq k,m\},$$
$$\Delta_1=\{{(j,k,l,m)}\in\Delta:\mbox{There are at least 2 compoment in}\{\pm1,\cdots,\pm N\}\}.$$
Split the non normal form part of $G$ into two parts:
\begin{equation}\label{3.8}Q_1=\frac{1}{4\pi}\sum_{(j,k,l,m)\in\Delta_1}\gamma_j\gamma_k\gamma_l
\gamma_mq_j\bar{q}_kq_l\bar{q}_m,\end{equation}
\begin{equation}\label{3.9}Q_2=\frac{1}{4\pi}\sum_{(j,k,l,m)\in\Delta\setminus\Delta_1}
\gamma_j\gamma_k\gamma_l
\gamma_mq_j\bar{q}_kq_l\bar{q}_m.\end{equation}
Then the Hamiltonian can be written as
\begin{equation}\label{3.10}H=\Lambda+B+Q_1+Q_2+K.\end{equation}
In this section, the symplectic structure is
\begin{equation}\label{3.12}-\mi\sum_{j\neq0}\sigma_jdq_j\wedge d\bar{q}_j,\end{equation}
and the corresponding Poisson bracket for two Hamiltonians $H$, $F$ is
\begin{equation}\label{3.13}\{H,F\}=-\mi\sum_{j\neq0}
\sigma_j(\frac{\partial H}{\partial q_j}\frac{\partial F}{\partial\bar{q}_j}-\frac{\partial H}{\partial \bar{q}_j}\frac{\partial F}{\partial q_j}).\end{equation}

For $J=\emptyset$, i.e., in absence of $x,y$-variables, denote the momentum majorant norm  $||\cdot||_{s,r,p-1,\mathbf{a}}$ in \eqref{2.1.16} as
$||\cdot||_{r,p-1,\mathbf{a}}$. From the analyticity of $g(x,z)$ in $x$ and $z$, we can choose $a>0,\mathbf{a}>0,\tilde{r}>0$ such that the vector field $X_K$ is analytic from some neighbourhood of the origin of $\bar{\ell}^{a,p}$ into $\bar{\ell}^{a,p-1}$ with
\begin{equation}\label{3.6}||X_K||_{\tilde{r},p-1,\mathbf{a}}
=O(\tilde{r}^3).\end{equation}
On the other hand, it's easy to see that the functions $B, Q_1, Q_2$ are analytic in $\bar{\ell}^{a,p}$ with real value, and the vector fields
$X_B,$ $X_{Q_1}, X_{Q_2}$ are analytic maps from $\bar{\ell}^{a,p}$ into $\bar{\ell}^{a,p-1}$ with
\begin{equation}\label{3.11}||X_B||_{\tilde{r},p-1,\mathbf{a}},
||X_{Q_1}||_{\tilde{r},p-1,\mathbf{a}},
||X_{Q_2}||_{\tilde{r},p-1,\mathbf{a}}=O(\tilde{r}^2).\end{equation}

In the following lemma, we will search for a symplectic coordinate transformation which is the time $1$-map of the flow of the Hamiltonian vector field $X_F$, then eliminate $Q_1$ in the Hamiltonian and thus get a partial Birkhoff normal form up to order four.
By Lemma \ref{lem17.12.22.1} and Lemma \ref{lem17.12.23.1} with the absence of $x,y$-variables, we get
\begin{equation}||X_{\{B,F\}}||_{\frac{\tilde{r}}{2},p-1,\mathbf{a}}
\leq2^{2n+4}||X_B||_{\tilde{r},p-1,\mathbf{a}}||X_F||_{\tilde{r},p-1,\mathbf{a}},
\end{equation}

\begin{equation}
||X_{B\circ\Phi_F^1}||_{\frac{\tilde{r}}{2},p-1,\mathbf{a}}\leq
\frac{||X_B||_{\tilde{r},p-1,\mathbf{a}}}{1-2^{2n+6}e
||X_F||_{\tilde{r},p,\mathbf{a}}},
\end{equation}
and similar estimates for $Q_1,Q_2,K$. Therefore we obtain the following lemma, which is Lemma 3.2 in \cite{L-Y} with the momentum majorant norm estimates
 instead of the usual sup-norm estimates.
\begin{lem}\label{NFlem} There exists a real analytic symplectic coordinate transformation $\Psi$ defined in a neighborhood
of the origin of $\bar{\ell}^{a,p}$ which transforms the above Hamiltonian $H$ into a partial Birkhoff
normal form up to order four. More precisely,
\begin{equation}\label{3.14}H\circ\Psi=\Lambda+B+Q_2+R\end{equation}
with
\begin{equation}\label{3.15}||X_R||_{\frac{\tilde{r}}{2},p-1,\mathbf{a}}
=O(\tilde{r}^3).\end{equation}
\end{lem}
\subsection{Using the KAM Theorem}

For the given $J=\{j_1<j_2<\cdots<j_n\}$ in Theorem \ref{1.1}, define $N:=\max(|j_1|,\cdots,|j_n|)$. Then by the transformation $\Psi$ in Lemma 3.1, we get a new Hamiltonian, still denoted by $H$,
\begin{equation}\label{3.28}H=\Lambda+B+Q_2+R,\end{equation}
which is analytic in some neighbourhood $U$ of the origin of $\bar{\ell}^{a,p}$ with $\Lambda$ in \eqref{3.4}, $B$ in \eqref{3.7}, $Q_2$ in \eqref{3.9}, $R$ satisfying \eqref{3.15}.

Introduce new symplectic coordinates $(x,y,z,\bar{z})$ by setting
\begin{equation}\label{3.29}
\begin{cases}q_{j_b}=\sqrt{\zeta_b+y_b}e^{\mi x_b},
\quad\bar{q}_{j_b}=\sqrt{\zeta_b+y_b}e^{-\mi x_b},\quad b=1,\cdots,n,\\
q_j=z_j,\quad \bar{q}_j=\bar{z}_j,\quad j\in\mathbb{Z}_*:=\bar{\mathbb{Z}}\setminus J,\end{cases}\end{equation}
where $\zeta=(\zeta_1,\cdots,\zeta_n)\in\mathbb{R}_+^n$. Then
\begin{equation}\label{3.30}\Lambda=\sum_{1\leq b\leq n}\sigma_{j_b}j_b^2(\zeta_b+y_b)+\sum_{j\in\mathbb{Z}_*}\sigma_jj^2|z_j|^2,
\end{equation}
\begin{eqnarray}\label{3.31}\nonumber B&=&\frac{1}{4\pi}\sum_{1\leq b\leq n}j_b^2(\zeta_b+y_b)^2+\frac{1}{4\pi}
\sum_{j\in\mathbb{Z}_*}j^2|z_j|^4\\
&&+\frac{1}{2\pi}\sum_{1\leq b,b'\leq n\atop{b\neq b'}}|j_bj_{b'}|(\zeta_b+y_b)(\zeta_{b'}+y_{b'})\nonumber\\
&&+\frac{1}{\pi}\sum_{1\leq b\leq n\atop{j\in\mathbb{Z}_*}}|j_bj|(\zeta_b+y_b)|z_j|^2+\frac{1}{2\pi}\sum_{j,l\in
\mathbb{Z}_*\atop{j\neq l}}|jl||z_j|^2|z_l|^2.
\end{eqnarray}
Thus the new Hamiltonian, still denoted by $H$, up to a constant depending only on $\xi$ , is given by
\begin{equation}\label{3.32}H=N+P=\sum_{1\leq b\leq n}\sigma_{j_b}\omega_by_b+\sum_{j\in\mathbb{Z}_*}\sigma_j\Omega_jz_j\bar{z}_j+\tilde{Q}+Q_2+R\end{equation}
with the symplectic structure
\begin{equation}\label{3.33}\sum_{1\leq b\leq n}\sigma_{j_b}dx_b\wedge dy_b-\mi\sum_{j\in\mathbb{Z}_*}\sigma_jdz_j\wedge d\bar{z}_j,\end{equation}
where
\begin{eqnarray}\label{3.34}\nonumber\omega_b&=&j_b^2+
\frac{\sigma_{j_b}}{2\pi}j_b^2\zeta_b+\frac{\sigma_{j_b}}{\pi}\sum_{1\leq b'\leq n\atop{b'\neq b}}|j_bj_{b'}|\zeta_{b'}\\&&=j_b^2+
\frac{j_b}{\pi}(\frac12|j_b|\zeta_b+\sum_{1\leq b'\leq n\atop{b'\neq b}}|j_{b'}|\zeta_{b'}),\end{eqnarray}
\begin{eqnarray}\label{3.35}\Omega_j&=&j^2+\frac{\sigma_j}{\pi}\sum_{1\leq b\leq n}|j_bj|\zeta_b\nonumber\\
\label{3.35}&=&j^2+\frac{j}{\pi}\sum_{1\leq b\leq n}|j_b|\zeta_b,\end{eqnarray}
\begin{eqnarray}\tilde{Q}&=&\frac{1}{4\pi}\sum_{1\leq b\leq n}j_b^2y_b^2+\frac{1}{4\pi}
\sum_{j\in\mathbb{Z}_*}j^2|z_j|^4\nonumber\\
&&+\frac{1}{2\pi}\sum_{1\leq b,b'\leq n\atop{b\neq b'}}|j_bj_{b'}|y_by_{b'}\nonumber\\
\label{17.10.19.1}&&+\frac{1}{\pi}\sum_{1\leq b\leq n\atop{j\in\mathbb{Z}_*}}|j_bj|y_b|z_j|^2+\frac{1}{2\pi}
\sum_{j,l\in\mathbb{Z}_*\atop{j\neq l}}|jl||z_j|^2|z_l|^2.\end{eqnarray}

For simplicity, introduce \begin{equation}\label{18.3.25.1}\xi_b=\frac{1}{\pi}
(\frac12|j_b|\zeta_b+\sum_{b'\neq b}|j_{b'}|\zeta_{b'}),\quad 1\leq b\leq n.\end{equation}
By direct calculation,
we have
\begin{equation}\label{17.10.19.2}\omega_{b}=j_b^2+j_b\xi_b,\quad
\Omega_{j}=j^2+j\frac{\sum_{b=1}^n\xi_b}{n-\frac12},\end{equation}
\begin{equation} \frac{\partial\xi}{\partial\zeta}=
\frac{1}{\pi}
 \begin{pmatrix}\frac12&1&\cdots&1\\
  1&\frac12&\cdots&1\\
  \cdots&\cdots&\cdots&\cdots\\
 1&1&\cdots&\frac12\end{pmatrix}\mbox{diag}(j_b:1\leq b\leq n)\end{equation}
and
\begin{equation}(\frac{\partial\xi}{\partial\zeta})^{-1}=
\frac{4\pi}{2n-1}\mbox{diag}(j_b^{-1}:1\leq b\leq n)
 \begin{pmatrix}\frac32-n&1&\cdots&1\\
  1&\frac32-n&\cdots&1\\
  \cdots&\cdots&\cdots&\cdots\\
 1&1&\cdots&\frac32-n\end{pmatrix}.
 \end{equation}
 Therefore, it is equivalent to treat $\xi$ as parameters. In view of \eqref{17.10.19.2}, for $k\in\mathbb{Z}^n, l\in\mathbb{Z}^{\mathbb{Z}_*}$, we have
 \begin{equation}\langle k,\omega(\xi)\rangle+\langle l,\Omega(\xi)\rangle=(\sum_{b=1}^n
k_bj_b^2+\sum_{j\in\mathbb{Z}_*}l_jj^2)
+\sum_{b=1}^n(k_bj_b+\sum_{j\in\mathbb{Z}_*}
\frac{l_jj}{n-\frac12})\xi_b.\end{equation}
The following lemma is used to check assumption (C) in the KAM theorem.
 \begin{lem}\label{lem18.3.27.1}For $k\in\mathbb{Z}^n$, $|l|\leq2$,
 at least one of the following $n+1$ inequalities holds:
\begin{equation}\label{17.10.23.1}|\sum_{b=1}^n
k_bj_b^2+\sum_{j\in\mathbb{Z}_*}l_jj^2|
\geq\frac{1}{100n}\max\{|k|,
\sum_{j\in\mathbb{Z}_*}|jl_j|\},\end{equation}
\begin{equation}\label{17.10.23.2}|k_bj_b+
\sum_{j\in\mathbb{Z}_*}
\frac{l_jj}{n-\frac12}|\geq \frac{1}{100n\sum_{b=1}^n|j_b|}
\max\{|k|,\sum_{j\in\mathbb{Z}_*}|jl_j|
\},\quad b=1,\cdots,n.
\end{equation}
 \end{lem}
 \begin{proof}
 We prove this lemma in the following cases:

Case 1: $|k|\geq2\sum_{j\in\mathbb{Z}_*}|jl_j|$.

We prove that at least one of the $n$ inequalities in \eqref{17.10.23.2} is true. Otherwise, for any $1\leq b\leq n$,
 \begin{equation}|k_bj_b+
\sum_{j\in\mathbb{Z}_*}
\frac{l_jj}{n-\frac12}|< \frac{|k|}{100n\sum_{b=1}^n|j_b|},
\end{equation}
 and thus
 \begin{equation}|k_b|\leq|k_bj_b|<
\sum_{j\in\mathbb{Z}_*}
\frac{|l_jj|}{n-\frac12}+ \frac{|k|}{100n\sum_{b=1}^n|j_b|}\leq
\frac{|k|}{2n-1}+\frac{|k|}{100n\sum_{b=1}^n|j_b|}.\end{equation}
Taking the sum of the above inequalities with respect to $1\leq b\leq n$, we get
\begin{equation}|k|<\big(\frac{n}{2n-1}
+\frac{1}{100\sum_{b=1}^n|j_b|}\big)|k|,\end{equation}
 which is impossible by noticing that $n\geq2$.

Case 2: $|k|<2\sum_{j\in\mathbb{Z}_*}|jl_j|$.

Supposing the lemma not true, then
\begin{equation}\label{17.10.23.1}|\sum_{b=1}^n
k_bj_b^2+\sum_{j\in\mathbb{Z}_*}l_jj^2|
<\frac{1}{50n}
\sum_{j\in\mathbb{Z}_*}|jl_j|,\end{equation}
\begin{equation}\label{18.3.27.2}|k_bj_b+
\sum_{j\in\mathbb{Z}_*}
\frac{l_jj}{n-\frac12}|< \frac{1}{50n\sum_{b=1}^n|j_b|}
\sum_{j\in\mathbb{Z}_*}|jl_j|,\quad b=1,\cdots,n.
\end{equation}
In the following, we will derive contradiction in all possible cases. Denote $e_j\in\mathbb{Z}^{\mathbb{Z}_*}$ as the sequence with all zeros except the $j$-th component, which is $1$.

Subcase 2.1: $l=e_j$ (the same for $-e_j$). From \eqref{17.10.23.1} and \eqref{18.3.27.2}, we get
\begin{equation}\label{17.11.11.1}|\sum_{b=1}^n
k_bj_b^2+ j^2|<\frac{|j|}{50n},\end{equation}
\begin{equation}\label{17.11.11.2}|k_bj_b+\frac{j}{n-\frac12}|<
\frac{|j|}{50n\sum_{b=1}^n|j_b|},\quad b=1,\cdots,n.
\end{equation}
 By \eqref{17.11.11.2} and $n\geq2$, we know
 \begin{equation}\label{17.11.11.7}|k_bj_b|<
\frac{|j|}{n-\frac12}+\frac{|j|}{50n\sum_{b=1}^n|j_b|}
<\frac{|j|}{n-\frac59},\quad b=1,\cdots,n,\end{equation}
and thus
\begin{eqnarray}
\nonumber|\sum_{b=1}^nk_bj_b^2+j^2|&\geq& j^2-\sum_{b=1}^n|k_bj_b^2|\\
\nonumber&=& j^2-\sum_{k_b\neq0}|k_bj_b^2|\\
\nonumber&\geq& j^2-\sum_{k_b\neq0}k_b^2j_b^2\\
\nonumber&> &j^2-\frac{n}{(n-\frac59)^2}j^2\\
&>&\frac{1}{50n}j^2,\end{eqnarray}
which contradicts with \eqref{17.11.11.1}.

Subcase 2.2: $l=e_i+e_j$ (the same for $-e_i-e_j$). From \eqref{17.10.23.1} and \eqref{18.3.27.2}, we get

\begin{equation}\label{17.10.25.3}|\sum_{b=1}^n
k_bj_b^2+i^2+j^2|<\frac{1}{50n}(|i|+|j|),
\end{equation}
\begin{equation}\label{17.10.25.6}|k_bj_b+\frac{i+j}{n-\frac12}|<
\frac{|i|+|j|}{50n\sum_{b=1}^n|j_b|},\quad b=1,\cdots,n.
\end{equation}
By \eqref{17.10.25.6} and $n\geq2$, we know
 \begin{equation}\label{17.10.23.5}|k_bj_b|< \frac{|i+j|}{n-\frac12}+\frac{|i|+|j|}
 {50n\sum_{b=1}^n|j_b|}
 <\frac{|i|+|j|}{n-\frac{9}{17}},\quad b=1,\cdots,n.\end{equation}
Therefore, for $n\geq3$,
\begin{eqnarray}
\nonumber|\sum_{b=1}^nk_bj_b^2|&=&|\sum_{k_b\neq0}k_bj_b^2|\\
\nonumber&\leq&\sum_{k_b\neq0}k_b^2j_b^2\\
\nonumber&<&\frac{n}{(n-\frac{9}{17})^2}(|i|+|j|)^2\\
\nonumber&\leq&\frac{5}{294}(i^2+j^2),
\end{eqnarray}
and thus
\begin{eqnarray}\nonumber|\sum_{b=1}^nk_bj_b^2
+i^2+j^2|&\geq& i^2+j^2-\sum_{b=1}^n|k_bj_b^2|\\
\nonumber&>&i^2+j^2-
\frac{5}{294}(i^2+j^2)\\
&>&\frac{1}{50n}(i^2+j^2),\end{eqnarray}
which contradicts with \eqref{17.10.25.3}. Finally we
prove the case $n=2$ with $j_1<0<j_2$. If $ij<0$, by the first inequality of \eqref{17.10.23.5},
\begin{eqnarray}
\nonumber|\sum_{b=1}^2k_bj_b^2|&\leq&\sum_{k_b\neq0}k_b^2j_b^2\\
\nonumber&<&2\Big(\frac{2|i+j|}{3}+\frac{|i-j|}{
100(|j_1|+|j_2|)}\Big)^2\\
\nonumber&=&\frac{8}{9}(i+j)^2+
\frac{2|i^2-j^2|}{75(|j_1|+|j_2|)}
+\frac{(i-j)^2}{5000(|j_1|+|j_2|)^2},
\end{eqnarray}
and thus
\begin{eqnarray}\nonumber|\sum_{b=1}^2k_bj_b^2
+i^2+j^2|&\geq& i^2+j^2-\sum_{b=1}^2|k_bj_b^2|\\
\nonumber&>&(i^2+j^2)\Big(1-\frac{8}{9}-\frac{2}{75(|j_1|+|j_2|)}
-\frac{1}{2500(|j_1|+|j_2|)^2}\Big)\\
&>&\frac{1}{100}(i^2+j^2),\end{eqnarray}
which contradicts with \eqref{17.10.25.3}.
 Otherwise, $ij>0$. From \eqref{17.10.25.6}, $k_bj_b$ has the same sign as $i+j$. Then in view of
 $j_1<0<j_2$, we conclude that $k_1$ and $k_2$ have different signs, and thus by the second inequality of
\eqref{17.10.23.5},
\begin{equation}\label{18.3.27.1}
|\sum_{b=1}^2k_bj_b^2|\leq
\max\{|k_1j_1^2|,|k_2j_2^2|\}
\leq\frac{1}{(2-\frac{9}{17})^2}(i+j)^2\leq
\frac{2}{(2-\frac{9}{17})^2}(i^2+j^2).
\end{equation}
Using \eqref{18.3.27.1}, we obtain
\begin{eqnarray*}
|\sum_{b=1}^2k_bj_b^2+i^2+j^2|&\geq&(i^2+j^2)-
|\sum_{b=1}^2k_bj_b^2|\\
&>&(i^2+j^2)-\frac{2}{(2-\frac{9}{17})^2}(i^2+j^2)\\
&>&\frac{1}{100}(i^2+j^2),
\end{eqnarray*}
which contradicts with \eqref{17.10.25.3}.

Subcase 2.3: $l=e_i-e_j$, $ij>0, i\neq j$. From \eqref{17.10.23.1} and \eqref{18.3.27.2}, we get
\begin{equation}\label{17.11.12.5}|\sum_{b=1}^n
k_bj_b^2+i^2-j^2|<\frac{1}{50n}|i+j|,\end{equation}
\begin{equation}\label{17.11.12.6}|k_bj_b+\frac{i-j}{n-\frac12}|<
\frac{1}{50n\sum_{b=1}^n|j_b|}|i+j|,\quad b=1,\cdots,n.
\end{equation}
By \eqref{17.11.12.6}, we get
\begin{equation}\label{17.11.12.7}|k_bj_b|<
\frac{|i-j|}{n-\frac12}+
\frac{1}{50n\sum_{b=1}^n|j_b|}|i+j|,\quad b=1,\cdots,n,\end{equation}
and thus
\begin{eqnarray}\nonumber|\sum_{b=1}^nk_bj_b^2|&=&|\sum_{k_b\neq0}k_bj_b^2|\\
\nonumber&\leq&\sum_{k_b\neq0}|k_bj_b||j_b|\\
\nonumber&<&\sum_{k_b\neq0}
 \Big(\frac{|i-j|}{n-\frac12}+
 \frac{1}{50n\sum_{b=1}^n|j_b|}|i+j|\Big)|j_b|\\
 \nonumber&=&\frac{|i-j|}{n-\frac12}\sum_{k_b\neq0}
|j_b|+\frac{\sum_{k_b\neq0}|j_b|}{50n\sum_{b=1}^n|j_b|}|i+j|\\
\nonumber&\leq&\frac{|i-j|}{n-\frac12}\sum_{k_b\neq0}
|k_bj_b|+\frac{1}{50n}|i+j|\\
\nonumber&<&\frac{|i-j|}{n-\frac12}\sum_{k_b\neq0}\Big(
 \frac{|i-j|}{n-\frac12}+\frac{1}
 {50n\sum_{b=1}^n|j_b|}|i+j|\Big)
 +\frac{1}{50n}|i+j|\\
\label{17.11.12.8}&\leq&\frac{n}{(n-\frac12)^2}|i-j|^2+
 \frac{|i^2-j^2|}{50(n-\frac12)\sum_{b=1}^n|j_b|}+
 \frac{1}{50n}|i+j|.
\end{eqnarray}
Furthermore, for $n\geq2$, we obtain
 \begin{eqnarray}\nonumber|\sum_{b=1}^nk_bj_b^2+i^2-j^2|&\geq&
 |i^2-j^2|-|\sum_{b=1}^nk_bj_b^2|\\
 \nonumber&>&|i^2-j^2|\Big(1-\frac{n}{(n-\frac12)^2}-
 \frac{1}{50(n-\frac12)\sum_{b=1}^n|j_b|}-\frac{1}{50n}\Big)\\
 \nonumber&\geq&|i^2-j^2|\big(1-\frac89-\frac{1}{150}
 -\frac{1}{100}\big)\\
 \nonumber &>&\frac{1}{50n}|i^2-j^2|,
\end{eqnarray}
 which contradicts with \eqref{17.11.12.5}.

Subcase 2.4: $l=e_i-e_j$, $ij<0$. From \eqref{17.10.23.1} and \eqref{18.3.27.2}, we
get
\begin{equation}\label{17.11.12.3}|\sum_{b=1}^n
k_bj_b^2+i^2-j^2|<\frac{1}{50n}|i-j|,\end{equation}
\begin{equation}\label{17.11.12.1}|k_bj_b+\frac{i-j}
{n-\frac12}|<\frac{1}{50n\sum_{b=1}^n|j_b|}|i-j|,\quad b=1,\cdots,n.
\end{equation}
 In view of
\eqref{17.11.12.1}, we get
\begin{eqnarray}
\nonumber|\sum_{b=1}^nk_bj_b^2+\frac{i-j}{n-\frac12}\sum_{b=1}^nj_b|
&\leq&\sum_{b=1}^n|k_bj_b+\frac{i-j}{n-\frac12}||j_b|\\
\nonumber&<&\frac{1}{50n\sum_{b=1}^n|j_b|}|i-j|
\cdot\sum_{b=1}^n|j_b|\\
\label{17.11.12.2}&=&\frac{1}{50n}|i-j|.
\end{eqnarray}
From \eqref{17.11.12.3} and \eqref{17.11.12.2}, we get
\begin{equation}\label{17.11.12.4}|(i^2-j^2)-
\frac{i-j}{n-\frac12}
\sum_{b=1}^nj_b|<\frac{1}{25n}|i-j|,\end{equation}
and thus
\begin{equation}\label{18.3.26.8}|(i+j)-
\frac{2}{2n-1}
\sum_{b=1}^nj_b|<\frac{1}{25n}.\end{equation}
On the other hand, recalling $2n-1\nmid\sum_{b=1}^nj_b$, we have
\begin{equation}
|(i+j)-\frac{2}{2n-1}\sum_{b=1}^nj_b|>\frac{1}{2n-1},
\end{equation}
which contradicts with \eqref{18.3.26.8}.
\end{proof}

 Denote the normal form $\sum_{1\leq b\leq n}\sigma_{j_b}\omega_by_b+\sum_{j\in\mathbb{Z}_*}
\sigma_j\Omega_jz_j\bar{z}_j$ by $N$ and the perturbation $\tilde{Q}+Q_2+R$ by $P$. Consider the Hamiltonian $H=N+P$ on $D(s,r)\times\Xi_r$, where the phase domain $D(s,r)$ is defined in \eqref{17.9.19.1} with $s,r>0$ suitably small and the parameter domain
\begin{equation}\label{18.3.31.1}\Xi_r
=\{\xi\in\mathbb{R}^n_+
:|\xi|\leq r^{3/2}\}.\end{equation}
In view of \eqref{17.10.19.1}, we have
\begin{equation}\label{18.3.25.2}||X_{\tilde{Q}}||
_{s,r,p-1,\mathbf{a};\Xi_r}=O(r^2).\end{equation}
In view of \eqref{3.9}, $Q_2$ is at least $3$-order about $z,\bar{z}$, and then following the proof of Proposition 7.2 in \cite{Berti2}, we get
\begin{equation}\label{18.3.25.3}||X_{Q_2}||_{s,r,p-1,\mathbf{a};
\Xi_r}=O(r^{\frac74}).\end{equation}
In view of \eqref{3.15}, following the proof of Proposition 7.2 in \cite{Berti2}, we get
\begin{equation}\label{18.3.25.4}||X_{R}||_{s,r,p-1,\mathbf{a};
\Xi_r}=O(r^{\frac74}).\end{equation}
We conclude from \eqref{18.3.25.2}-\eqref{18.3.25.4},
\begin{equation}\label{18.3.25.5}||X_{P}||_{s,r,p-1,\mathbf{a};
\Xi_r}=O(r^{\frac74}).\end{equation}

Define
\begin{equation}\mathcal{O}_r:=U_{-\alpha}
\Xi_r,\end{equation}
where $U_{-\rho}\Xi$ is the subset of all points in $\Xi$ with boundary distance greater than $\rho$.
Now study the Hamiltonian $H=N+P$ on $D(s,r)\times\mathcal{O}_r$ by the KAM theorem. In view of
 \eqref{17.10.19.2}, the assumption (A) is fulfilled with $m=\frac12$; the assumption (B) is fulfilled with $\delta=1$ and
  \begin{equation}M_1=\max_{1\leq b\leq n}|j_b|,\quad M_2=\frac{n}{n-\frac12};\end{equation}
 by Lemma \ref{lem18.3.27.1}, the assumption (C) is fulfilled with
 \begin{equation}M_3=\frac{1}{100n\sum_{b=1}^n
 |j_b|}.\end{equation}
 Choose
 \begin{equation}\alpha=r^{\frac85}\gamma^{-1},\quad
 \beta=\frac{1}{13},\end{equation}
 where $\gamma$ is taken from the KAM theorem.
 Set $M=M_1+M_2$, which only depends on the set $J$. Observe that when $r$ is small enough,
 \begin{equation}\label{18.4.1.2}\epsilon:=
 ||X_P||_{s,r,p-1,\mathbf{a};\mathcal{O}_r}+
 \frac{\alpha}{M}
 ||X_P||^{lip}_{s,r,p-1,\mathbf{a};
 \mathcal{O}_r}=O(r^{\frac74})\leq
 (\alpha\gamma)^{1+\beta},\end{equation}
  which is the smallness condition \eqref{17.7.21.20}.
Applying Theorem \ref{thm7.12.1}, we obtain a Cantor set $\mathcal{O}_r^{-}\subset\mathcal{O}_r$ with
\begin{equation}|\mathcal{O}_r\setminus
\mathcal{O}_r^{-}|=O((r^{\frac32})^{n-1}\alpha
)=O(r^{\frac32n+\frac{1}{10}}),\end{equation}
a Lipschitz family of smooth torus embeddins $\Phi_r:\mathbb{T}^n\times
\mathcal{O}_r^-\rightarrow\mathcal{P}^{a,p}$, and a Lipschitz frequency map: $\mathcal{O}_r^-\rightarrow\mathbb{R}^n$, such that for each $\xi\in\mathcal{O}_r^-$ the map $\Phi_r$ restricted to $\mathbb{T}^n\times\{\xi\}$ is a smooth embedding of a rotational torus with frequencies $\phi_r(\xi)$ for the perturbed Hamiltonian $H$ at $\xi$. Moreover, for every non-negative integer multi-index $k=(k_1,\cdots,k_n)$,
\begin{equation}\label{18.3.26.2}||\partial_x^k(\Phi_r-\Phi_0)||
_{s,r,p;\mathbb{T}^n\times\mathcal{O}_r^-}+
\frac{\alpha}{M}||\partial_x^k(\Phi_r-\Phi_0)||^{lip}
_{s,r,p;\mathbb{T}^n\times\mathcal{O}_r^-}=
O(\epsilon^{\frac{1}{1+\beta}}/\alpha)=O(r^{1/40}),
\end{equation}
\begin{equation}\label{18.3.26.3}|\phi_r-\omega|_{\mathcal{O}_r^-}
+\frac{\alpha}{M}|\phi_r-\omega|^{lip}
_{\mathcal{O}_r^-}=O(\epsilon)=O(r^{7/4}).\end{equation}

The Cantor set $\mathcal{O}_r^-$ by itself is not dense at the origin. To obtain such a set, following \cite{Poschel3}, we take the union of a stable sequence of subsets of $\mathcal{O}_r^-$.
Set $r_j=\frac{r_0}{2^j},j\geq0$ and define
\begin{equation}\mathcal{C}:=\bigcup_{j\geq0}
\mathcal{C}_{r_j},\end{equation}
where  \begin{equation}\mathcal{C}_r=\mathcal{O}_r^-\cap
\Big(U_{-\alpha}(\Xi_r\setminus\Xi_{\frac r2})\Big).\end{equation}
The same as the proof in \cite{L-Y}, the Cantor set $\mathcal{C}$ has full density at the origin.
 Then,
define the embedding $\Phi:\mathbb{T}^n\times\mathcal{C}
\rightarrow\mathcal{P}^{a,p}$ and the frequency map $\phi:\mathcal{C}\rightarrow\mathbb{R}^n$ by piecing together the corresponding definitions on each components. Furthermore, define $\Psi:=\Phi+T_{\xi}$, where $T_{\xi}(x,\xi)=(0,\xi,0,0)$.
The estimates of $\Psi-\Psi_0$ and $\phi-\omega$ on $\mathcal{C}\cap\Xi_{r_j}$ follow from \eqref{18.3.26.2} and \eqref{18.3.26.3}.
This finally completes the proof of Theorem \ref{1.1}.

\section{The Homological Equations}
In this section, the Poisson bracket $\{H,F\}$ for two Hamiltonians $H,F$ is defined with respect to the symplectic structure \eqref{7.13.4}, i.e.,
$$\{H,F\}=\sum_{1\leq b\leq n}\sigma_{j_b}(\frac{\partial H}{\partial x_b}\frac{\partial F}{\partial y_b}-
\frac{\partial H}{\partial y_b}\frac{\partial F}{\partial x_b})-\mathbf{i}\sum_{j\in\mathbb{Z}_*}\sigma_j(\frac{\partial H}{\partial z_j}\frac{\partial F}{\partial \bar{z}_j}-
\frac{\partial H}{\partial \bar{z}_j}\frac{\partial F}{\partial z_j}).$$

\subsection{Derivation of homological equations}

The proof of
 Theorem \ref{thm7.12.1} employs the rapidly converging iteration scheme of Newton type to deal with small-divisor problems introduced by Kolmogorov, involving infinite sequence of coordinate transformations. At the $\nu$-th step of the scheme, the Hamiltonian
$$H_{\nu}=N_{\nu}+P_{\nu}$$
is considered, where $N_{\nu}$ is a generalized normal form
$$N_{\nu}=\sum_{1\leq b\leq n}\sigma_{j_b}\omega_{\nu,b}(\xi)y_b+\sum_{j\in\mathbb{Z}_*}\sigma_j\Omega_{\nu,j}(\xi)z_j\bar{z}_j,$$
$P_{\nu}$ is a small perturbation. A transformation $\Phi_{\nu}$ is set up so that
$$H_{\nu+1}=H_{\nu}\circ\Phi_{\nu}=N_{\nu+1}+P_{\nu+1},$$
where $N_{\nu+1}$ is another generalized normal form, $P_{\nu+1}$ is a much smaller perturbation. We drop the index $\nu$ of $H_{\nu}$, $N_{\nu}$, $P_{\nu}$, $\omega_{\nu}$, $\Omega_{\nu}$, $\Phi_{\nu}$ and shorten the index $\nu+1$ as $+$.

 For a function $u$ on $\mathbb{T}^n$, let
 $$[u]=\frac{1}{(2\pi)^n}\int_{\mathbb{T}^n}u(x)dx.$$
Let $R$ be $2$-order Taylor polynomial truncation of $P$, that is,
\begin{equation}\label{7.12.4} R=R^x+\langle R^y,y\rangle+\langle R^z,z\rangle+\langle R^{\bar{z}},\bar{z}\rangle+
\langle R^{zz}z,z\rangle+\langle R^{z\bar{z}}z,\bar{z}\rangle+\langle R^{\bar{z}\bar{z}}\bar{z},\bar{z}\rangle,\end{equation}
where $\langle\cdot,\cdot\rangle$ denotes the formal products for two column vectors and $R^x,R^y,R^z,R^{\bar{z}},R^{zz},R^{z\bar{z}},R^{\bar{z}\bar{z}}$ depend on $x$ and $\xi$. Denote $[[R]]$ as the part of $R$ in generalized normal form as follows:
$$[[R]]=[R^x]+\langle[R^y],y\rangle+\langle \mbox{diag}(R^{z\bar{z}})z,\bar{z}\rangle,$$
where $\mbox{diag}(R^{z\bar{z}})$ is the diagonal of $R^{z\bar{z}}$. In the following, the term $[R^x]$ will be omitted since it does not affect the dynamics.

The coordinate transformation $\Phi$ is obtained as the time $1$-map $X_F^t|_{t=1}$ of a Hamiltonian vector field $X_F$, where $F$ is of the same form as $R$:
\begin{equation}\label{7.12.5} F=F^x+\langle F^y,y\rangle+\langle F^z,z\rangle+\langle F^{\bar{z}},\bar{z}\rangle+
\langle F^{zz}z,z\rangle+\langle F^{z\bar{z}}z,\bar{z}\rangle+\langle F^{\bar{z}\bar{z}}\bar{z},\bar{z}\rangle,\end{equation}
and $[[F]]=0$. Denote $\partial_{\omega}=\sum_{1\leq b\leq n}\omega_b\frac{\partial}{\partial x_b},$
$\Lambda=\mbox{diag}(\Omega_j:j\in\mathbb{Z}_*)$. Then we have
\begin{align}\nonumber H_+=&H\circ\Phi\\
\nonumber=&(N+R)\circ X_F^1+(P-R)\circ X_F^1\\
 \nonumber=&N+\{N,F\}+R+\int_0^1\{(1-t)\{N,F\}+R,F\}\circ X_F^tdt+(P-R)\circ
X_F^1\\
\label{17.7.21.3}=&N+\sum_{j\in\mathbb{Z}_*}\sigma_j\langle \partial_x\Omega_j,F^y\rangle z_j\bar{z}_j+\langle[R^y],y\rangle+
\langle\mbox{diag}(R^{z\bar{z}})z,\bar{z}\rangle\\
\label{17.7.21.4}&+(-\partial_{\omega}F^x+R^x)\\
\label{17.7.21.5}&+\langle -\partial_{\omega}F^y+R^y-[R^y],y\rangle\\
\label{17.7.21.6}&+\langle-\partial_{\omega}F^z+\mathbf{i}\Lambda F^z+R^z,z\rangle\\
\label{17.7.21.7}&+\langle-\partial_{\omega}F^{\bar{z}}-\mathbf{i}\Lambda F^{\bar{z}}+R^{\bar{z}},\bar{z}\rangle\\
\label{17.7.21.8}&+\langle(-\partial_{\omega}F^{zz}+\mathbf{i}\Lambda F^{zz}+\mathbf{i}F^{zz}\Lambda+R^{zz})z,z\rangle\\
\label{17.7.21.9}&+\langle(-\partial_{\omega}F^{\bar{z}\bar{z}}-\mathbf{i}\Lambda F^{\bar{z}\bar{z}}-\mathbf{i}F^{\bar{z}\bar{z}}\Lambda+R^{\bar{z}\bar{z}})\bar{z},\bar{z}\rangle\\
\label{17.7.21.10}&+\langle(-\partial_{\omega}F^{z\bar{z}}-\mathbf{i}\Lambda F^{z\bar{z}}+\mathbf{i}F^{z\bar{z}}\Lambda+R^{z\bar{z}}-
\mbox{diag}(R^{z\bar{z}}))z,\bar{z}\rangle\\
\label{17.7.21.11}&+\int_0^1\{(1-t)\{N,F\}+R,F\}\circ X_F^tdt+(P-R)\circ
X_F^1.
\end{align}
We wish to find the function $F$ such that
\eqref{17.7.21.4}-\eqref{17.7.21.10} vanish. To this end,
$F^x,F^y,F^z,F^{\bar{z}}, F^{zz},F^{z\bar{z}}$ and
$F^{\bar{z}\bar{z}}$ should satisfy the homological equations:
\begin{align}\label{14.2.4.10}\partial_{\omega}F^x&=R^x,\\
\label{14.2.4.11}\partial_{\omega}F^y&=R^y-[R^y],\\
 \label{14.2.4.12}\partial_{\omega}F^z_j-\mi\Omega_jF^z_j&=R^z_j,\quad j\in\mathbb{Z}_*,\\
\label{14.2.4.13}\partial_{\omega}F^{\bar{z}}_j+\mi\Omega_jF^{\bar{z}}_j&=R^{\bar{z}}_j,\quad j\in\mathbb{Z}_*,\\
\label{14.2.4.14}\partial_{\omega}F^{zz}_{ij}-\mi(\Omega_i+\Omega_j)F^{zz}_{ij}&=R^{zz}_{ij},\quad
i,j\in\mathbb{Z}_*,\\
\label{14.2.4.15}\partial_{\omega}F^{\bar{z}\bar{z}}_{ij}+\mi(\Omega_i+\Omega_j)F^{\bar{z}\bar{z}}_{ij}&=R_{ij}^{\bar{z}\bar{z}},\quad
i,j\in\mathbb{Z}_*,\\
\label{14.2.4.16}\partial_{\omega}F^{z\bar{z}}_{ij}+\mi(\Omega_i-\Omega_j)F^{z\bar{z}}_{ij}&=R^{z\bar{z}}_{ij},\quad i,j\in\mathbb{Z}_*,i\neq j.
\end{align}

\subsection{Two lemmas for solving the homological equations}

The homological equations \eqref{14.2.4.10} and \eqref{14.2.4.11} can be directly
solved by comparing Fourier coefficients; the homological equations \eqref{14.2.4.12}-\eqref{14.2.4.16} are solved by using Lemma \ref{lem17.7.31.1} and Lemma \ref{lem17.9.22.1} below.

Recall that $J=\{j_1<\cdots<j_n\}\subset\bar{\mathbb{Z}}$. Denote
$C_J=\max_{1\leq b\leq n}|j_b|$ and
 \begin{equation}\label{17.11.4.1}
 \pi(k,\mathbf{m})=\sum_{b=1}^nk_bj_b+
 \mathbf{m},\ \mathbf{m}\in\mathbb{Z}.\end{equation}
For an analytic function $u(x)$ on $D(s)$, define its momentum majorant norm as
\begin{equation*}|u|_{s,\mathbf{a},\mathbf{m}}:=
\sum_{k\in\mathbb{Z}^n}
|\hat{u}_k|e^{|k|s}e^{\mathbf{a}|\pi(k,\mathbf{m})|},\end{equation*}
 where $\hat{u}_k$ is the $k$-Fourier coefficient of $u:$ $\hat{u}_k=(2\pi)^{-n}\int_{\mathbb{T}^n}u(x)e^{-\mathbf{i}k\cdot x}dx$.
 The following lemma is Theorem 1.4 in \cite{L-Y2} with the  momentum majorant norm estimate instead of the sup-norm estimate, see \eqref{17.11.3.1}.

\begin{lem}\label{lem17.7.31.1}Consider the first-order partial differential equation
\begin{equation}\label{17.7.20.15}-\mathbf{i}\partial_{\omega}u+\lambda u+\mu(x)u=p(x),\quad x\in\mathbb{T}^n\end{equation}
for the unknown function $u$ defined on the torus $\mathbb{T}^n$, where $\omega=(\omega_1,\cdots,\omega_n)\in\mathbb{R}^n$ and $\lambda\in\mathbb{C}$. Assume:
\begin{itemize}\item[(1)] There are constants $\alpha_1,\alpha_2,\tilde{\gamma}>0$ and $\tau>n$ such that
\begin{equation}\label{17.7.20.16}|\langle k,\omega\rangle|\geq\frac{\alpha_1}{|k|^{\tau}},\quad k\in\mathbb{Z}^n\setminus\{0\}, \end{equation}
\begin{equation}\label{17.7.20.17}|\langle k,\omega\rangle+\lambda|\geq\frac{\alpha_2\tilde{\gamma}}
{1+|k|^{\tau}},\quad k\in\mathbb{Z}^n.  \end{equation}
\item[(2)] $\mu:D(s)\rightarrow\mathbb{C}$ is real analytic (here `real' means $\mu(\mathbb{T}^n)\subset\mathbb{R}$) and is of zero average, i.e., $\int_{\mathbb{T}^n}\mu(\phi)d\phi=0$.
Moreover, assume there is a constant $C>0$ such that
\begin{equation}\label{17.7.20.18}|\mu|_{s,\tau+1}:=\sum_{k\in\mathbb{Z}^n}
|\hat{\mu}_k||k|^{\tau+1}e^{|k|s}
\leq C\tilde{\gamma},\end{equation}
where $\hat{\mu}_k$ is the $k$-Fourier coefficient of $\mu$.
\item[(3)] $p(x)$ is analytic in $x$ in $D(s)$ with finite momentum majorant norm.
 \end{itemize}
Then \eqref{17.7.20.15} has a unique solution $u(x)$ which is defined in a narrower domain $D(s-\sigma )$ with $0<\sigma<s$, and which satisfies
\begin{equation}\label{17.11.3.1}|u|_{s-\sigma,\mathbf{a},\mathbf{m}}
\leq\frac{c(n,\tau)}{\alpha_2\tilde{\gamma}
\sigma^{2n+\tau}}
e^{2C\tilde{\gamma s}/\alpha_1}|p|_{s,\mathbf{a},\mathbf{m}}\end{equation}
for $4\mathbf{a}C_J\leq\sigma<\min\{1,s\}$ and the constant $c(n,\tau)=4^{n+\tau}(8e+8)^n(6e+6)^n(1+(\frac{3\tau}{e})^{\tau})$.
\end{lem}
\begin{proof}By Theorem 1.4 in \cite{L-Y2}, we have
\begin{equation}\label{17.7.20.20}\sup_{x\in D(s-\frac{\sigma}{2})}|u(x)|\leq\frac{c_1(n,\tau)}{\alpha_2\tilde{\gamma}
\sigma^{n+\tau}}
e^{2C\tilde{\gamma s}/\alpha_1}\sup_{x\in D(s-\frac{\sigma}{4})}|p(x)|,\end{equation}
where $c_1(n,\tau)=4^{n+\tau}(6e+6)^n(1+(\frac{3\tau}{e})^{\tau})$. In view of $\mathbf{a}C_J\leq\frac{\sigma}{4}$ and \eqref{17.7.20.1}, we have
\begin{eqnarray}\nonumber|u|_{s-\sigma,\mathbf{a},\mathbf{m}}&=&
\sum_{k\in\mathbb{Z}^n}|\hat{u}_k|e^{|k|(s-\sigma)}
e^{\mathbf{a}|\pi(k,\mathbf{m})|}\\
\nonumber&\leq&\sum_{k\in\mathbb{Z}^n}|
\hat{u}_k|e^{|k|(s-\sigma)}e^{\mathbf{a}C_J|k|}e^{\mathbf{a}|\mathbf{m}|}\\
\nonumber&\leq&e^{\mathbf{a}|\mathbf{m}|}\sum_{k\in\mathbb{Z}^n}|
\hat{u}_k|e^{|k|(s-\frac{3}{4}\sigma)}\\
\nonumber&\leq&e^{\mathbf{a}|\mathbf{m}|}(\sup_{x\in D(s-\frac{\sigma}{2})}|u(x)|)
\sum_{k\in\mathbb{Z}^n}e^{-|k|\frac{\sigma}{4}}\\
\label{17.11.3.2}&\leq&\frac{(8e+8)^n}{\sigma^n}e^{\mathbf{a}|\mathbf{m}|}
(\sup_{x\in D(s-\frac{\sigma}{2})}|u(x)|),
\end{eqnarray}
\begin{eqnarray}
\nonumber\sup_{x\in D(s-\frac{\sigma}{4})}|p(x)|&\leq&\sum_{k\in\mathbb{Z}^n}|\hat{p}_k|
e^{|k|(s-\frac{\sigma}{4})}\\
\nonumber&\leq&
\sum_{k\in\mathbb{Z}^n}|\hat{p}_k|
e^{|k|(s-\frac{\sigma}{4})}e^{\mathbf{a}(|
\sum_{b=1}^nk_bj_b+\mathbf{m}|+|\sum_{b=1}^nk_bj_b|-|\mathbf{m}|)}
\\
\nonumber&\leq&\sum_{k\in\mathbb{Z}^n}|\hat{p}_k|
e^{|k|(s-\frac{\sigma}{4})}e^{\mathbf{a}|
\sum_{b=1}^nk_bj_b+\mathbf{m}|}e^{\mathbf{a}C_J|k|}
e^{-\mathbf{a}|\mathbf{m}|}
\\
\label{17.11.4.2}&\leq&e^{-\mathbf{a}|\mathbf{m}|}|p|_{s,\mathbf{a},\mathbf{m}}.
\end{eqnarray}
Combining \eqref{17.7.20.20} \eqref{17.11.3.2} and \eqref{17.11.4.2}, we obtain \eqref{17.11.3.1}.
\end{proof}

For any positive number $K$, we introduce a truncation operator $\Gamma_K$ as follows:
 $$(\Gamma_Kf)(x):=\sum_{|k|\leq K}\hat{f}_ke^{\mathbf{i}k\cdot x},\quad \forall \  f:\mathbb{T}^n\rightarrow\mathbb{C},$$
where $\hat{f}_k$ is the $k$-Fourier coefficient of $f$. The following lemma is Lemma 2.6 in \cite{L-Y2} with the momentum majorant norm estimate instead of the sup-norm estimate, see \eqref{17.9.20.7} and \eqref{17.9.20.8}.

\begin{lem}\label{lem17.9.22.1}Consider the first-order partial differential equation with the truncation operator $\Gamma_K$
\begin{equation}\label{17.7.20.9}-\mathbf{i}\partial_{\omega}u+\lambda u+\Gamma_K(\mu u)=\Gamma_Kp,\quad x\in\mathbb{T}^n\end{equation}
for the unknown function $u$ defined on the torus $\mathbb{T}^n$, where $\omega\in\mathbb{R}^n$, $0\neq\lambda\in\mathbb{C}$, $0<2K|\omega|\leq|\lambda|$, $|\omega|:=\max_{1\leq\nu\leq n}|\omega_{\nu}|$.
Assume that $\mu$ is real analytic in $x\in D(s)$ with
\begin{equation}\label{17.7.20.10}
\sum_{k\in\mathbb{Z}^n}|\hat{\mu}_k|e^{|k|s}
e^{\mathbf{a}|\pi(k,0)|}
\leq\frac{|\lambda|}{4}\end{equation}
for some constant $\mathbf{a}\geq0$, and assume $p(x)$ is analytic in $x\in D(s)$ with finite momentum majorant norm.
 Then \eqref{17.7.20.9} has a unique solution $u=\Gamma_Ku$ and
\begin{equation}\label{17.9.20.7}|u|_{s,\mathbf{a},\mathbf{m}}\leq\frac{4}
{|\lambda|}|p|_{s,\mathbf{a},\mathbf{m}},\end{equation}
\begin{equation}\label{17.9.20.8}|(1-\Gamma_K)(\mu u)|_{s-\sigma,\mathbf{a},\mathbf{m}}\leq e^{-K\sigma}
|p|_{s,\mathbf{a},\mathbf{m}}\end{equation}
for $0<\sigma<s$.
\end{lem}
\begin{proof}The proof of this lemma is parallel to that of Lemma 2.6 in \cite{L-Y2}. However, we give the details for completeness. Suppose that \eqref{17.7.20.9} has a solution with $u=\Gamma_Ku$. Then we can write
$u(\phi)=\sum_{|k|\leq K}\hat{u}_ke^{\mathbf{i}k\cdot\phi}$. Inserting this formula into \eqref{17.7.20.9} and checking the coefficients of the mode $e^{\mathbf{i}k\cdot\phi}$, we can change \eqref{17.7.20.9} into
\begin{equation}\label{17.9.20.2}(\Lambda+\tilde{\mu})\hat{u}=\hat{p},\end{equation}
where $\Lambda=\mbox{diag}(k\cdot\omega+\lambda:|k|\leq K)$, $\tilde{\mu}=(\hat{\mu}_{k-l})_{|k|,|l|\leq K}$, $\hat{u}=(\hat{u}_k)_{|k|\leq K}$ and
$\hat{p}=(\hat{p}_k)_{|k|\leq K}$.

Recall we have assumed $0<K|\omega|\leq\frac{|\lambda|}{2}$ in this lemma. It follows that
\begin{equation}\label{17.9.20.6}|k\cdot\omega+\lambda|\geq\frac{|\lambda|}{2}\quad \mbox{for}\ |k|\leq K.\end{equation}
Set $\Omega=\mbox{diag}(e^{|k|s}e^{\mathbf{a}|\pi(k,\mathbf{m})|}:|k|\leq K)$. Then we have
$$\Omega\tilde{\mu}\Omega^{-1}=(\hat{\mu}_{k-l}e^{(|k|-|l|)s}e^{\mathbf{a}
(|\pi(k,\mathbf{m})|-|\pi(l,\mathbf{m})|)})_{|k|,|l|\leq K}.$$
Furthermore,
\begin{eqnarray}\nonumber||\Omega\tilde{\mu}\Omega^{-1}||_{\ell^1\rightarrow
\ell^1}&=&\max_{|l|\leq K}\sum_{|k|\leq K}|\hat{\mu}_{k-l}|e^{(|k|-|l|)s}e^{\mathbf{a}(|\pi(k,\mathbf{m})
|-|\pi(l,\mathbf{m})|)}\\
\nonumber&\leq&\max_{|l|\leq K}\sum_{|k|\leq K}|\hat{\mu}_{k-l}|e^{(|k-l|)s}e^{\mathbf{a}|\pi(k-l,0)|}\\
\nonumber&\leq&\sum_{k\in\mathbb{Z}^n}|\hat{\mu}_{k}
|e^{|k|s}e^{\mathbf{a}|\pi(k,0)|}\\
\label{17.9.20.5}&\leq&\frac{|\lambda|}{4},
\end{eqnarray}
where we have used \eqref{17.7.20.10} in the last inequality. Consequently, in view of \eqref{17.9.20.2}, we have
\begin{eqnarray}
\nonumber |p|_{s,\mathbf{a},\mathbf{m}}\geq|\Gamma_kp|_{s,\mathbf{a},\mathbf{m}}
=||\Omega\hat{p}||_{\ell^1}&=&||\Omega
(\Lambda+\tilde{\mu})\hat{u}||_{\ell^1}\\
\nonumber&\geq&||\Omega\Lambda\hat{u}||_{\ell^1}-||\Omega\tilde{\mu}\hat{u}||
_{\ell^1}\\\nonumber&\geq&||\Lambda\Omega\hat{u}||_{\ell^1}-
||\Omega\tilde{\mu}\Omega^{-1}||
_{\ell^1\rightarrow\ell^1}||\Omega\hat{u}||_{\ell^1}\\
\label{17.9.20.4}&\geq&\frac{|\lambda|}{2}||\Omega\hat{u}||_{\ell^1}-
\frac{|\lambda|}{4}||\Omega\hat{u}||_{\ell^1}\\
\nonumber&=&\frac{|\lambda|}{4}||\Omega\hat{u}||_{\ell^1}\\
\label{17.9.20.3}&=&
\frac{|\lambda|}{4}|u|_{s,\mathbf{a},\mathbf{m}},
\end{eqnarray}
where inequalities \eqref{17.9.20.6} and \eqref{17.9.20.5} are used in \eqref{17.9.20.4}. From \eqref{17.9.20.3}, we get \eqref{17.9.20.7}. The proof of \eqref{17.9.20.8} is as follows:
\begin{eqnarray}
\nonumber|(1-\Gamma_K)(\mu u)|_{s-\sigma,\mathbf{a},\mathbf{m}}&\leq&\sum_{|k|>K}|
\sum_{|l|\leq K}\hat{\mu}_{k-l}\hat{u}_l|e^{|k|(s-\sigma)}e^{\mathbf{a}|\pi(k,\mathbf{m})|}\\
\nonumber&\leq&e^{-K\sigma}\sum_{|k|\geq K}|\sum_{|l|\leq K}\hat{\mu}_{k-l}\hat{u}_l|e^{|k|s}e^{\mathbf{a}|
\pi(k,\mathbf{m})|}\\
\nonumber&\leq& e^{-K\sigma}(\sum_{k\in\mathbb{Z}^n}|\hat{\mu}_{k}|e^{|k|s}e^{\mathbf{a}|\pi(k,0)
|})(\sum_{|k|\leq K}|\hat{u}_k|e^{|k|s}e^{\mathbf{a}|\pi(k,\mathbf{m})|})\\
\nonumber&\leq&e^{-K\sigma}\sum_{|k|\leq K}|\hat{p}_k|e^{|k|s}
e^{\mathbf{a}|\pi(k,\mathbf{m})|}\leq e^{-K\sigma}|p|_{s,\mathbf{a},\mathbf{m}},
\end{eqnarray}
where the last inequality comes from \eqref{17.7.20.10} and \eqref{17.9.20.3}.
\end{proof}

\subsection{Solving the homological equations}
Consider the conditions $\delta\leq1$ and $p-q\leq1$. Without loss of generality, we assume $\delta=p-q=1$ by increasing $\delta$ and  decreasing $q$ if necessary.
Let $\Omega=(\Omega_j:{j\in\mathbb{Z}_*}),\bar{\Omega}=[\Omega]$ and $\tilde{\Omega}=\Omega-[\Omega]$. Define
$$\langle k\rangle=\max\{1,|k|\},\quad\langle l\rangle_{\infty}=\max\{1,\sup_{j\in\mathbb{Z}_*}|jl_j|\}.$$
Equations \eqref{14.2.4.10}-\eqref{14.2.4.16} will be solved under the following conditions: uniformly on $\mathcal{O}$,
\begin{eqnarray}
\label{17.7.20.22}|\langle l,\bar{\Omega}(\xi)\rangle|&\geq& m|\sum_{j\in\mathbb{Z}_*}j^2l_j|,\quad |l|\leq2,\\
\label{17.7.20.23}|\tilde{\Omega}_j(\xi)|_{s,\tau+1}+
|\tilde{\Omega}_j(\xi)|_{s,\mathbf{a},0}
&\leq&\alpha_1\gamma_0|j|,\quad j\in\mathbb{Z}_*,\\
\label{17.9.21.1}|\langle k,\omega(\xi)\rangle+\langle l,\bar{\Omega}(\xi)\rangle|&\geq&\alpha_1\frac{\langle l\rangle_{\infty}}{\langle k\rangle^{\tau}},
\quad k\in\mathbb{Z}^n\setminus\{0\}, |l|\leq2, \ l\neq e_{-j}-e_j,\\
\label{17.7.20.21}|\langle k,\omega(\xi)\rangle+ \bar{\Omega}_{-j}(\xi)-\bar{\Omega}_{j}(\xi)|&\geq&\alpha_2\frac{|j|}{\langle k\rangle^{\tau}},\quad k\in\mathbb{Z}^n,\pm j\in\mathbb{Z}_*,|j|\leq\Pi,
\end{eqnarray}
with constants $\tau\geq n,0<\gamma_0\leq\frac14$, $\Pi>0$, $0<\alpha_2\leq\alpha_1\leq m$, $0<s<1$, $0<\mathbf{a}\leq\frac{s}{80C_J}$.
 We mention that $\alpha_1,\alpha_2,m,\mathbf{a},s,\Pi$
will be the iteration parameters $\alpha_{1,\nu},\alpha_{2,\nu},m_{\nu},\mathbf{a}_{\nu},s_{\nu},\Pi_{\nu}$ in the $\nu$-th KAM step.

Equations \eqref{14.2.4.10} and \eqref{14.2.4.11} can be easily solved by a standard approach in classical, finite dimensional KAM theory, we only give the related results at this end of subsection. Equations \eqref{14.2.4.12}-\eqref{14.2.4.15} are easier than \eqref{14.2.4.16} and can be solved in the same way as
\eqref{14.2.4.16} done, so we only give the details of solving \eqref{14.2.4.16}.

Set $C_0=2|\omega|_{\mathcal{O}}/m$ and $K$ being positive numbers which will be the iteration parameter $K_{\nu}$ in the $\nu$-th KAM step. Recall the definition of the momentum $\pi(k,\alpha,\beta)$ of the scalar monomial $e^{\mathbf{i}k\cdot x}y^lz^{\alpha}\bar{z}^{\beta}$ in \eqref{2.1.14},
 the momentum for $e^{\mathbf{i}k\cdot x}z_i\bar{z}_j$ is
\begin{equation*}\pi(k,i,j)=\sum_{b=1}^nk_bj_b+i-j=\pi(k,i-j),\end{equation*}
where in the last equality we use \eqref{17.11.4.1}.
\begin{itemize}
\item[(1)] For $\max\{|i|,|j|\}\leq C_0K, i\neq \pm j$, we solve exactly \eqref{14.2.4.16}:
\begin{equation}\label{17.7.26.3}\partial_{\omega}F^{z\bar{z}}_{ij}+\mathbf{i}\ (\Omega_i-\Omega_j)F^{z\bar{z}}_{ij}=R^{z\bar{z}}_{ij};\end{equation}
\item[(2)] For $\max\{|i|,|j|\}>C_0K, i\neq \pm j$, we solve the truncated equation of \eqref{14.2.4.16}:
\begin{equation}\label{17.7.26.4}\partial_{\omega}F^{z\bar{z}}_{ij}+\mathbf{i}\ \Gamma_K((\Omega_i-\Omega_j)F^{z\bar{z}}_{ij})=\Gamma_KR^{z\bar{z}}_{ij},
\quad \Gamma_KF^{z\bar{z}}_{ij}=F^{z\bar{z}}_{ij};\end{equation}
\item[(3)] For $i=-j$ and $|j|\leq\Pi$, we solve exactly \eqref{14.2.4.16}:
\begin{equation}\label{17.9.20.1}\partial_{\omega}F^{z\bar{z}}_{(-j)j}
+\mathbf{i}(\Omega_{-j}-
\Omega_j)
F_{(-j)j}^{z\bar{z}}=R_{(-j)j}^{z\bar{z}};\end{equation}
\item[(4)] For $i=-j$ and $|j|>\Pi$, let
    \begin{equation}\label{17.9.26.4}F^{z\bar{z}}_{(-j)j}=0 \end{equation}
     and put $R^{z\bar{z}}_{(-j)j}$ into the perturbation.
\end{itemize}
Combining \eqref{17.7.26.3}-\eqref{17.9.26.4}, we find that \eqref{17.7.21.10} does not vanish. Actually, at this time, \eqref{17.7.21.10} is equal to $\langle \hat{R}^{z\bar{z}}z,\bar{z}\rangle$ with the elements of $\hat{R}^{z\bar{z}}$ being defined by
\begin{equation}\label{17.7.28.1}\hat{R}^{z\bar{z}}_{ij}=\begin{cases}0,\quad \quad \quad\quad
\quad\quad\quad\quad\quad\quad\quad\quad\quad\ \ \max\{|i|,|j|\}\leq C_0K,i\neq \pm j,\\
(1-\Gamma_K)(-\mathbf{i}(\Omega_i-\Omega_j)F^{z\bar{z}}_{ij}+
R^{z\bar{z}}_{ij}),\quad \max\{|i|,|j|\}>C_0K,i\neq \pm j,\\
0,\quad\quad\quad\quad\quad\quad\quad\ \quad\quad\quad\quad\quad\quad \ i=-j,|j|\leq \Pi,\\
R^{z\bar{z}}_{(-j)j},\quad \quad\quad\quad\quad\quad\quad\quad\quad\quad\quad\ \  i=-j, |j|>\Pi.\end{cases}\end{equation}
Letting $\Omega_{ij}=\Omega_i-\Omega_j=\bar{\Omega}_{ij}+\tilde{\Omega}_{ij}$ and dropping the superscript $z\bar{z}$ for brevity, \eqref{17.7.26.3}-\eqref{17.7.28.1} become
\begin{equation}\label{17.7.28.2}-\mathbf{i}\partial_{\omega}F_{ij}
+\bar{\Omega}_{ij}F_{ij}+
\tilde{\Omega}_{ij}F_{ij}=-\mathbf{i}\ R_{ij},\end{equation}
\begin{equation}\label{17.7.28.3}-\mathbf{i}\partial_{\omega}F_{ij}+
\bar{\Omega}_{ij}F_{ij}+
\Gamma_K(\tilde{\Omega}_{ij}F_{ij})=-
\mathbf{i}\ \Gamma_KR_{ij},\quad \Gamma_KF_{ij}=F_{ij},\end{equation}
\begin{equation}\label{17.9.22.1}-\mathbf{i}\partial_{\omega}F_{(-j)j}+
\bar{\Omega}_{(-j)j}F_{(-j)j}+\tilde{\Omega}_{(-j)j}F_{(-j)j}
=-\mathbf{i}R_{(-j)j},\end{equation}
\begin{equation}F_{(-j)j}=0,\end{equation}
\begin{equation}\label{17.7.28.4}\hat{R}_{ij}=\begin{cases}0,&
\max\{|i|,|j|\}\leq C_0K,i\neq \pm j,\\
(1-\Gamma_K)(-\mathbf{i}\tilde{\Omega}_{ij}F_{ij}+
R_{ij}),& \max\{|i|,|j|\}>C_0K,i\neq \pm j,\\
0,&i=-j,|j|\leq\Pi,\\
R_{(-j)j},& i=-j,|j|>\Pi.\end{cases}\end{equation}

We are now in position to solve the homological equation \eqref{17.7.28.2}-\eqref{17.9.22.1} by using Lemma \ref{lem17.7.31.1} and Lemma \ref{lem17.9.22.1} in the last subsection.
In what follows the notation $a\lessdot b$ stands for ``there exists a positive constant $c$ such that $a\leq cb$, where $c$
can only depend on $n,\tau$.''

First, let us consider \eqref{17.7.28.2} for $(i,j)$ with $\max\{|i|,|j|\}\leq C_0K,i\neq \pm j$. From \eqref{17.9.21.1}, we get
\begin{equation}\label{17.7.28.5}|\langle k,\omega\rangle|\geq\frac{\alpha_1}{|k|^{\tau}},\quad k\in\mathbb{Z}^n\setminus\{0\},\end{equation}
\begin{equation}\label{17.7.28.6}|\langle k,\omega\rangle+\bar{\Omega}_{ij}|\geq
\frac{\alpha_1\max\{|i|,|j|\}}{1+|k|^{\tau}},
\quad k\in\mathbb{Z}^n.\end{equation}
From \eqref{17.7.20.23} we get
\begin{eqnarray}\label{17.7.31.1}|\tilde{\Omega}_{ij}|_{s,\tau+1}\leq
\alpha_1\gamma_0(|i|+|j|)
\leq2\alpha_1\gamma_0
\max\{|i|,|j|\}.\end{eqnarray}
Setting $\sigma=\frac{s}{20}$, then we get $4\mathbf{a}C_J\leq\sigma$.
Applying Lemma \ref{lem17.7.31.1} to \eqref{17.7.28.2}, we have
\begin{equation}\label{17.7.31.2}|F_{ij}|_{s-\sigma,\mathbf{a},i-j}
\lessdot\frac{e^{4\gamma_0\max\{|i|,|j|\}s}}{\alpha_1
\max\{|i|,|j|\}\sigma^{2n+\tau}}
|R_{ij}|_{s,\mathbf{a},i-j}.\end{equation}
In view of $$\max\{|i|,|j|\}\leq C_0K,$$ we get
\begin{equation}\label{17.7.31.2}|F_{ij}|_{s-\sigma,\mathbf{a},i-j}
\lessdot\frac{e^{4C_0\gamma_0Ks}}{\alpha_1\max\{|i|,|j|\}\sigma^{2n+\tau}}
|R_{ij}|_{s,\mathbf{a},i-j}.\end{equation}

Then, let us consider \eqref{17.7.28.3} for $(i,j)$ with $\max\{|i|,|j|\}>C_0K,i\neq \pm j$. From \eqref{17.7.20.22} \eqref{17.7.20.23} and $C_0=2|\omega|_{\mathcal{O}}/m$, we get
\begin{equation}|\bar{\Omega}_{ij}|\geq m|i^2-j^2|\geq m\max\{|i|,|j|\}>mC_0K=2|\omega|_{\mathcal{O}}K,\end{equation}
$$|\tilde{\Omega}_{ij}|_{s,\mathbf{a},0}
\leq\alpha_1\gamma_0(|i|+|j|)
\leq\frac{\alpha_1\gamma_0|\bar{\Omega}_{ij}|}{m}
\leq\frac{|\bar{\Omega}_{ij}|}{4}
.$$
Now applying Lemma \ref{lem17.9.22.1} to \eqref{17.7.28.3}, we have
\begin{equation}\label{17.9.12.1}|F_{ij}|_{s,\mathbf{a},i-j}\lessdot
\frac{1}{m\max\{|i|,|j|\}}|R_{ij}|_{s,\mathbf{a},i-j},\end{equation}

\begin{equation}\label{17.9.12.2}|(1-\Gamma_K)
(\tilde{\Omega}_{ij}F_{ij})|_{s,\mathbf{a},i-j}
\leq e^{-K\sigma}
|R_{ij}|_{s,\mathbf{a},i-j}.\end{equation}

Finally, let us consider \eqref{17.9.22.1} for $i=-j$. From \eqref{17.9.21.1} \eqref{17.7.20.21}, we get
\begin{equation}|\langle k,\omega\rangle|\geq\frac{\alpha_1}{\langle k\rangle^{\tau}},\quad k\in\mathbb{Z}^n\setminus\{0\},\end{equation}
\begin{equation}|\langle k,\omega\rangle+\bar{\Omega}_{-j}-\bar{\Omega}_j|
\geq\frac{\alpha_2|j|}{1+|k|^{\tau}},\quad k\in\mathbb{Z}^n.\end{equation}
From \eqref{17.7.20.23}, we get
\begin{equation}|\tilde{\Omega}_{-j}-\tilde{\Omega}_{j}|_{s,\tau+1}\leq
2\alpha_1\gamma_0|j|.\end{equation}
Applying Lemma \ref{lem17.7.31.1} to \eqref{17.9.22.1}, we have
\begin{equation}\label{17.9.26.1}|F_{(-j)j}|_{s-\sigma,\mathbf{a},-2j}
\lessdot\frac{e^{4\gamma_0|j|s}}{\alpha_2
|j|\sigma^{2n+\tau}}
|R_{(-j)j}|_{s,\mathbf{a},-2j}.\end{equation}
In view of $$|j|\leq \Pi,$$ we get
\begin{equation}\label{17.9.26.2}|F_{(-j)j}|_{s-\sigma,\mathbf{a},-2j}
\lessdot\frac{e^{4\gamma_0\Pi s}}{\alpha_2
|j|\sigma^{2n+\tau}}
|R_{(-j)j}|_{s,\mathbf{a},-2j}.\end{equation}

By the definition of the Hamiltonian vector field in \eqref{7.13.2}, we obtain
\begin{equation}X_{\langle F^{z\bar{z}}z,\bar{z}\rangle}=
(0,-(\sigma_{j_b}\partial_{x_b}\langle F^{z\bar{z}}z,\bar{z}\rangle)_{1\leq b\leq n},-\mathbf{i}(\sigma_j\partial_{\bar{z}_j}\langle F^{z\bar{z}}z,\bar{z}\rangle
)_{j\in\mathbb{Z}_*},\mathbf{i}(\sigma_j\partial_{z_j}\langle F^{z\bar{z}}z,\bar{z}\rangle
)_{j\in\mathbb{Z}_*})^T.\end{equation}
Using Lemma \ref{lem17.9.7.1} and \eqref{17.7.31.2} \eqref{17.9.12.1} \eqref{17.9.26.2}, we get
\begin{eqnarray}\nonumber||X_{\langle F^{z\bar{z}}z,\bar{z}\rangle}||_{s-2\sigma,r,p,\mathbf{a}}
&\lessdot&\frac{\max\{e^{4\gamma_0C_0Ks},e^{4\gamma_0\Pi s}\}}{\alpha_2\sigma^{3n+\tau}}||X_{\langle R^{z\bar{z}}z,\bar{z}\rangle}||_{s,r,p-1,\mathbf{a}}\\
\label{17.9.27.3}&\leq&
\frac{\max\{e^{4\gamma_0C_0Ks},e^{4\gamma_0\Pi s}\}}{\alpha_2\sigma^{3n+\tau}}||X_R||_{s,r,p-1,\mathbf{a}}.\end{eqnarray}

To obtain the Lipschitz semi-norm, we proceed as follows. Shorting $\Delta_{\xi\eta}$ as $\Delta$ and applying it to
\eqref{17.7.28.2}-\eqref{17.9.22.1}, one gets that, for $(i,j)$ with $\max\{|i|,|j|\}\leq C_0K$,
\begin{equation}\label{17.9.13.1}\mathbf{i}\partial_{\omega}(\Delta F_{ij})+\bar{\Omega}_{ij}\Delta F_{ij}+\tilde{\Omega}_{ij}\Delta F_{ij}=-\mathbf{i}\partial_{\Delta\omega}F_{ij}-(\Delta\Omega_{ij})F_{ij}+\mathbf{i}\Delta R_{ij}:=Q_{ij},\end{equation}
 for $(i,j)$ with $\max\{|i|,|j|\}>C_0K$,
\begin{equation}\label{17.9.13.2}\mathbf{i}\partial_{\omega}(\Delta F_{ij})+\bar{\Omega}_{ij}\Delta F_{ij}+\Gamma_K(\tilde{\Omega}_{ij}\Delta F_{ij})
=-\mathbf{i}\partial_{\Delta{\omega}}F_{ij}-\Gamma_K((\Delta\Omega_{ij})
F_{ij}-\mathbf{i}\Delta R_{ij}):=Q_{ij},\end{equation}
and that, for $i=-j$ and $|j|\leq\Pi$,
\begin{eqnarray}\nonumber&&\mathbf{i}\partial_{\omega}(\Delta F_{(-j)j})+\bar{\Omega}_{(-j)j}\Delta F_{(-j)j}+\tilde{\Omega}_{(-j)j}\Delta F_{(-j)j}\\
\label{17.9.27.1}&&=-\mathbf{i}\partial_{\Delta\omega}F_{(-j)j}
-(\Delta\Omega_{(-j)j})F_{(-j)j}+
\mathbf{i}\Delta R_{(-j)j}:=Q_{(-j)j}.\end{eqnarray}
For $\max\{|i|,|j|\}<C_0K$, we have
\begin{eqnarray}\nonumber|Q_{ij}|_{s-2\sigma,\mathbf{a},i-j} &\leq&|\Delta\omega|\cdot\sup_{k\in\mathbb{Z}^n}|k|e^{-|k|\sigma}\cdot
|F_{ij}|_{s-\sigma,\mathbf{a},i-j}\\
\nonumber&&+|\Delta\Omega_{ij}|_{s-2\sigma,
\mathbf{a},0}\cdot|F_{ij}|_{s-2\sigma,\mathbf{a},i-j}+|\Delta R_{ij}|_{s-2\sigma,\mathbf{a},i-j}\\
\nonumber &\lessdot&\frac{e^{4C_0\gamma_0Ks}}{\alpha_1\sigma^{2n+\tau+1}}
(|\Delta\omega|+|\Delta\Omega|_{-1,s-2\sigma,\mathbf{a},0})
|R_{ij}|_{s,\mathbf{a},i-j}+|\Delta R_{ij}|_{s-2\sigma,\mathbf{a},i-j}\\
\label{17.9.13.3}&\lessdot&\frac{e^{4C_0\gamma_0Ks}}{\sigma^{2n+\tau+1}}
(\frac{|\Delta\omega|+|\Delta\Omega|_{-1,s-2\sigma,\mathbf{a},0}}{\alpha_1}
|R_{ij}|_{s,\mathbf{a},i-j}
+|\Delta R_{ij}|_{s-2\sigma,\mathbf{a},i-j}).
\end{eqnarray}
Applying Lemma \ref{lem17.7.31.1} to \eqref{17.9.13.1}, we have
\begin{eqnarray}\label{17.9.13.4}&&|\Delta F_{ij}|_{s-3\sigma,\mathbf{a},i-j}\\
 \nonumber&\lessdot&\frac{e^{8C_0\gamma_0Ks}}
{\alpha_1\max\{|i|,|j|\}\sigma^{4n+2\tau+1}}
(\frac{|\Delta\omega|+|\Delta{\Omega}|_{-1,s-2\sigma,
\mathbf{a},0}}{\alpha_1}
|R_{ij}|_{s,\mathbf{a},i-j}+|\Delta R_{ij}|_{s-2\sigma,\mathbf{a},i-j}).
\end{eqnarray}
For $\max\{|i|,|j|\}>C_0K$, similarly to \eqref{17.9.13.3}, we have
\begin{equation}\label{17.9.13.7}|Q_{ij}|_{s-\sigma,\mathbf{a},i-j}
 \lessdot\frac{1}{\sigma}
(\frac{|\Delta\omega|+
|\Delta\Omega|_{-1,s-\sigma,\mathbf{a},0}}{m}
|R_{ij}|_{s-\sigma,\mathbf{a},i-j}+|\Delta R_{ij}|_{s-\sigma,\mathbf{a},i-j}).
\end{equation}
Applying Lemma \ref{lem17.9.22.1} to \eqref{17.9.13.2}, we have
\begin{eqnarray}\nonumber&&|\Gamma_K\Delta F_{ij}|_{s-\sigma,\mathbf{a},i-j}\\
\label{17.9.13.5}&\lessdot&\frac{1}{m\max\{|i|,|j|\}\sigma}
(\frac{|\Delta\omega|+|\Delta\Omega|_{-1,s-\sigma,\mathbf{a},0}}{m}
|R_{ij}|_{s,\mathbf{a},i-j}+|\Delta R_{ij}|_{s-\sigma,\mathbf{a},i-j}),
\end{eqnarray}
\begin{eqnarray}
|(1-\Gamma_K)(\tilde{\Omega}_{ij}\Delta F_{ij})|_{s-\sigma,\mathbf{a},i-j}
\label{17.9.13.6}\lessdot
\frac{e^{-K\sigma}}{\sigma}
(\frac{|\Delta\omega|+
|\Delta\Omega|_{-1,s-\sigma,\mathbf{a},0}}{m}
|R_{ij}|_{s,\mathbf{a},i-j}+|\Delta R_{ij}|_{s-\sigma,\mathbf{a},
i-j}).
\end{eqnarray}
For $i=-j$ with $|j|\leq\Pi$, similarly to \eqref{17.9.13.4}, we have
\begin{eqnarray}\nonumber&&|\Delta F_{(-j)j}|
_{s-3\sigma,\mathbf{a},-2j}\\
\label{17.9.27.2}&\lessdot&\frac{e^{8\gamma_0\Pi s}}{
\alpha_2|j|\sigma^{4n+2\tau+1}}
\big(\frac{|\Delta\omega|+|\Delta \Omega|_{-1,s-2\sigma,\mathbf{a},0}}{\alpha_2}
|R_{(-j)j}|_{s,\mathbf{a},-2j}
+|\Delta R_{(-j)j}|_{s-2\sigma,\mathbf{a},-2j}\big).
\end{eqnarray}
In view of \eqref{17.9.13.4} \eqref{17.9.13.5} \eqref{17.9.27.2}, using Lemma \ref{lem17.9.7.1}, we get the estimate of the Hamiltonian vector field $X_{\langle\Delta F^{z\bar{z}}z,\bar{z}\rangle}$:
 \begin{eqnarray*}
 \label{17.9.27.4}\nonumber||X_{\langle \Delta F^{z\bar{z}}z,\bar{z}\rangle}||_{s-4\sigma,r,p,\mathbf{a}}
&\lessdot&
 \frac{\max\{e^{8C_0\gamma_0Ks},e^{8\gamma_0\Pi s}\}}
{\alpha_2\sigma^{4n+2\tau+2}}(\frac{|\Delta\omega|+
|\Delta\Omega|_{-1,s-4\sigma,\mathbf{a},0}}{\alpha_2}||X_{\langle R^{z\bar{z}}z,\bar{z}\rangle}||_{s,r,p-1,\mathbf{a}}\\
\nonumber&&+||X_{\langle\Delta R^{z\bar{z}}z,\bar{z}\rangle}||_{s,r,p-1,\mathbf{a}}).\end{eqnarray*}
Dividing by $|\xi-\zeta|\neq0$ and taking the supremum over $\mathcal{O}$, we get

\begin{eqnarray} \label{17.9.27.5}||X_{\langle F^{z\bar{z}}z,\bar{z}\rangle}||^{lip}_{s-4\sigma,r,p,\mathbf{a}
;\mathcal{O}}
&\lessdot&\frac{\max\{e^{8C_0\gamma_0Ks},e^{8\gamma_0\Pi s}\}}
{\alpha_2\sigma^{4n+2\tau+2}}\\
\nonumber&&\cdot(\frac{M}{\alpha_2}||X_{\langle R^{z\bar{z}}z,\bar{z}\rangle}||_{s,r,p-1,\mathbf{a};\mathcal{O}}+
||X_{\langle R^{z\bar{z}}z,\bar{z}\rangle}||^{lip}_{s,r,p-1,\mathbf{a};\mathcal{O}}),
\end{eqnarray}
where $M:=|\omega|_{\mathcal{O}}^{lip}+
|\Omega|_{-1,s,\mathbf{a},0;\mathcal{O}}^{lip}$.

Recall the definition of $||X||^{\lambda}_{s,r,q,\mathbf{a};\mathcal{O}}$ in \eqref{17.9.19.2}.
 Set $0\leq\lambda\leq\alpha_2/M$. From \eqref{17.9.27.3} and \eqref{17.9.27.5}, we get
\begin{equation}\label{17.9.13.12}||X_{\langle F^{z\bar{z}}z,\bar{z}\rangle}||_{s-4\sigma,r,p,\mathbf{a};
\mathcal{O}}^{\lambda}
\lessdot\frac{\max\{e^{8C_0\gamma_0Ks},e^{8\gamma_0\Pi s}\}}
{\alpha_2\sigma^{4n+2\tau+2}}||X_R||_{s,r,p-1,\mathbf{a};
\mathcal{O}}^{\lambda}
.\end{equation}

Now considering the homological equations \eqref{14.2.4.10} and \eqref{14.2.4.11} by a standard approach in finite dimensional KAM theory, we can easily get
\begin{equation}\label{17.7.25.2}||X_{F^x}||_{s-\sigma,r,p,\mathbf{a};
\mathcal{O}},||X_{\langle F^y,y\rangle}||_{s,r,p,\mathbf{a};\mathcal{O}}\lessdot
\frac{1}{\alpha_1\sigma^{\tau}}||X_R||_{s,r,p-1,\mathbf{a};\mathcal{O}}
,\end{equation}
\begin{equation}\label{17.7.25.3}||X_{F^x}||_{s-2\sigma,r,p,\mathbf{a};
\mathcal{O}}^{lip},||X_{\langle F^y,y\rangle}||_{s-2\sigma,r,p,\mathbf{a};\mathcal{O}}^{lip}
\lessdot\frac{1}{\alpha_1\sigma^{2\tau+1}}\big(\frac{M}{\alpha_1}
||X_R||_{s,r,p-1,\mathbf{a};\mathcal{O}}+
||X_R||_{s,r,p-1,\mathbf{a};\mathcal{O}}^{lip}\big)
.\end{equation}
From \eqref{17.7.25.2} and \eqref{17.7.25.3}, we get
\begin{equation}\label{17.9.13.13}||X_{F^x}||_{s-2\sigma,r,p,\mathbf{a};
\mathcal{O}}^{\lambda},||X_{\langle F^y,y\rangle}||_{s-2\sigma,r,p,\mathbf{a};\mathcal{O}}^{\lambda}\\
\lessdot\frac{1}{\alpha_1\sigma^{2\tau+1}}
||X_R||^{\lambda}_{s,r,p-1,\mathbf{a},\mathcal{O}}.\end{equation}

For the other terms of $F$, i.e., $\langle F^z,z\rangle,\langle F^{\bar{z}},\bar{z}\rangle$, $\langle F^{zz}z,z\rangle,$ $\langle F^{\bar{z}\bar{z}}\bar{z},\bar{z}\rangle$, the same results-even better- than \eqref{17.9.13.12} can be obtained. Thus, we finally get the estimate for $F$:
\begin{equation}\label{17.10.11.9}||X_{F}||_{s-4\sigma,r,p,\mathbf{a};
\mathcal{O}}^{\lambda}
\lessdot\frac{\max\{e^{8C_0\gamma_0Ks},e^{8\gamma_0\Pi s}\}}
{\alpha_2\sigma^{4n+2\tau+2}}||X_R||_{s,r,p-1,\mathbf{a};
\mathcal{O}}^{\lambda}.\end{equation}

\section{The New Hamiltonian}

From \eqref{17.7.21.3}-\eqref{17.7.21.11} we get the new Hamiltonian
\begin{equation}H\circ\Phi=N_++P_+,\end{equation}
where $N_+=\eqref{17.7.21.3}$ and
\begin{equation}\label{17.10.11.8}P_+=\hat{R}+\int_{0}^1\{(1-t)(\hat{N}+\hat{R})+tR,F\}\circ X_F^tdt+(P-R)\circ X_F^1,\end{equation}
where $\hat{R}=\eqref{17.7.21.6}+\cdots+\eqref{17.7.21.10}:=
\langle\hat{R}^z,z\rangle+\langle\hat{R}^{\bar{z}},\bar{z}\rangle+\langle\hat{R}^{zz}z,z\rangle
+\langle\hat{R}^{z\bar{z}}z,\bar{z}\rangle+\langle\hat{R}^{\bar{z}\bar{z}}\bar{z},\bar{z}\rangle.$ The aim of this subsection is to estimate the new normal form
$N_+$ and the new perturbation $P_+$.
\\

\subsection{ The New Normal Form} In view of \eqref{17.7.21.3}, denote $N_+=N+\hat{N}$ with
$$\hat{N}=\langle\hat{\omega},y\rangle+
\sum_{j\in\mathbb{Z}_*}\hat{\Omega}_jz_j\bar{z}_j,$$
where \begin{equation}\label{17.9.13.14}\hat{\omega}:=[R^y],\end{equation}
\begin{equation}\label{17.9.13.15}\hat{\Omega}_j:=R_{jj}+\sigma_j\langle\partial_x\Omega_j,F^y\rangle
=R_{jj}+\sigma_j\langle\partial_x\tilde{\Omega}_j,F^y\rangle.\end{equation}
From \eqref{17.9.13.14} we easily get
 \begin{equation}\label{17.9.14.1}|\hat{\omega}|^{\lambda}_{\mathcal{O}}
 \lessdot s||X_R||_{s,r,p-1,\mathbf{a};\mathcal{O}}^{\lambda}.\end{equation}
In the following, we estimate $\hat{\Omega}=(\hat{\Omega}_j:j\in\mathbb{Z}_*)$. In view of \eqref{17.7.25.2},
\begin{equation}|\sigma_j\langle\partial_x\tilde{\Omega}_j,F^y\rangle
|_{s-\sigma,\mathbf{a},0}\lessdot (s-\sigma)\frac{|\tilde{\Omega}_j|_{s,\mathbf{a},0}}
{\sigma}||X_{\langle F^y,y\rangle}||_{s-\sigma,r,p,\mathbf{a}}
\lessdot s\frac{\gamma_0|j|}{\sigma^{\tau+1}}
||X_R||_{s,r,p-1,\mathbf{a}}.\end{equation}
Thus, together with
\begin{equation}|R_{jj}|_{s-\sigma,\mathbf{a},0}\leq |j|\cdot||X_R||_{s,r,p-1,\mathbf{a}},\end{equation}
we get
\begin{equation}\label{17.9.14.2}|\hat{\Omega}|
_{-1,s-\sigma,\mathbf{a},0}\lessdot\frac{1}{\sigma^{\tau+1}}
||X_R||_{s,r,p-1,\mathbf{a}}.\end{equation}
Applying $\Delta$ to $\hat{\Omega}_j$, we have
\begin{equation}\Delta\hat{\Omega}_j=\Delta R_{jj}+\sigma_j\langle\partial_x\Delta\tilde{\Omega}_j,F^y\rangle+\sigma_j
\langle\partial_x\tilde{\Omega}_j,\Delta F^y\rangle.\end{equation}
Since
$$|\Delta R_{jj}|_{s,\mathbf{a},0}\leq |j|\cdot||\Delta X_R||_{s,r,p-1,\mathbf{a}},$$
\begin{eqnarray*}|\langle\partial_x\Delta\tilde{\Omega}_j,F^y\rangle|_{s-\sigma,
\mathbf{a},0}&\lessdot& (s-\sigma)\frac{1}{\sigma}|\Delta\tilde{\Omega}_j|_{s,
\mathbf{a},0}||X_{\langle F^y,y\rangle}||_{s-\sigma,r,p,\mathbf{a}}\\
&\lessdot&s|j|\frac{|\Delta\tilde{\Omega}|_{-1,s,\mathbf{a},0}
}{\alpha_1\sigma^{\tau+1}}||X_R||_{s,r,p-1,\mathbf{a}}
\end{eqnarray*}
and
\begin{eqnarray*}|\langle\partial_x\tilde{\Omega}_j,\Delta F^y\rangle|_{s-2\sigma,\mathbf{a},0}
&\lessdot&(s-2\sigma)\frac{1}{\sigma}|\tilde{\Omega}_j|_{s-\sigma,\mathbf{a},0}||\Delta X_{\langle F^y,y\rangle}||_{s-2\sigma,r,p,\mathbf{a}}\\
&\lessdot&s\frac{\gamma_0|j|}{\sigma^{2\tau+2}}(\frac{M}{\alpha_1}
|| X_R||_{s,r,p-1,\mathbf{a}}+
||\Delta X_R||_{s,r,p-1,\mathbf{a}}),
\end{eqnarray*}
we get
\begin{equation}\label{17.9.14.3}|\hat{\Omega}|^{lip}
_{-1,s-2\sigma,\mathbf{a},0;\mathcal{O}}
\lessdot\frac{1}{\sigma^{2\tau+2}}
(\frac{M}{\alpha_1}||X_R||_{s,r,p-1,\mathbf{a};\mathcal{O}}
+||X_R||_{s,r,p-1,\mathbf{a};\mathcal{O}}^{lip}).
\end{equation}
Thus, from \eqref{17.9.14.2} and \eqref{17.9.14.3}, we get
\begin{equation}\label{17.9.14.4}|\hat{\Omega}|^{\lambda}
_{-1,s-2\sigma,\mathbf{a},0;\mathcal{O}}\lessdot\frac{1}{\sigma^{2\tau+2}}
||X_R||_{s,r,p-1,\mathbf{a};\mathcal{O}}^{\lambda}.\end{equation}

\subsection {The New Perturbation} We firstly estimate the error term $\hat{R}^{z\bar{z}}$
with its matrix elements $\hat{R}_{ij}$ in \eqref{17.7.28.4}.

For $\max\{|i|,|j|\}>C_0K,i\neq \pm j$, by \eqref{17.9.12.2} and
\begin{equation*}|(1-\Gamma_K)R_{ij}|
_{s-\sigma,\mathbf{a},i-j}\leq e^{-K\sigma}
|R_{ij}|_{s,\mathbf{a},i-j},\end{equation*}
we obtain
\begin{equation}\label{17.11.10.4}|\hat{R}_{ij}|_{s-\sigma,\mathbf{a},i-j}\leq2e^{-K\sigma}
|R_{ij}|_{s,\mathbf{a},i-j}\ ;\end{equation}
by \eqref{17.9.13.6} and
\begin{eqnarray}\nonumber|(1-\Gamma_K)(\Delta\tilde{\Omega}_{ij}\cdot F_{ij})|_{s-\sigma,\mathbf{a},i-j}&\leq&e^{-K\sigma}(|i|+|j|)
|\Delta\Omega|_{-1,s,\mathbf{a},0}
|F_{ij}|_{s,\mathbf{a},i-j}\\
\nonumber&\lessdot&\frac{e^{-K\sigma}}{m}
|\Delta\Omega|_{-1,s,\mathbf{a},0}|R_{ij}|_{s,\mathbf{a},i-j},
\end{eqnarray}
\begin{equation*}\label{17.11.9.1}|(1-\Gamma_K)\Delta R_{ij}|
_{s-\sigma,\mathbf{a},i-j}\leq e^{-K\sigma}
|\Delta R_{ij}|_{s,\mathbf{a},i-j},\end{equation*}
we obtain
\begin{equation}\label{17.11.10.5}|\Delta \hat{R}_{ij}|_{s-\sigma,\mathbf{a},i-j}
\lessdot\frac{e^{-K\sigma}}{\sigma}(\frac{|\Delta\omega|+
|\Delta\Omega|_{-1,s,\mathbf{a},0}}{m}|R_{ij}|_{s,\mathbf{a},i-j}+
|\Delta R_{ij}|_{s,\mathbf{a},i-j}).\end{equation}

For $i=-j,|j|>\Pi$, assuming $0<(\mathbf{a}-\mathbf{a}')C_J<\sigma$, then
$$|\pi(k,-2j)|\geq|2j|-|\sum_{b=1}^nk_bj_b|\geq2\Pi-C_J|k|\geq2\Pi-
\frac{|k|\sigma}{\mathbf{a}-\mathbf{a}'},$$
and thus
\begin{eqnarray}\nonumber|\hat{R}_{(-j)j}|_{s-\sigma,\mathbf{a}',-2j}&=&
\sum_{k\in\mathbb{Z}^n}|\hat{R}_{(-j)jk}|e^{|k|(s-\sigma)}
e^{\mathbf{a}|\pi(k,-2j)|}e^{-(\mathbf{a}-\mathbf{a}')
|\pi(k,-2j)|}\\ \nonumber&\leq&\sum_{k\in\mathbb{Z}^n}|\hat{R}_{(-j)jk}|
e^{|k|(s-\sigma)}e^{\mathbf{a}|\pi(k,-2j)|}
e^{-(\mathbf{a}-\mathbf{a}')(2\Pi-\frac{|k|\sigma}{\mathbf{a}-\mathbf{a}'})}\\
\nonumber&\leq&e^{-2(\mathbf{a}-\mathbf{a}')\Pi}\sum_{k\in\mathbb{Z}^n}
|\hat{R}_{(-j)jk}|e^{|k|s}e^{\mathbf{a}|\pi(k,-2j)|}\\
\label{17.10.11.2}&=&e^{-2(\mathbf{a}-\mathbf{a}')\Pi}
|R_{(-j)j}|_{s,\mathbf{a},-2j}.\end{eqnarray}
Similarly, we have
\begin{equation}\label{17.11.9.2}|\Delta \hat{R}_{(-j)j}|_{s-\sigma,\mathbf{a}',-2j}
\leq e^{-2(\mathbf{a}-\mathbf{a}')\Pi}
|\Delta R_{(-j)j}|_{s,\mathbf{a},-2j}.\end{equation}

Using Lemma \ref{lem17.9.7.1} below, from \eqref{17.11.10.4} \eqref{17.10.11.2}, we get
\begin{eqnarray}\nonumber||X_{\langle \hat{R}^{z\bar{z}}z,\bar{z}\rangle}||_{s-2\sigma,r,p-1,\mathbf{a}'}
&\lessdot&\frac{\max\{e^{-K\sigma},
e^{-2(\mathbf{a}-\mathbf{a}')\Pi}\}}{\sigma}||X_{\langle R^{z\bar{z}}z,\bar{z}\rangle}||_{s,r,p-1,\mathbf{a}}\\
\label{17.10.9.1}&\leq&\frac{\max\{e^{-K\sigma},e^{-2(\mathbf{a}-\mathbf{a}')\Pi}
\}}{\sigma}
||X_R||_{s,r,p-1,\mathbf{a}},\end{eqnarray}
and from \eqref{17.11.10.5} \eqref{17.11.9.2}, we get
\begin{equation}\label{17.10.11.5}||X_{\langle\hat{R}^{z\bar{z}}z,
\bar{z}\rangle}||^{lip}_{s-2\sigma,r,p-1,\mathbf{a'};\mathcal{O}}
\lessdot
\frac{\max\{e^{-K\sigma},
e^{-2(\mathbf{a}-\mathbf{a}')\Pi}\}}{\sigma^{2}}
(\frac{M}{m}||X_{R}||_{s,r,p-1,\mathbf{a};\mathcal{O}}+||X_{R}||^{lip}_{
s,r,p-1,\mathbf{a};\mathcal{O}}).\end{equation}
Therefore, from \eqref{17.10.9.1} and \eqref{17.10.11.5}, we get
\begin{equation}\label{17.10.11.6}||X_{\langle\hat{R}^{z\bar{z}}z,\bar{z}
\rangle}||^{\lambda}_{s-2\sigma,r,p-1,\mathbf{a'};\mathcal{O}}\lessdot
\frac{\max\{e^{-K\sigma},e^{-2(\mathbf{a}-\mathbf{a}')\Pi}\}}{\sigma^{2}}
||X_{R}||_{s,r,p-1,\mathbf{a};
\mathcal{O}}^{\lambda}.\end{equation}

For the other terms of $\hat{R}$, i.e., $\langle\hat{R}^z,z\rangle, \langle\hat{R}^{\bar{z}},\bar{z}\rangle,\langle\hat{R}^{zz}z,z\rangle,
\langle\hat{R}^{\bar{z}\bar{z}}\bar{z},\bar{z}\rangle$, the same results- even better- than \eqref{17.10.11.6} can be obtained. Thus, we finally get the estimate for the error term $\hat{R}$:
\begin{equation}\label{17.10.11.7}||X_{\hat{R}}
||^{\lambda}_{s-2\sigma,r,p-1,\mathbf{a'};\mathcal{O}}
\lessdot\frac{\max\{e^{-K\sigma},
e^{-2(\mathbf{a}-\mathbf{a}')\Pi}\}}{\sigma^2}||X_{R}||_{s,r,p-1,\mathbf{a};
\mathcal{O}}^{\lambda}.\end{equation}

Now we consider the new perturbation \eqref{17.10.11.8}.
 By setting $R(t)=(1-t)(\hat{N}+\hat{R})+tR$, we have
\begin{equation}
X_{P_+}=X_{\hat{R}}+\int_0^1 X_{\{R(t),F\}\circ\Phi_F^t}dt
+X_{(P-R)\circ\Phi_F^1}.
\end{equation}
We assume that
\begin{equation}\label{17.10.11.11}||X_P||_{s,r,p-1,\mathbf{a};\mathcal{O}}
^{\lambda}\leq\frac{\alpha_2\eta^2}{B_{\sigma}}\cdot
\frac{1}{\max\{e^{8C_{0}\gamma_0Ks},
 e^{8\gamma_0\Pi s}\}}
\end{equation}
for $0\leq\lambda\leq\alpha_2/M$ with some $0<\eta<1$ and $0<s<1,\sigma=s/20$, where
$B_{\sigma}=c\sigma^{-(4n+4\tau+5)}$ with $c$ being a sufficiently large constant depending only on $n$, $\tau$ and $|\omega|_{\mathcal{O}}$. Since $R$ is $2$-order Taylor polynomial truncation in $y,z,\bar{z}$ of $P$, we can obtain
\begin{equation}\label{17.10.11.10}||X_R||^{\lambda}_{s,r,p-1,\mathbf{a};
\mathcal{O}}\leq||X_P||^{\lambda}_{s,r,p-1,\mathbf{a};\mathcal{O}}
,\end{equation}
\begin{equation}\label{17.10.12.3}||X_P-X_R||^{\lambda}_{s,\frac65\eta r,p-1,\mathbf{a};\mathcal{O}}\leq\frac65\eta||X_P||^{\lambda}
_{s,r,p-1,\mathbf{a};\mathcal{O}}.\end{equation}
From \eqref{17.10.11.9} \eqref{17.9.14.1} \eqref{17.9.14.4} and \eqref{17.10.11.7}, we get
\begin{equation}\label{17.12.23.3}||X_{F}||_{s-4\sigma,r,p,\mathbf{a};
\mathcal{O}}^{\lambda}\lessdot\frac{\max\{e^{8C_0\gamma_0Ks},
e^{8\gamma_0\Pi s}\}}
{\alpha_2\sigma^{4n+2\tau+2}}||X_P||_{s,r,p-1,\mathbf{a};
\mathcal{O}}^{\lambda},
\end{equation}
\begin{equation}\label{17.10.12.1}||X_{\hat{N}}||_{s-2\sigma,r,p-1,\mathbf{a};
\mathcal{O}}^{\lambda}\lessdot\frac{1}{\sigma^{2\tau+2}}
||X_P||_{s,r,p-1,\mathbf{a};\mathcal{O}}^{\lambda},\end{equation}
\begin{equation}\label{17.10.12.2}||X_{\hat{R}}
||^{\lambda}_{s-2\sigma,r,p-1,\mathbf{a'};\mathcal{O}}\lessdot
\frac{\max\{e^{-K\sigma},e^{-2(\mathbf{a}-\mathbf{a}')\Pi}\}}{\sigma^{2}}
||X_{P}||_{s,r,p-1,\mathbf{a};
\mathcal{O}}^{\lambda}.\end{equation}
Therefore, by \eqref{17.10.11.10} \eqref{17.10.12.1} \eqref{17.10.12.2}, we get
\begin{equation}\label{17.12.24.2}||X_{R(t)}||^{\lambda}_{s-2\sigma,r,p-1,\mathbf{a}';
\mathcal{O}}\lessdot\frac{1}{\sigma^{2\tau+2}}
||X_P||_{s,r,p-1,\mathbf{a};\mathcal{O}}^{\lambda}.\end{equation}
By \eqref{17.12.23.3} \eqref{17.12.24.2} and Lemma \ref{lem17.12.22.1}, we get
\begin{eqnarray}\nonumber||[X_{\{R(t),F\}}||^{\lambda}_{s-5\sigma,r/2,p-1,
\mathbf{a}';\mathcal{O}}&\lessdot&||X_{R(t)}||
^{\lambda}_{s-4\sigma,r,p-1,\mathbf{a}';\mathcal{O}}
||X_F||^{\lambda}_{s-4\sigma,r,p,\mathbf{a}';\mathcal{O}}\\
\label{17.12.23.4}&\lessdot&\frac{\max\{e^{8C_0\gamma_0Ks},
e^{8\gamma_0\Pi s}\}}
{\alpha_2\sigma^{4n+4\tau+4}}(||X_P||^{\lambda}_{s,r,p-1,\mathbf{a};
\mathcal{O}})^2.\end{eqnarray}
Moreover, together with the smallness assumption \eqref{17.10.11.11}, by properly choosing $c$, we get
\begin{equation}\label{17.12.23.5}||X_F||^{\lambda}
_{s-4\sigma,r,p,\mathbf{a};\mathcal{O}}
\leq\frac{\eta^2\sigma}{c_0}\end{equation}
with some suitably large constant $c_0\geq1$ and thus
\begin{equation}\label{17.12.23.6}2^{2n+5}e\max\{\frac{s-4\sigma}
{(s-4\sigma)-(s-5\sigma)},\frac{r}{r-r/2}\}
||X_F||^{\lambda}_{s-4\sigma,r,p,\mathbf{a}';\mathcal{O}}\leq
2^{2n+9}e\frac{\eta^2\sigma}{c_0}<\frac{1}{10}.\end{equation}
By \eqref{17.12.23.4} \eqref{17.12.23.6} and using Lemma \ref{lem17.12.23.1}, for $-1\leq t\leq1$, the time $t$-Hamiltonian flow $$\Phi_F^t:D(s-5\sigma,r/2)\rightarrow D(s-4\sigma,r)$$
and we get
\begin{eqnarray}\nonumber||X_{\{R(t),F\}\circ\Phi_F^t}
||^{\lambda}_{s-5\sigma,r/2,p-1,\mathbf{a}';\mathcal{O}}
&\leq&\frac{10}{9}||[X_{\{R(t),F\}}||^{\lambda}_{s-4\sigma,r,p-1,\mathbf{a}';
\mathcal{O}}\\
&\lessdot&\frac{\max\{e^{8C_0\gamma_0Ks},
e^{8\gamma_0\Pi s}\}}
{\alpha_2\sigma^{4n+4\tau+4}}(||X_P||^{\lambda}_{s,r,p-1,\mathbf{a};
\mathcal{O}})^2.
\end{eqnarray}
Hence also
\begin{equation}||X_{\{R(t),F\}\circ\Phi_F^t}
||^{\lambda}_{s-5\sigma,\eta r,p-1,\mathbf{a}';\mathcal{O}}
\lessdot\frac{\max\{e^{8C_0\gamma_0Ks},
e^{8\gamma_0\Pi s}\}}
{\alpha_2\eta^2\sigma^{4n+4\tau+4}}(||X_P||^{\lambda}_{s,r,p-1,\mathbf{a};
\mathcal{O}})^2.\end{equation}
From \eqref{17.12.23.5}, we have
\begin{equation}||X_F||^{\lambda}_{s-4\sigma,\frac65\eta r,p,\mathbf{a};\mathcal{O}}\leq\frac{25\sigma}{36c_0},\end{equation}
 and thus
\begin{equation}\label{17.12.24.1}2^{2n+5}e\max\{\frac{s-4\sigma}
{(s-4\sigma)-(s-5\sigma)},\frac{\frac65\eta r}{\frac65\eta r-\eta r}\}
||X_F||^{\lambda}_{s-4\sigma,\frac65\eta r,p,\mathbf{a};\mathcal{O}}\leq
2^{2n+9}e\frac{25\sigma}{36c_0}<\frac{1}{10}.\end{equation}
By \eqref{17.10.12.3} \eqref{17.12.24.1} and using Lemma \ref{lem17.12.23.1}, the time $1$-Hamiltonian flow $$\Phi_F^1:D(s-5\sigma,\eta r)\rightarrow D(s-4\sigma,\frac65\eta r)$$
and we get
\begin{eqnarray}\nonumber||X_{(P-R)\circ\Phi_F^1}
||^{\lambda}_{s-5\sigma,\eta r,p-1,\mathbf{a};\mathcal{O}}
&\leq&\frac{10}{9}||X_{P-R}||^{\lambda}_{s-4\sigma,\frac65\eta r,p-1,\mathbf{a};\mathcal{O}}\\
&\leq&\frac43\eta ||X_P||^{\lambda}_{s,r,p-1,\mathbf{a};\mathcal{O}}.
\end{eqnarray}
Together with the estimate of $\hat{R}$ in \eqref{17.10.12.2}, we finally arrive at the estimate
\begin{eqnarray}\label{17.10.16.13}||X_{P_+}||_{s-5\sigma,\eta r,p-1,\mathbf{a}';\mathcal{O}}^{\lambda}
&\leq&(\frac{B_{\sigma}\max\{e^{8C_0\gamma_0Ks},
e^{8\gamma_0\Pi s}\}}{\alpha_2\eta^2}||X_P||^{\lambda}_{s,r,p-1,\mathbf{a};
\mathcal{O}}\\ \nonumber&&+
\frac{B_{\sigma}\max\{e^{-K\sigma},
e^{-2(\mathbf{a}-\mathbf{a}')\Pi}\}
}{\alpha_2\eta^2}+\frac43\eta)||X_P||^{\lambda}_{s,r,p-1,\mathbf{a};
\mathcal{O}}.
\end{eqnarray}
This is the bound for the new perturbation.

\section{Iteration and Convergence}
Set $\beta'=\frac12\min\{\frac{\beta}{1+\beta},\frac14\}$ and $\kappa=\frac{4}{3}-\frac{\beta'}{3}$. Now we give the precise set-up of iteration parameters. Let $\nu\geq0$ be the $\nu$-th KAM step.
\begin{itemize}

\item[]
$m_{\nu}=\frac{m_0}{10}(9+2^{-\nu})$, which is used for describing the growth of external frequencies,
\item[]
$E_{\nu}=\frac{E_0}{9}(10-2^{-\nu})$, which is used to dominate the norm of internal frequencies,
\item[]
$M_{1,\nu}=\frac{M_{1,0}}{9}(10-2^{-\nu}),M_{2,\nu}=
\frac{M_{2,0}}{9}(10-2^{-\nu}),
M_{\nu}=M_{1,\nu}+M_{2,\nu}$, which are used to dominate the Lipschitz semi-norm of frequencies,
\item[]
$M_{3,\nu}=\frac{M_{3,0}}{10}(9+2^{-\nu})$, which describes the lower bound for the sup-norm or Lipschitz semi-norm of the small divisors,
\item[]
$J_0=0$, $J_{\nu}=\gamma_0^{-\frac{\kappa^{\nu-1}}{\tau+1}}, \nu\geq1$,
  which are used for the estimate of measure,
\item[]
$s_{\nu}=s_02^{-\nu}$, which dominates the width of the angle variable $x$,
\item[]
$\sigma_{\nu}=s_{\nu}/20$, which serves as a bridge from $s_{\nu}$ to $s_{\nu+1}$,
\item[]
$\mathbf{a}_{\nu}=\sigma_{\nu}/C_J$, which is used to control higher momentum term,
\item[]
$B_{\nu}=24B_{\sigma_{\nu}}=c\sigma_{\nu}^{-(4n+4\tau+5)}$, here $c$ is a large constant only depending on $n,\tau$ and $E_0$,
\item[]
$\epsilon_{\nu}=(\epsilon_0\Pi_{\mu=0}^{\nu-1}(\frac{2^{\mu}
B_{\mu}}{\alpha_{0}})
^{\frac{1}{3\kappa^{\mu+1}}})^{\kappa^{\nu}}$, which dominates the size of the perturbation $P_{\nu}$ in the $\nu$-th KAM iteration,
\item[]
$K_{\nu}=5|\ln\epsilon_{\nu}|/(4\sigma_{\nu})$, which is the length of the truncation of Fourier series,
\item[]
$\Pi_{\nu}=5|\ln\epsilon_{\nu}|/(2\mathbf{a}_{\nu})$, which controls the number of homological equations with double normal frequencies,
\item[]
$\alpha_{1,\nu}=\frac{\alpha_0}{10}(9+2^{-\nu})$ and $\alpha_{2,\nu}=\alpha_0{2^{-\nu}}\Pi_{\nu}^{-1}$, which are used to dominate the measure of removed parameters,
\item[]
$\lambda_0=\frac{\alpha_0}{M_0},\quad\lambda_{\nu}=
\frac{\alpha_{2,\nu}}{M_{\nu}},\quad \nu\geq1$,
\item[]
$(2\eta_{\nu})^3=
\epsilon_{\nu}^{1-\beta'}
\alpha_{2,\nu}^{-1}B_{\nu}, \quad r_{\nu+1}=\eta_{\nu}r_{\nu},\quad D_{\nu}=D(s_{\nu},r_{\nu}).$
\end{itemize}

\subsection{Iterative Lemma}

\begin{lem}
Suppose that
\begin{equation}\label{17.10.16.1}\epsilon_0\leq
(\frac{\alpha_0\gamma_0}{80})^{\frac{1}{1-2\beta'}}
\prod_{\mu=0}^{\infty}(2^{\mu}B_{\mu})^{-\frac{1}{3\kappa^{\mu+1}}},
\quad\alpha_0\leq\min\{\frac{m_0}{10},\frac{M_{3,0}}{2}\}.
\end{equation}
Suppose $H_{\nu}=N_{\nu}+P_{\nu}$ is regular on $D_{\nu}\times\mathcal{O}_{\nu}$, where $N_{\nu}$ is a generalized normal form with coefficients satisfying
\begin{equation}\label{17.10.16.8}|\omega_{\nu}|_{\mathcal{O}_{\nu}}\leq E_{\nu},\quad|\omega_{\nu}|^{lip}_{\mathcal{O}_{\nu}}\leq M_{1,\nu},\end{equation}
\begin{equation}
\label{17.10.16.9}|\Omega_{\nu}|_{-1,s_{\nu},\mathbf{a}_{\nu},0;
\mathcal{O}_{\nu}}^{lip}\leq M_{2,\nu},
\end{equation}
\begin{equation}\label{17.10.16.10}|\langle l,\bar{\Omega}_{\nu}(\xi)\rangle|\geq m_{\nu}|\sum_{j\in\mathbb{Z}_*}j^2l_j|,\quad |l|\leq2,\end{equation}
\begin{equation}\label{17.10.16.6}|\tilde{\Omega}_{\nu,j}(\xi)|_{s_{\nu},
\tau+1}
+|\tilde{\Omega}_{\nu,j}(\xi)|_{s_{\nu},\mathbf{a}_{\nu},0}
\leq(\alpha_0-\alpha_{1,\nu})\gamma_0|j|,\quad j\in\mathbb{Z}_*,\end{equation}
\begin{eqnarray}\nonumber&&\inf_{\xi\in\mathcal{O}_{\nu}}|\langle k,\omega_{\nu}(\xi)\rangle+\langle l,\bar{\Omega}_{\nu}(\xi)\rangle|
+\inf_{\xi-\zeta// v_{kl}}\frac{|\Delta_{\xi\zeta}
    (\langle k,\omega_{\nu}\rangle+\langle l,\bar{\Omega}_{\nu}\rangle)|}{|\xi-\zeta|}\\ \label{17.11.14.1} &&\geq M_{3,\nu}\max\{|k|,\sum_{j\in\mathbb{Z}_*}|jl_j|\},\quad
    k\in\mathbb{Z}^n,
|l|\leq2\end{eqnarray}
on $\mathcal{O}_{\nu}$ and $P_{\nu}$ satisfies
\begin{equation}\label{17.12.24.5}||X_{P_{\nu}}||^{\lambda_{\nu}}_{s_{\nu},
r_{\nu},p-1,\mathbf{a}_{\nu};\mathcal{O}_{\nu}}
\leq\epsilon_{\nu}.\end{equation}
Let
\begin{equation}\label{17.10.16.11}\mathcal{O}_{\nu+1}=\mathcal{O}_{\nu}
\setminus\big(\bigcup_{k\in\mathbb{Z}^n\setminus\{0\},|l|\leq2\atop{l\neq e_{-j}-e_j}}
\mathcal{R}_{kl}^{\nu}(\alpha_{1,\nu})\cup\bigcup_{k\in\mathbb{Z}^n,\pm j\in\mathbb{Z}_*\atop{|j|\leq \Pi_{\nu}}}
\mathcal{R}_{k(-j)j}^{\nu}
(\alpha_{2,\nu})\big),\end{equation}
where
\begin{equation}\mathcal{R}_{kl}^{\nu}(\alpha_{1,\nu})=
\{\xi\in\mathcal{O}_{\nu}:|\langle k,\omega_{\nu}(\xi)\rangle+\langle l,\bar{\Omega}_{\nu}(\xi)\rangle|<\alpha_{1,\nu}\frac{\langle l\rangle_{\infty}}{\langle k\rangle^{\tau}}\},\end{equation}
\begin{equation}\mathcal{R}_{k(-j)j}^{\nu}(\alpha_{2,\nu})=
\{\xi\in\mathcal{O}_{\nu}:|\langle k,\omega_{\nu}(\xi)\rangle+ \bar{\Omega}_{\nu,(-j)}(\xi)-\bar{\Omega}_{\nu,j}(\xi)|
<\alpha_{2,\nu}\frac{|j|}{\langle k\rangle^{\tau}}\}.\end{equation}
Then there exists a Lipschitz family of real close-to-the-identity analytic symplectic coordinate transformation
$\Phi_{\nu+1}:D_{\nu+1}\times\mathcal{O}_{\nu+1}\rightarrow D_{\nu}$ satisfying
\begin{eqnarray}\nonumber&&||\Phi_{\nu+1}-\mbox{id}||^{\lambda_{\nu}}
_{s_{\nu},r_{\nu},p;D_{\nu+1}\times\mathcal{O}_{\nu+1}},
||D\Phi_{\nu+1}-I||^{\lambda_{\nu}}_{s_{\nu},r_{\nu},p,p;D_{\nu+1}
\times\mathcal{O}_{\nu+1}},\\
\label{17.12.24.3}&&||D\Phi_{\nu+1}-I||^{\lambda_{\nu}}
_{s_{\nu},r_{\nu},q,q;D_{\nu+1}
\times\mathcal{O}_{\nu+1}}\leq\frac{B_{\nu}}{\alpha_{2,\nu}}
\epsilon_{\nu}^{1-\beta'},\end{eqnarray}
where $||\cdot||_{s,r,p,p}$ denotes the operator norm induced by $||\cdot||_{s,r,p}$ and $||\cdot||_{s,r,p}$ in the source and target spaces, respectively,
such that for $H_{\nu+1}=H_{\nu}\circ\Phi_{\nu+1}=N_{\nu+1}+P_{\nu+1}$, the estimate
\begin{equation}\label{17.10.16.2}|\omega_{\nu+1}-\omega_{\nu}|
_{\mathcal{O}_{\nu+1}}^{\lambda_{\nu}},
\quad|\Omega_{\nu+1}-\Omega_{\nu}|^{\lambda_{\nu}}
_{-1,s_{\nu+1},\mathbf{a}_{\nu},0;
\mathcal{O}_{\nu+1}}\leq B_{\nu}\epsilon_{\nu}\end{equation}
holds and the same assumptions as above are satisfied with `$\nu+1$' in place of `$\nu$'.
\end{lem}
\begin{proof}
Setting $C_{0,\nu}=2E_{\nu}/m_{\nu}$, then it is obvious $C_{0,\nu}\leq4C_{0,0}$.
 %
Thus we have
\begin{equation}\label{17.12.24.4}e^{8C_{0,\nu}\gamma_0K_{\nu}s_{\nu}},
 e^{8\gamma_0\Pi_{\nu}s_{\nu}}\leq \epsilon_{\nu}^{-\beta'}\end{equation}
by $$
K_{\nu}s_{\nu}=20K_{\nu}\sigma_{\nu}=25|\ln\epsilon_{\nu}|,\quad
\quad\Pi_{\nu}s_{\nu}=50C_J|\ln\epsilon_{\nu}|$$
and choosing $\gamma_0$ small enough such that $800\gamma_0C_{0,0},400C_J\gamma_0\leq\beta'$. In view of the definition of $\eta_{\nu}$, namely, $(2\eta_{\nu})^3=\epsilon_{\nu}^{1-\beta'}\alpha_{2,\nu}^{-1}B_{\nu}$, the smallness condition \eqref{17.10.11.11}, namely,
 $$\epsilon_{\nu}\leq\frac{\alpha_{2,\nu}\eta^2_{\nu}}
 {B_{\nu}}\cdot
 \frac{1}{\max\{e^{8C_{0,\nu}\gamma_0K_{\nu}s_{\nu}},
 e^{8\gamma_0\Pi_{\nu}s_{\nu}}\}},$$
 is satisfied if
\begin{equation}\label{18.3.22.3}\epsilon_{\nu}^{1-\beta'}\leq
\frac{\alpha_{2,\nu}}{64B_{\nu}}.\end{equation}
  To verify the last inequality we argue as follows. As $\frac{2^{\nu}}{\alpha_{0}}$ and $B_{\nu}$ are increasing with $\nu$,
  \begin{equation}(\frac{2^{\nu}B_{\nu}}{\alpha_{0}})^{\frac{1}{1-2\beta'}}
  =(\frac{2^{\nu}B_{\nu}}{\alpha_{0}})^{\frac{1}{3(\kappa-1)}}=
 (\prod_{\mu=\nu}^{\infty}(\frac{2^{\nu}B_{\nu}}{\alpha_{0}})^{\frac{1}
 {3\kappa^{\mu+1}}
  })^{\kappa^{\nu}}\leq(\prod_{\mu=\nu}^{\infty}(\frac{2^{\mu}B_{\mu}}
  {\alpha_{0}})^
  {\frac{1}{3\kappa^{\mu+1}}})^{\kappa^{\nu}}.\end{equation}
  By the definition of $\epsilon_{\nu}$ above and the smallness condition on $\epsilon_0$ in \eqref{17.10.16.1},
  \begin{equation}\label{17.12.28.1}\epsilon_{\nu}^{1-2\beta'}
  \frac{2^{\nu}B_{\nu}}{\alpha_{0}}
 \leq(\epsilon_{0}\prod_{\mu=0}^{\infty}(\frac{2^{\mu}B_{\mu}}{\alpha_{0}})
 ^{\frac{1}{3\kappa^{\mu+1}}})^{\kappa^{\nu}(1-2\beta')}
 \leq(\frac{\gamma_0}{80})^{\kappa^{\nu}},\end{equation}
 and thus we can choose $\gamma_0$ small enough such that \begin{equation}\label{18.3.22.2}\epsilon_{\nu}^{\beta'}
 \leq2^{-2\nu}\Pi_{\nu}^{-2}.\end{equation}
  In view of \eqref{17.12.28.1} and \eqref{18.3.22.2}, we get
  \begin{equation}\label{18.3.22.4}\epsilon_{\nu}^{1-\beta'}
  \frac{2^{\nu}B_{\nu}}{\alpha_{0}}\leq
  (\frac{\gamma_0}{80})^{\kappa^{\nu}}2^{-2\nu}\Pi_{\nu}^{-2},\end{equation}
   which implies \eqref{18.3.22.3} since $(\frac{\gamma_0}{80})^{\kappa^{\nu}}2^{-2\nu}\Pi_{\nu}^{-1}
   \leq\frac{1}{64}$, and thus the smallness condition \eqref{17.10.11.11} is satisfied for each $\nu\geq0$. In particular, noticing $\kappa\geq\frac54,$ we have
 \begin{equation}\label{17.10.16.5}\epsilon_{\nu}^{1-\beta'}
 \frac{B_{\nu}}{\alpha_{2,\nu}}
 \leq\frac{\gamma_0}{2^{\nu+6}}.\end{equation}
  By \eqref{17.11.8.2} \eqref{17.12.23.3} \eqref{17.12.24.5} and \eqref{17.12.24.4}, we have
  \begin{eqnarray}\nonumber||X_{F_{\nu}}||^{\lambda_{\nu}}
  _{s_{\nu}-4\sigma_{\nu},r_{\nu},p;D(s_{\nu}-4\sigma_{\nu},
  r_{\nu})\times\mathcal{O}_{\nu+1}}&\leq&
||X_{F_{\nu}}||_{s_{\nu}-4\sigma_{\nu}, r_{\nu},p,\mathbf{a}_{\nu};
\mathcal{O}_{\nu+1}}^{\lambda_{\nu}}\\
\nonumber &\lessdot&\frac{\max\{e^{8C_{0,\nu}\gamma_0K_{\nu}s_{\nu}},
e^{8\gamma_0\Pi_{\nu}s_{\nu}}\}}{\alpha_{2,\nu}\sigma_{\nu}
^{4n+2\tau+2}}||X_{P_{\nu}}||_{s_{\nu},
r_{\nu},p-1,\mathbf{a}_{\nu};\mathcal{O}_{\nu+1}}^{\lambda_{\nu}}\\
\label{17.12.25.1}&\leq&\frac{1}{\alpha_{2,\nu}\sigma_{\nu}
^{4n+2\tau+2}}
\epsilon_{\nu}^{1-\beta'}.
\end{eqnarray}
In view of \eqref{17.10.16.5} and \eqref{17.12.25.1}, for suitably small $\gamma_0$, we have
 $$||X_{F_{\nu}}||^{\lambda_{\nu}}
  _{s_{\nu}-4\sigma_{\nu},r_{\nu},p;D(s_{\nu}-4\sigma_{\nu},
  r_{\nu})\times\mathcal{O}_{\nu+1}}\leq\frac{1}{15}.$$
Then the flow $X_{F_{\nu}}^t$ of the vector field $X_{F_{\nu}}$ exists on
$D(s_{\nu}-6\sigma_{\nu},r_{\nu}/2)$ for $-1\leq t\leq1$ and takes this domain into $D(s_{\nu}-5\sigma_{\nu},r_{\nu})$. Similarly, it takes $D(s_{\nu}-7\sigma_{\nu},r_{\nu}/4)$ into $D(s_{\nu}-6\sigma_{\nu},r_{\nu}/2)$.
 Moreover, in the same way as Lemma 19.3 and (20.6) of \cite{KP}, we get
 \begin{eqnarray}\nonumber&&||DX_{F_{\nu}}||
 ^{\lambda_{\nu}}_{s_{\nu},r_{\nu},p,p;
D(s_{\nu}-5\sigma_{\nu},r_{\nu})\times
\mathcal{O}_{\nu+1}},||DX_{F_{\nu}}||^{\lambda_{\nu}}
_{s_{\nu},r_{\nu},q,q;
D(s_{\nu}-5\sigma_{\nu},r_{\nu})\times\mathcal{O}_{\nu+1}}\\
\label{17.12.25.2}&&\lessdot
\frac{1}{\sigma_{\nu}}||X_{F_{\nu}}||^{\lambda_{\nu}}_{s_{\nu},
r_{\nu},p;D(s_{\nu}-4\sigma_{\nu},r_{\nu})\times\mathcal{O}_{\nu+1}},
\end{eqnarray}
 \begin{equation}
||X_{F_{\nu}}^t-\mbox{id}||^{\lambda_{\nu}}_{s_{\nu},r_{\nu},p;
D(s_{\nu}-6\sigma_{\nu},r_{\nu}/2)
\times\mathcal{O}_{\nu+1}}
\lessdot||X_{F_{\nu}}||^{\lambda_{\nu}}_{s_{\nu},r_{\nu},p;
D(s_{\nu}-5\sigma_{\nu},r_{\nu})
\times\mathcal{O}_{\nu+1}}
,\end{equation}
\begin{equation}||DX^t_{F_{\nu}}-I||^{\lambda_{\nu}}_{s_{\nu},r_{\nu},p,p;
D(s_{\nu}-7\sigma_{\nu},r_{\nu}/4)
\times\mathcal{O}_{\nu+1}}\lessdot||DX_{F_{\nu}}||^{\lambda_{\nu}}
_{s_{\nu},r_{\nu},p,p;
D(s_{\nu}-5\sigma_{\nu},r_{\nu})\times\mathcal{O}_{\nu+1}},\end{equation}
\begin{equation}\label{17.12.25.3}||DX^t_{F_{\nu}}-I||^{\lambda_{\nu}}_{s_{\nu},r_{\nu},
q,q;D(s_{\nu}-7\sigma_{\nu},r_{\nu}/4)
\times\mathcal{O}_{\nu+1}}\lessdot||DX_{F_{\nu}}||^{\lambda_{\nu}}_{s_{\nu},
r_{\nu},q,q;D(s_{\nu}-5\sigma_{\nu},r_{\nu})\times\mathcal{O}_{\nu+1}}.
\end{equation}
 Now there exists a coordinate transformation $$\Phi_{\nu+1}:=X_{F_{\nu}}^1:D_{\nu+1}
\times\mathcal{O}_{\nu+1}\rightarrow D_{\nu}$$
 taking $H_{\nu}$ into $H_{\nu+1}$. Moreover, \eqref{17.12.24.3} is obtained by \eqref{17.12.25.1}-\eqref{17.12.25.3} and
 \eqref{17.10.16.2} is obtained by \eqref{17.10.12.1}. More explicitly, \eqref{17.10.16.2} is written as
 \begin{equation}\label{17.10.16.3}|\omega_{\nu+1}-\omega_{\nu}|_
 {\mathcal{O}_{\nu+1}},\quad |\Omega_{\nu+1}-\Omega_{\nu}|_{-1,s_{\nu+1},\mathbf{a}_{\nu},0;
 \mathcal{O}_{\nu+1}}\leq B_{\nu}\epsilon_{\nu},\end{equation}
 \begin{equation}\label{17.10.16.4}|\omega_{\nu+1}-\omega_{\nu}|^{lip}
 _{\mathcal{O}_{\nu+1}},\quad |\Omega_{\nu+1}-\Omega_{\nu}|^{lip}_{-1,s_{\nu+1},\mathbf{a}_{\nu},0;
 \mathcal{O}_{\nu+1}}\leq\frac{M_{\nu}}{\alpha_{2,\nu}} B_{\nu}\epsilon_{\nu}.\end{equation}
 Actually, from \eqref{17.10.12.1}, we have
 \begin{equation}\label{18.3.19.1}|\Omega_{\nu+1}-\Omega_{\nu}|_{-1,
 s_{\nu-2\sigma_{\nu}},\mathbf{a}_{\nu},0;\mathcal{O}_{\nu+1}}\leq \sigma_{\nu}^{4n+2\tau+3}B_{\nu}\epsilon_{\nu},\end{equation}
and thus by Lemma \ref{lem18.3.19.1}, we have
 \begin{eqnarray}\nonumber|\Omega_{\nu+1,j}-\Omega_{\nu,j}|_{s_{\nu+1}
 ,\tau+1}
 &\leq&(\frac{\tau+1}{\tau})^{\tau+1}\frac{1}{(18\sigma_{\nu}
 )^{\tau+1}}|\Omega_{\nu+1,j}-\Omega_{\nu,j}|_{s_{\nu-2\sigma_{\nu}}
 ,\mathbf{a}_{\nu},0}\\
 \label{18.3.19.3}&\leq&(\frac{\tau+1}{18\tau})^{\tau+1}
 \sigma_{\nu}^{4n+\tau+2}B_{\nu}\epsilon_{\nu}|j|.
 \end{eqnarray}
 By \eqref{18.3.22.4}, we have
  \begin{equation}\label{17.12.28.2}B_{\nu}\epsilon_{\nu}\leq
  (\frac{\gamma_0}{80})^{\kappa^{\nu}}\alpha_{2,\nu}
  \epsilon_{\nu}^{\beta'}\leq\frac{\alpha_{2,\nu}
  \gamma_0^{\kappa^{\nu}}}{2^{\nu+6}}.\end{equation}
 In view of \eqref{17.10.16.3} \eqref{17.10.16.4} \eqref{18.3.19.3} \eqref{17.12.28.2}, by choosing $\gamma_0$ properly small,
  \eqref{17.10.16.8}-\eqref{17.10.16.6} are satisfied with `$\nu+1$' instead of `$\nu$'.

  In the following we only need to check \eqref{17.11.14.1} \eqref{17.12.24.5} with `$\nu+1$'.
 By \eqref{17.12.28.2} and choosing $\gamma_0$ properly small, we have
  \begin{equation}\label{17.12.28.3}B_{\nu}\epsilon_{\nu}
  \leq\frac{\alpha_{2,\nu}\gamma_0}{2^{\nu+6}}\leq \frac12\min\{\frac{M_{\nu}}{\alpha_{2,\nu}},1\}(M_{3,\nu}-M_{3,\nu+1}).
  \end{equation}
In view of \eqref{17.10.16.3} \eqref{17.10.16.4} \eqref{17.12.28.3}, for $k\in\mathbb{Z}^n$ and $|l|\leq2$, we have
\begin{eqnarray}\nonumber|\langle k,\omega_{\nu+1}-\omega_{\nu}\rangle+
  \langle l,\bar{\Omega}_{\nu+1}-\bar{\Omega}_{\nu}\rangle|
  &\leq&|k||\omega_{\nu+1}-\omega_{\nu}|+(\sum_{j\in\mathbb{Z}_*}|jl_j|)
  |\bar{\Omega}_{\nu+1}-\bar{\Omega}_{\nu}|_{-1}\\
   \nonumber &\leq&\max\{|k|,\sum_{j\in\mathbb{Z}^*}|jl_j|\}B_{\nu}\epsilon_{\nu}\\
\label{17.12.28.5}&\leq&\frac12(M_{3,\nu}-M_{3,\nu+1})\max\{|k|,
\sum_{j\in\mathbb{Z}^*}|jl_j|\}
  \end{eqnarray}
on $\mathcal{O}_{\nu+1}$, and
\begin{eqnarray}\nonumber|\langle k,\omega_{\nu+1}-\omega_{\nu}\rangle+
  \langle l,\bar{\Omega}_{\nu+1}-\bar{\Omega}_{\nu}\rangle|
  ^{lip}_{\mathcal{O}_{\nu+1}}
  &\leq&|k||\omega_{\nu+1}-\omega_{\nu}|^{lip}_{\mathcal{O}_{\nu+1}}
  +(\sum_{j\in\mathbb{Z}_*}|jl_j|)
  |\bar{\Omega}_{\nu+1}-\bar{\Omega}_{\nu}|^{lip}_{-1,\mathcal{O}_{\nu+1}}\\
   \nonumber &\leq&\max\{|k|,\sum_{j\in\mathbb{Z}^*}|jl_j|\}\frac{M_{\nu}}
   {\alpha_{2,\nu}}B_{\nu}\epsilon_{\nu}\\
\label{17.12.28.4}&\leq&\frac12(M_{3,\nu}-M_{3,\nu+1})\max\{|k|,
\sum_{j\in\mathbb{Z}^*}|jl_j|\}.
  \end{eqnarray}
Therefore, \eqref{17.11.14.1} is obtained by \eqref{17.12.28.5} \eqref{17.12.28.4} with `$\nu+1$' in place of `$\nu$'.
   Finally, from \eqref{17.10.16.13} we get
  \begin{eqnarray}\nonumber
||X_{P_+}||_{s_{\nu+1},r_{\nu+1},p-1,\mathbf{a}_{\nu+1};
\mathcal{O}_{\nu+1}}^{\lambda_{\nu+1}}
&\leq&(\frac{B_{\sigma_{\nu}}\max\{e^{8C_0\gamma_0Ks},
e^{4\gamma_0(\Pi+C_JK)s}\}}{\alpha_{2,\nu}\eta_{\nu}^2}\epsilon_{\nu}\\
\nonumber&&+
\frac{B_{\sigma_{\nu}}(e^{-9K_{\nu}\sigma_{\nu}/10}+e^{-(\mathbf{a}_{\nu}-
\mathbf{a}_{\nu+1})\Pi_{\nu}})
}{\alpha_{2,\nu}\eta_{\nu}^2}+\frac43\eta_{\nu})\epsilon_{\nu}\\
\nonumber&\leq&(\frac{B_{\nu}}{24\alpha_{2,\nu}\eta_{\nu}^2}
\epsilon_{\nu}^{1-\beta'}
+\frac{B_{\nu}}{24\alpha_{2,\nu}\eta_{\nu}^2}
\epsilon_{\nu}^{1-\beta'}+
\frac43\eta_{\nu})\epsilon_{\nu}\\
\nonumber&=&(\frac{B_{\nu}}{\alpha_{2,\nu}})^{\frac13}
\epsilon_{\nu}^{\kappa}\\
&=&\epsilon_{\nu+1}.
  \end{eqnarray}
  This completes the proof of the iterative lemma.
\end{proof}

\subsection{Convergence} We are now in a position to prove the KAM theorem. To apply the iterative lemma with $\nu=0$, we set
$$N_0=N,\quad P_0=P,\quad \mathcal{O}_0=\mathcal{O},\quad s_0=s,\quad r_0=r,\quad \mathbf{a}_0=\mathbf{a}$$
and similarly $E_0=E,M_{1,0}=M_1,M_{2,0}=M_2,M_0=M_1+M_2,M_{3,0}=M_3,
m_0=m,\alpha_0=\alpha$, $\lambda_0=\frac{\alpha}{M}$. Define $\gamma$ in the KAM theorem by setting
\begin{equation}\gamma=\gamma_0\gamma_s,\quad \gamma_s=\frac{1}{80}(\prod_{\mu=0}^{\infty}(2^{\mu}B_{\mu})^
{-\frac{1}{3\kappa^{\mu+1}}})^{1-2\beta'},\end{equation}
where $\gamma_0$ is the same parameter as before and $\gamma_s$ only depends on $n,\tau,E,s,\beta$. The smallness condition \eqref{17.10.16.1} of the iterative lemma is then satisfied by the assumption of the KAM theorem:
\begin{equation}\epsilon_0:=||X_{P_0}||_{s_0,r_0,p-1,\mathbf{a}_0;
\mathcal{O}_0}^{\lambda_0}\leq(\alpha\gamma)^{1+\beta}\leq
(\alpha_0\gamma_0\gamma_s)^{\frac{1}{1-2\beta'}}.\end{equation}
The other conditions \eqref{17.10.16.8}-\eqref{17.11.14.1} about the unperturbed frequencies are obviously true.

Hence, the iterative lemma applies, and we obtain a decreasing sequence of domains $D_{\nu}\times\mathcal{O}_{\nu}$, and a sequence of transformations
$$\Phi^{\nu}=\Phi_1\circ\cdots\Phi_{\nu}:D_{\nu}\times\mathcal{O}_{\nu}
\rightarrow D_0$$
such that $H\circ\Phi_{\nu}=N_{\nu}+P_{\nu}$ for $\nu\geq1$. Moreover, the estimates \eqref{17.12.24.3} and \eqref{17.10.16.2} hold. The following proof of the convergence is parallel to that in \cite{L-Y1}, where the small difference lies in that the norm in the source space $\mathcal{P}^{a,p}$ is $(s,r)$-weighted instead of $r$-weighted in \cite{L-Y1}. However, for completeness we still give the proof.

Shorten $||\cdot||_{s,r,p}$ as $||\cdot||_{s,r}$ and consider the operator norm
$$||L||_{s,r,\tilde{s},\tilde{r}}
=\sup_{W\neq0}\frac{||LW||_{s,r}}{||W||_{\tilde{s},\tilde{r}}}.$$
For $s\geq\tilde{s},r\geq\tilde{r}$, these norms satisfy $||AB||_{s,r,\tilde{s},\tilde{r}}\leq||A||_{s,r,s,r}
||B||_{\tilde{s},\tilde{r},\tilde{s},\tilde{r}}$, since $||W||_{s,r}\leq||W||_{\tilde{s},\tilde{r}}$. For $\nu\geq1$, by the chain rule, using \eqref{17.12.24.3} \eqref{18.3.22.4} \eqref{17.10.16.5} , we get
\begin{equation}||D\Phi^{\nu}||_{s_0,r_0,s_{\nu},r_{\nu};D_{\nu}\times
\mathcal{O}_{\nu}}\leq\prod_{\mu=1}^{\nu}||D\Phi_{\mu}||
_{s_{\mu},r_{\mu},s_{\mu},r_{\mu};D_{\mu}\times\mathcal{O}_{\mu}}
\leq\prod_{\mu=1}^{\infty}(1+\frac{\gamma_0}{2^{\mu+6}})\leq2,\end{equation}
\begin{eqnarray}\nonumber||D\Phi^{\nu}||^{lip}_{s_0,r_0,s_{\nu},
r_{\nu};D_{\nu}\times\mathcal{O}_{\nu}}&\leq&\sum_{\mu=1}^{\nu}
||D\Phi_{\mu}||
_{s_{\mu},r_{\mu},s_{\mu},r_{\mu};D_{\mu}\times\mathcal{O}_{\mu}}
\prod_{1\leq\rho\leq\nu,\rho\neq\nu}||D\Phi_{\rho}||_{s_{\rho},s_{\rho},
r_{\rho},r_{\rho};D_{\rho}\times\mathcal{O}_{\rho}}\\
\nonumber&\leq&2\sum_{\mu=1}^{\nu}||D\Phi_{\mu}-I||^{lip}
_{s_{\mu},r_{\mu},s_{\mu},r_{\mu};D_{\mu}\times\mathcal{O}_{\mu}}\\
\nonumber&\leq&2\sum_{\mu=1}^{\infty}\frac{M_{\mu}}
{\alpha_{2,\mu}}\frac{B_{\mu}}{\alpha_{2,\mu}
}\epsilon_{\mu}^{1-\beta'}\\
\nonumber&\leq&2\sum_{\mu=1}^{\infty}\frac{M_{\mu}}
{\alpha_02^{\mu}}(\frac{\gamma_0}{80})^{\kappa^{\mu}}\\
&\leq&\frac{M_0}{\alpha_0}.\end{eqnarray}
Thus, with the mean value theorem we obtain
\begin{eqnarray}
\nonumber||\Phi^{\nu+1}-\Phi^{\nu}||_{s_0,r_0;D_{\nu+1}
\times\mathcal{O}_{\nu+1}}&\leq&
||D\Phi^{\nu}||_{s_0,r_0,s_{\nu},r_{\nu};D_{\nu}\times\mathcal{O}_{\nu}}
||\Phi_{\nu+1}-\mbox{id}||_{s_{\nu},r_{\nu};D_{\nu+1}\times
\mathcal{O}_{\nu+1}}\\
&\leq&2||\Phi_{\nu+1}-\mbox{id}||_
{s_{\nu},r_{\nu};D_{\nu+1}\times\mathcal{O}_{\nu+1}},
\end{eqnarray}
\begin{eqnarray}
\nonumber||\Phi^{\nu+1}-\Phi^{\nu}||^{lip}_{s_0,r_0;D_{\nu+1}
\times\mathcal{O}_{\nu+1}}&\leq&||D\Phi^{\nu}||^{lip}_{s_0,r_0,s_{\nu},r_{\nu};
D_{\nu}\times\mathcal{O}_{\nu}}||\Phi_{\nu+1}-\mbox{id}||_{s_{\nu},r_{\nu};
D_{\nu+1}\times\mathcal{O}_{\nu+1}}\\
\nonumber&+&||D\Phi^{\nu}||_{s_0,r_0,s_{\nu},r_{\nu};
D_{\nu}\times\mathcal{O}_{\nu}}||\Phi_{\nu+1}-\mbox{id}||^{lip}
_{s_{\nu},r_{\nu};
D_{\nu+1}\times\mathcal{O}_{\nu+1}}\\
\nonumber&\leq&\frac{M_0}{\alpha_0}||\Phi_{\nu+1}-\mbox{id}||
_{s_{\nu},r_{\nu};
D_{\nu+1}\times\mathcal{O}_{\nu+1}}+2||\Phi_{\nu+1}-\mbox{id}||
^{lip}_{s_{\nu},r_{\nu};
D_{\nu+1}\times\mathcal{O}_{\nu+1}}.
\end{eqnarray}
It follows that
\begin{equation}\label{17.12.22.3}||\Phi^{\nu+1}-\Phi^{\nu}||^{\lambda_0}
_{s_0,r_0;D_{\nu+1}\times\mathcal{O}_{\nu+1}}\leq3||\Phi_{\nu+1}-\mbox{id}||
^{\lambda_{\nu}}_{s_{\nu},r_{\nu};D_{\nu+1}\times\mathcal{O}_{\nu+1}}.\end{equation}
From \eqref{17.12.24.3} and \eqref{17.12.22.3}, we get
\begin{equation}||\Phi^{\nu+1}-\Phi^{\nu}||^{\lambda_0}_{s_0,r_0;
D_{\nu+1}\times\mathcal{O}_{\nu+1}}\leq3\frac{B_{\nu}}{\alpha_{2,\nu}}
\epsilon_{\nu}^{1-\beta'}.\end{equation}
For every non-negative integer multi-index $k=(k_1,\cdots,k_n)$, by Cauchy's estimate we have
\begin{equation}||\partial_{x}^k(\Phi^{\nu+1}-\Phi^{\nu})||^{\lambda_0}
_{s_0,r_0;D_{\nu+2}\times\mathcal{O}_{\nu+1}}\leq3\frac{B_{\nu}}{\alpha_{2,\nu}}
\epsilon_{\nu}^{1-\beta'}\frac{k_1!\cdots k_n!}{(\frac{s_0}{2^{\nu+2}})^{|k|}},\end{equation}
the right side of which super-exponentially decay with $\nu$. This shows that $\Phi^{\nu}$ converge uniformly on $D_*\times\mathcal{O}_{\alpha}$, where $D_*=\mathbb{T}^n\times\{0\}\times\{0\}\times\{0\}$ and $\mathcal{O}_{\alpha}=\cap_{\nu\geq0}\mathcal{O}_{\nu}$, to a Lipschitz continuous family of smooth torus embeddings
$$\Phi:\mathbb{T}^n\times\mathcal{O}_{\alpha}\rightarrow\mathcal{P}^{a,p},$$
for which the estimate \eqref{17.7.21.22} holds. Similarly, the frequencies $\omega_{\nu}$ converge uniformly on $\mathcal{O}_{\alpha}$ to a Lipschitz continuous limit $\omega_*$, and the frequencies $\Omega_{\nu}$
converge uniformly on $D_*\times\mathcal{O}_{\alpha}$ to a regular limit $\Omega_*$, with the estimate \eqref{17.7.21.23} holding. Moreover, $X_H\circ\Phi=D\Phi\cdot  X_{N_*}$ on $D_*$ for each $\xi\in\mathcal{O}_{\alpha}$, where $N_*$ is the generalized normal form with frequencies $\omega_*$ and $\Omega_*$. Thus, the embedded tori are invariant under the perturbed Hamiltonian flow, and the flow on them is linear. Now it only remains to prove the claim about the set $\mathcal{O}\setminus\mathcal{O}_{\alpha}$, which is the subject of the next section.

%

\section{Measure Estimate}
We know
\begin{equation}\mathcal{O}\setminus\mathcal{O}_{\alpha}
=\Theta^1_{\alpha}\cup\Theta^2_{\alpha},
\end{equation}
\begin{equation}
\Theta^1_{\alpha}=\bigcup_{\nu\geq0}\bigcup_{k\in\mathbb{Z}^n\setminus
\{0\},|l|\leq2
\atop{l\neq e_{-j}-e_j}}\mathcal{R}_{kl}^{\nu}
(\alpha_{1,\nu}),\end{equation}
\begin{equation}\Theta^2_{\alpha}=\bigcup_{\nu\geq0}\bigcup_{k\in\mathbb{Z}^n, \pm j\in\mathbb{Z}_*\atop{|j|\leq\Pi_{\nu}}}\mathcal{R}_{k(-j)j}^{\nu}
(\alpha_{2,\nu}),\end{equation}
where
\begin{equation}\mathcal{R}_{kl}^{\nu}(\alpha_{1,\nu})=
\{\xi\in\mathcal{O}_{\nu}:|\langle k,\omega_{\nu}(\xi)\rangle+\langle l,\bar{\Omega}_{\nu}(\xi)\rangle|<\alpha_{1,\nu}\frac{\langle l\rangle_{\infty}}{\langle k\rangle^{\tau}}\},\end{equation}
\begin{equation}\label{17.10.24.1}\mathcal{R}_{k(-j)j}^{\nu}(\alpha_{2,\nu})=
\{\xi\in\mathcal{O}_{\nu}:|\langle k,\omega_{\nu}(\xi)\rangle+ \bar{\Omega}_{\nu,-j}(\xi)-\bar{\Omega}_{\nu,j}(\xi)|
<\alpha_{2,\nu}\frac{|j|}{\langle k\rangle^{\tau}}\}.\end{equation}
 Here, $\omega_{\nu}$ and $\bar{\Omega}_{\nu}$
satisfy \eqref{17.10.16.8}-\eqref{17.10.16.10} \eqref{17.11.14.1} on $\mathcal{O}_{\nu}$, and especially, $\omega_0=\omega,\bar{\Omega}_0=\Omega$ are the frequencies of the unperturbed system.
\begin{lem} \label{lem18.4.1.3}If $\gamma_0$ is sufficiently small and $\tau\geq n+3$, then
\begin{equation}\label{17.11.14.5}|\Theta_{\alpha}^1|\leq
c\rho^{n-1}\alpha,\end{equation}
where $\rho:=\mbox{diam}\mathcal{O}$ represents the diameter of $\mathcal{O}$ and $c>0$ is a constant depends on $n,E,M_3$ and $m$.
\end{lem}
\begin{proof} 
%
%

By \eqref{17.12.28.2} and the definition of $J_{\nu+1}$, we have
  \begin{equation}\label{17.12.26.1}B_{\nu}\epsilon_{\nu}
  \leq\frac{\alpha_{1,\nu}-
  \alpha_{1,\nu+1}}{3J_{\nu+1}^{\tau+1}}.\end{equation}
 For $\langle k\rangle\leq J_{\nu+1}$, $|l|\leq2, l\neq e_{-j}-e_j $, by \eqref{17.10.16.3} \eqref{17.12.26.1} we obtain
  \begin{eqnarray}\nonumber|\langle k,\omega_{\nu+1}-\omega_{\nu}\rangle+
  \langle l,\bar{\Omega}_{\nu+1}-\bar{\Omega}_{\nu}\rangle|
  &\leq&|k||\omega_{\nu+1}-\omega_{\nu}|+2\langle l\rangle_{\infty}|\bar{\Omega}_{\nu+1}-\Omega_{\nu}|_{-1}\\
  \nonumber&\leq&3\langle k\rangle\langle l\rangle_{\infty}B_{\nu}\epsilon_{\nu}\\
  \nonumber&\leq&(\alpha_{1,\nu}-\alpha_{1,\nu+1})\frac{\langle k\rangle\langle l\rangle_{\infty}}{J_{\nu+1}^{\tau+1}}\\
  &\leq&(\alpha_{1,\nu}-\alpha_{1,\nu+1})\frac{\langle l\rangle_{\infty}}{\langle k\rangle^{\tau}}\end{eqnarray}
  on $\mathcal{O}_{\nu+1}$, which implies $\mathcal{R}^{\nu+1}_{kl}(\alpha_{1,\nu+1})\subset
  \mathcal{R}^{\nu}_{kl}(\alpha_{1,\nu})$. Hence
  \begin{equation}\Theta^1_{\alpha}=\bigcup_{\nu\geq0}
  \bigcup_{|k|>J_{\nu},|l|\leq2
\atop{l\neq e_{-j}-e_j}}\mathcal{R}_{kl}^{\nu}
(\alpha_{1,\nu}).\end{equation}

We only need to give the proof of the most difficult case
that $l$ has two non-zero components of opposite sign. In this case, rewriting
\begin{equation}\mathcal{R}^{\nu}_{kij}(\alpha_{1,\nu})
=\{\xi\in\mathcal{O}_{\nu}:|\langle k,\omega_{\nu}(\xi)\rangle+ \bar{\Omega}_{\nu,i}(\xi)-\bar{\Omega}_{\nu,j}(\xi)|<
\alpha_{1,\nu}\frac{\max\{|i|,|j|\}}{\langle k\rangle^{\tau}}\},\quad i\neq \pm j.\end{equation}

Now we consider a fixed $\mathcal{R}^{\nu}_{kij}(\alpha_{1,\nu})$ with $i\neq \pm j$.

If $|k|<\frac{9m_{\nu}}{10E_{\nu}}|i^2-j^2|$, we get $|\langle k,\omega_{\nu}(\xi)\rangle|<\frac{9m_{\nu}}{10}|i^2-j^2|$. By \eqref{17.10.16.10} and $\alpha_{1,\nu}\leq\frac{m_{\nu}}{10}$, we know $\mathcal{R}^{\nu}_{kij}(\alpha_{1,\nu})$ is empty.

If $|k|\geq\frac{9m_{\nu}}{10E_{\nu}}|i^2-j^2|$, we have $|k|\geq(\frac{9}{10})^3\frac{m}{E}(|i|+|j|)$.
In view of \eqref{17.11.14.1}, if
$$\inf_{\xi\in\mathcal{O}_{\nu}}|\langle k,\omega_{\nu}(\xi)\rangle+
\bar{\Omega}_{\nu,i}(\xi)-\bar{\Omega}_{\nu,j}(\xi)|
\geq \frac12M_{3,\nu}\max\{|k|,|i|+|j|\},$$
then we know $\mathcal{R}^{\nu}_{kij}(\alpha_{1,\nu})$ is empty by noticing that $\alpha_{1,\nu}\leq \frac12M_{3,\nu}$; if
$$\inf_{\xi-\zeta// v_{kl}}\frac{|\Delta_{\xi\zeta}(\langle k,\omega_{\nu}\rangle+\bar{\Omega}_{\nu,i}-\bar{\Omega}_{\nu,j})|}
{|\xi-\zeta|}\geq\frac12
M_{3,\nu}\max\{|k|,|i|+|j|\},$$
then we have
\begin{equation}|\mathcal{R}_{kij}^{\nu}(\alpha_{1,\nu})|
\leq(\mbox{diam}
\mathcal{O}_{\nu})^{n-1}\frac{4\alpha_{1,\nu}\max\{|i|,|j|\}}
{M_{3,\nu}\max\{|k|,|i|+|j|\}
\langle k\rangle^{\tau}}\leq4(\frac{10}{9})^2
\rho^{n-1}\frac{\alpha}{M_3\langle k\rangle^{\tau}}.
\end{equation}
Consequently, we have that for any $|k|>J_{\nu}$,
\begin{equation}|\bigcup_{i\neq \pm j}\mathcal{R}_{kij}^{\nu}(\alpha_{1,\nu})|\leq
\bigcup_{|i|+|j|\leq(10/9)^3(E/m)|k|}|\mathcal{R}_{kij}^{\nu}
(\alpha_{1,\nu})|\leq c_1\rho^{n-1}\frac{\alpha}{\langle k\rangle^{\tau-2}},\end{equation}
where $c_1$ depends on $E,m,M_3$. Furthermore, since $\tau\geq n+3$,
we have
\begin{equation}|\bigcup_{|k|>J_{\nu},i\neq \pm j}\mathcal{R}_{kij}^{\nu}(\alpha_{1,\nu})|\leq c_1c_2\rho^{n-1}\frac{\alpha}{1+J_{\nu}},\end{equation}
where $c_2$ depends only on $n$. The sum of the latter inequality over all $\nu$ converges, and thus we obtain \eqref{17.11.14.5}.
\end{proof}

\begin{lem} \label{lem18.4.1.2}If $\gamma_0$ is sufficiently small and $\tau>n+1$, then
\begin{equation}\label{18.3.22.1}|\Theta_{\alpha}^2|\leq
c\rho^{n-1}\alpha,\end{equation}
where $c>0$ is a constant depends on $n,\tau,M_3$.
\end{lem}
\begin{proof}
We consider a fixed $\mathcal{R}^{\nu}_{k(-j)j}(\alpha_{2,\nu})$ with $|j|\leq|\Pi_{\nu}|$.
In view of \eqref{17.11.14.1}, if
$$\inf_{\xi\in\mathcal{O}_{\nu}}|\langle k,\omega_{\nu}(\xi)\rangle+
\bar{\Omega}_{\nu,-j}(\xi)-\bar{\Omega}_{\nu,j}(\xi)|\geq \frac12M_{3,\nu}\max\{|k|,2|j|\},$$
then we know $\mathcal{R}^{\nu}_{k(-j)j}(\alpha_{2,\nu})$ is empty by noticing that $\alpha_{2,\nu}\leq \frac12M_{3,\nu}$;
if
$$\inf_{\xi-\zeta// v_{kl}}\frac{|\Delta_{\xi\zeta}(\langle k,\omega_{\nu}\rangle+\bar{\Omega}_{\nu,-j}-\bar{\Omega}_{\nu,j}
)|}{|\xi-\zeta|}\geq\frac12
M_{3,\nu}\max\{|k|,2|j|\},$$
then we have
\begin{equation}|\mathcal{R}_{k(-j)j}^{\nu}(\alpha_{2,\nu})|
\leq(\mbox{diam}
\mathcal{O}_{\nu})^{n-1}\frac{4\alpha_{2,\nu}|j|}
{M_{3,\nu}\max\{|k|,2|j|\}\langle k\rangle^{\tau}}\leq2\rho^{n-1}\frac{\alpha_{2,\nu}}
{M_{3,\nu}\langle k\rangle^{\tau}}.
\end{equation}
Consequently, we have that for any $k\in\mathbb{Z}^n$,
\begin{eqnarray}\nonumber|\bigcup_{|j|\leq\Pi_{\nu}
}\mathcal{R}_{k(-j)j}^{\nu}(\alpha_{2,\nu})|&\leq&
\bigcup_{|j|\leq\Pi_{\nu}}|\mathcal{R}_{k(-j)j}^{\nu}
(\alpha_{2,\nu})|\\
\nonumber&\leq&2\rho^{n-1}\frac{\alpha_{2,\nu}}{M_{3,\nu}\langle k\rangle^{\tau}}(2\Pi_{\nu})\\
\nonumber&=&4\rho^{n-1}\frac{\alpha}{M_{3,\nu}2^{\nu}\langle k\rangle^{\tau}}.
%
\end{eqnarray}
Moreover, since $\tau>n+1$, we have
\begin{equation}\label{17.12.27.1}|
\bigcup_{k\in\mathbb{Z}^{n},\pm j\in\mathbb{Z}_*
\atop{|j|\leq\Pi_{\nu}}}\mathcal{R}_{k(-j)j}^{\nu}
(\alpha_{2,\nu})|\leq c\rho^{n-1}\frac{\alpha}{2^{\nu}},\end{equation}
where $c>0$ depends on $n,\tau,M_3$. The sum of the latter inequality over all $\nu$ converges and we finally obtain the estimate \eqref{18.3.22.1}.
\end{proof}

\section{Appendix}
\begin{lem}\label{lem18.4.1.1}
For $\sigma>0$ and $\nu>0$, the following inequalities hold true:
\begin{equation}\label{17.7.20.1}\sum_{k\in\mathbb{Z}^n}e^{-2|k|\sigma}
\leq\frac{1}{\sigma^n}(1+e)^n,\end{equation}
\begin{equation}\label{17.7.20.2}\sum_{k\in\mathbb{Z}^n}e^{-2|k|\sigma}|k|^{\nu}
\leq(\frac{\nu}{e})^{\nu}\frac{1}{\sigma^{\nu+n}}(1+e)^n,\end{equation}
\begin{equation}\label{17.11.10.1}\sup_{k\in\mathbb{Z}^n}(e^{-|k|\sigma}|k|^{\nu})
\leq(\frac{\nu}{e})^{\nu}\frac{1}{\sigma^{\nu}}.\end{equation}
\end{lem}
\begin{proof}\eqref{17.7.20.1} \eqref{17.7.20.2} can be found on page 22 in \cite{Bogo}, while \eqref{17.11.10.1} can be proved by direct calculation.
\end{proof}
\begin{lem}\label{lem18.3.19.1}Let $u(x)$ be an analytic function on $D(s)$ with finite momentum majorant norm. Then for $0<\sigma<s$, we have
\begin{equation}\label{App1}
|u|_{s-\sigma,\tau+1}\leq(\frac{\tau+1}{e})^{\tau+1}
\frac{1}{\sigma^{\tau+1}}|u|_{s,\mathbf{a},0}.\end{equation}
\end{lem}
\begin{proof} By the definition of $|\cdot|_{s,\tau+1}$ and $|\cdot|_{s,\mathbf{a},0}$, we otain
\begin{eqnarray}\nonumber|u|_{s-\sigma,\tau+1}
&=&\sum_{k\in\mathbb{Z}^n}|\hat{u}_k|e^{|k|(s-\sigma)}|k|^{\tau+1}
\\
\nonumber&\leq&\sup_{k\in\mathbb{Z}^n}(e^{-|k|\sigma}
|k|^{\tau+1})\sum_{k\in\mathbb{Z}^n}|\hat{u}_k|e^{|k|s}\\
\label{18.3.19.2}&\mathop\leq^{}&(\frac{\tau+1}{e})^{\tau+1}
\frac{1}{\sigma^{\tau+1}}\sum_{k\in\mathbb{Z}^n}|\hat{u}_k|e^{|k|s}\\
\nonumber&\leq&(\frac{\tau+1}{e})^{\tau+1}
\frac{1}{\sigma^{\tau+1}}\sum_{k\in\mathbb{Z}^n}|
\hat{u}_k|e^{|k|s}e^{\mathbf{a}|\sum_{b=1}^nk_bj_b|}\\
\nonumber&=&(\frac{\tau+1}{e})^{\tau+1}
\frac{1}{\sigma^{\tau+1}}|u|_{s,\mathbf{a},0},
\end{eqnarray}
where in \eqref{18.3.19.2} we use \eqref{17.11.10.1}.
\end{proof}

\begin{lem}\label{lem17.9.7.1} Let $R=(R_{ij})_{i,j\in\mathbb{Z}_*}$ be a matrix depending on $x\in D(s)$ such that the corresponding Hamiltonian
vector field $X_{\langle Rz,\bar{z}\rangle}$ has finite momentum majorant norm on $D(s,r)$.
Suppose $F=(F_{ij})_{i,j\in\mathbb{Z}_*}$ is another matrix depending on $x$ whose elements satisfy
\begin{equation}\label{17.9.30.1}\sum_{k\in\mathbb{Z}^n}|\hat{F}_{ijk}|
e^{|k|(s-\sigma)}e^{\mathbf{a}|\pi(k,i-j)|}\leq\frac{1}{\max\{|i|,|j|\}}
\sum_{k\in\mathbb{Z}^n}|\hat{R}_{ijk}|e^{|k|s}
e^{\mathbf{a}|\pi(k,i-j)|}\end{equation}
for $0<\sigma<\min\{1,s/2\}$. Then for $p\geq0$, $\mathbf{a}\geq0$, we have
\begin{equation}\label{17.9.30.2}||X_{\langle Fz,\bar{z}\rangle}||_{s-2\sigma,r,p,\mathbf{a}}
\leq\frac{3}{\sigma}||X_{\langle Rz,\bar{z}\rangle}||_{s,r,p-1,\mathbf{a}}.\end{equation}
\end{lem}
\begin{proof}
By \eqref{7.13.2}, we obtain
\begin{equation}X_{\langle Fz,\bar{z}\rangle}=
(0,-(\sigma_{j_b}\partial_{x_b}\langle Fz,\bar{z}\rangle)_{1\leq b\leq n},-\mathbf{i}(\sigma_j\partial_{\bar{z}_j}\langle Fz,\bar{z}\rangle
)_{j\in\mathbb{Z}_*},\mathbf{i}(\sigma_j\partial_{z_j}\langle Fz,\bar{z}\rangle
)_{j\in\mathbb{Z}_*})^T\end{equation}
and $X_{\langle Rz,\bar{z}\rangle}$ is defined similarly.
In view of
$\langle Fz,\bar{z}\rangle=\sum_{i,j\in\mathbb{Z}_*\atop{k\in\mathbb{Z}^n}}
\hat{F}_{ijk}e^{\mathbf{i}k\cdot x}z_i\bar{z}_j,$
we have
\begin{equation*}-(\sigma_{j_b}\partial_{x_b}\langle Fz,\bar{z}\rangle)_{1\leq b\leq n}=(\sum_{i,j\in\mathbb{Z}_*\atop{k\in\mathbb{Z}^n}}-\mathbf{i}\sigma_{j_b}k_b
\hat{F}_{ijk}e^{\mathbf{i}k\cdot x}z_i\bar{z}_j)_{1\leq b\leq n},\end{equation*}
\begin{equation*}-\mathbf{i}(\sigma_j\partial_{\bar{z}_j}\langle Fz,\bar{z}\rangle
)_{j\in\mathbb{Z}_*}=(\sum_{i\in\mathbb{Z}_*\atop{k}\in\mathbb{Z}^n}
-\mathbf{i}\sigma_j\hat{F}_{ijk}e^{\mathbf{i}k\cdot x}z_i)_{j\in\mathbb{Z}_*},\end{equation*}
\begin{equation*}\mathbf{i}(\sigma_j\partial_{z_j}\langle Fz,\bar{z}\rangle
)_{j\in\mathbb{Z}_*}=(\sum_{i\in\mathbb{Z}_*\atop{k}\in\mathbb{Z}^n}
\mathbf{i}\sigma_j\hat{F}_{jik}e^{\mathbf{i}k\cdot x}\bar{z}_i)_{j\in\mathbb{Z}_*}.\end{equation*}
The first ingredient of $||X_{\langle Fz,\bar{z}\rangle}||_{s-2\sigma,r,p,\mathbf{a}}$ is zero, and the second ingredient can be controlled by the third ingredient (or the fourth ingredient):
\begin{eqnarray}\nonumber&&\sup_{(y,z,\bar{z})\in D(r)}\frac{1}{r^2}|(\sum_{i,j\in\mathbb{Z}_*\atop{k\in\mathbb{Z}^n}}
|-\mathbf{i}\sigma_{j_b}k_b
\hat{F}_{ijk}|e^{\mathbf{a}|\pi(k,i-j)|}e^{|k|(s-2\sigma)}|z_i||\bar{z}_j|)
_{1\leq b\leq n}|_1\\
\nonumber&&=\sup_{(y,z,\bar{z})\in D(r)}\frac{1}{r^2}\sum_{i,j\in\mathbb{Z}_*\atop{k\in\mathbb{Z}^n}}
|k||\hat{F}_{ijk}|e^{\mathbf{a}|\pi(k,i-j)|}e^{|k|(s-2\sigma)}|z_i||\bar{z}_j|\\
\nonumber&&\leq\sup_{k\in\mathbb{Z}^n}|k|e^{-|k|\sigma}\cdot\sup_{(y,z,\bar{z})\in D(r)}\frac{1}{r^2}\sum_{
i,j\in\mathbb{Z}_*}\big(\sum_{k\in\mathbb{Z}^n}|\hat{F}_{ijk}|
e^{\mathbf{a}|\pi(k,i-j)|}e^{|k|(s-\sigma)}\big)|z_i||\bar{z}_j|\\
\nonumber&&\leq\frac{1}{e\sigma}\sup_{(y,z,\bar{z})\in D(r)}\frac{1}{r^2}\langle(\sum_{i\in\mathbb{Z}_*
\atop{k\in\mathbb{Z}^n}}|\hat{F}_{ijk}|e^{\mathbf{a}|\pi(k,i-j)|}
e^{|k|(s-\sigma)}|z_i|)
_{j\in\mathbb{Z}_*},|\bar{z}|\rangle\\
\nonumber&&\leq\frac{1}{e\sigma}\sup_{(y,z,\bar{z})\in D(r)}\frac{1}{r^2}||(\sum_{i\in\mathbb{Z}_*
\atop{k\in\mathbb{Z}^n}}|\hat{F}_{ijk}|e^{\mathbf{a}|\pi(k,i-j)|}
e^{|k|(s-\sigma)}|z_i|)_{j\in\mathbb{Z}_*}||_{-a,-p}
||\bar{z}|||_{a,p}\\
\label{17.11.6.2}&&\leq\frac{1}{e\sigma}\sup_{(y,z,\bar{z})\in D(r)}\frac{1}{r}||(\sum_{i\in\mathbb{Z}_*\atop{k\in\mathbb{Z}^n}}
|-\mathbf{i}\sigma_j\hat{F}_{ijk}|e^{\mathbf{a}|\pi(k,i-j)|}
e^{|k|(s-\sigma)
}|z_i|)_{j\in\mathbb{Z}_*}||_{a,p}.
\end{eqnarray}
Moreover, the estimates for the third and fourth ingredients of $||X_{\langle Fz,\bar{z}\rangle}||_{s-2\sigma,r,p,\mathbf{a}}$ are parallel. Thus there only remains the estimate for the third ingredient:
\begin{eqnarray}\nonumber&&\sup_{(y,z,\bar{z})\in D(r)}\frac{1}{r}||(\sum_{i\in\mathbb{Z}_*\atop{k\in\mathbb{Z}^n}}
|-\mathbf{i}\sigma_j\hat{F}_{ijk}|e^{\mathbf{a}|\pi(k,i-j)|}e^{|k|(s-\sigma)
}|z_i|)_{j\in\mathbb{Z}_*}||_{a,p}\\
\nonumber&&\leq\sup_{(y,z,\bar{z})\in D(r)}\frac{1}{r}||(\frac{1}{|j|}\sum_{i\in\mathbb{Z}_*\atop{k\in\mathbb{Z}^n}}
|\hat{R}_{ijk}|e^{\mathbf{a}|\pi(k,i-j)|}e^{|k|s
}|z_i|)_{j\in\mathbb{Z}_*}||_{a,p}\\
\label{17.9.26.7}&&=\sup_{(y,z,\bar{z})\in D(r)}\frac{1}{r}||(\sum_{i\in\mathbb{Z}_*\atop{k\in\mathbb{Z}^n}}
|-\mathbf{i}\sigma_j\hat{R}_{ijk}|e^{\mathbf{a}|\pi(k,i-j)|}e^{|k|s
}|z_i|)_{j\in\mathbb{Z}_*}||_{a,p-1}.
\end{eqnarray}
%
%
In view of \eqref{17.11.6.2} and \eqref{17.9.26.7}, we get \begin{equation}||X_{\langle Fz,\bar{z}\rangle}||_{s-2\sigma,r,p,\mathbf{a}}
\leq(\frac{2}{\sigma}+1)||X_{\langle Rz,\bar{z}\rangle}||_{s,r,p-1,\mathbf{a}},\end{equation}
and thus \eqref{17.9.30.2}.
\end{proof}

By \eqref{17.11.8.2}, if $X:D(s,r)\rightarrow\mathcal{P}^{a,q}$ with $||X||_{s,r,q,\mathbf{a}}<+\infty$, then $X$ is analytic, namely the
Frechet differential $D(s,r)\ni v\mapsto dX(v)\in\mathcal{L}(D(s,r),\mathcal{P}^{a,q})$ is continuous.
 The commutator of two vector fields $X:D(s,r)\rightarrow\mathcal{P}^{a,q},
Y:D(s,r)\rightarrow\mathcal{P}^{a,p}$ is
\begin{equation}[X,Y](v):=dX(v)[Y(v)]-dY(v)[X(v)],\quad\forall
\ v\in D(s,r).\end{equation}
The next lemma is the estimate for the momentum majorant norm of
the commutator of two vector fields. It is Proposition 2.1 in \cite{Berti1} with small modification, that is, the definition space and target space of one vector field are different.
\begin{lem}\label{lem17.12.22.1}Let  $X:D(s,r)\rightarrow\mathcal{P}^{a,q}$ and $Y:D(s,r)\rightarrow\mathcal{P}^{a,p}$ with $||X||_{s,r,q,\mathbf{a}},||Y||_{s,r,p,\mathbf{a}}<+\infty$. Then for $s/2\leq s'<s$ and $r/2\leq r'<r$,
\begin{equation}\label{17.12.22.2}||[X,Y]||_{s',r',q,\mathbf{a}}
\leq 2^{2n+3}\max\{\frac{s}{s-s'},\frac{r}{r-r'}\}||X||_{s,r,q,\mathbf{a}}
||Y||_{s,r,p,\mathbf{a}}.\end{equation}
\end{lem}
\begin{proof} For $\mathbf{a}=0$,
the proof is parallel to that of Lemma 2.15 in \cite{Berti2}, in which the following Cauchy estimate (Lemma 2.14 in \cite{Berti2}) is essential
\begin{equation}\label{17.12.22.4}\sup_{v\in D(s',r')}||dW(v)||_{\mathcal{L}
((\mathcal{P}^{a,p},||\cdot||_{s,r,p}),(\mathcal{P}^{a,p},
||\cdot||_{s',r',p}))}\leq 4\max
\{\frac{s}{s-s'},\frac{r}{r-r'}\}||W||_{s,r,p;
D(s,r)},\end{equation}
where
$W:D(s,r)\rightarrow\mathcal{P}^{a,p}\ \mbox{with}\  ||W||_{s,r,p,0}<+\infty.$
Note that $X:D(s,r)\rightarrow\mathcal{P}^{a,q}$. The difference between $X$ and $W$ lies in that the definition space and target space are different. However, parallelly to the proof of \eqref{17.12.22.4}, we can also get the Cauchy estimate
\begin{equation}\label{17.12.22.1}\sup_{v\in D(s',r')}||dX(v)||_{\mathcal{L}((\mathcal{P}^{a,p},||\cdot||_{s,r,p}),
(\mathcal{P}^{a,q},||\cdot||_{s',r',q}))}
\leq 4\max\{\frac{s}{s-s'},
\frac{r}{r-r'}\}||X||_{s,r,q;D(s,r)}.\end{equation}
Therefore, following the proof of Lemma 2.15 in \cite{Berti2} and using \eqref{17.12.22.1}, we get
\begin{equation}\label{17.12.22.7}||[X,Y]||_{s',r',q,0}
\leq 2^{2n+3}\max\{\frac{s}{s-s'},\frac{r}{r-r'}\}||X||_{s,r,q,0}
||Y||_{s,r,p,0}.\end{equation}

For $\mathbf{a}>0$, the proof follows the idea of Proposition 2.1 in \cite{Berti1}. Write
$X=\sum_{h\in\mathbb{Z}}X_h,$
where $X_h$ is the sum of the monomials with $\pi(X)=h$. Similarly, write $Y=\sum_{h\in\mathbb{Z}}Y_h$, and thus
$$[X,Y]=\sum_{h_1,h_2\in\mathbb{Z}}[X_{h_1},Y_{h_2}].$$
It is easy to see that
\begin{equation}\label{17.12.22.5}\pi([X_{h_1},Y_{h_2}])=h_1+h_2=\pi(X_{h_1})
+\pi(Y_{h_2}).\end{equation}
By \eqref{17.12.22.7} and \eqref{17.12.22.5}, the estimate
\begin{eqnarray}\nonumber
||[X_{h_1},Y_{h_2}]||_{s',r',q,\mathbf{a}}&=&
e^{\mathbf{a}(h_1+h_2)}||[X_{h_1},Y_{h_2}]||_{s',r',q,0}\\
\nonumber&\leq&2^{2n+3}\max\{\frac{s}{s-s'},\frac{r}{r-r'}\}
e^{\mathbf{a}(h_1+h_2)}||X_{h_1}||_{s,r,q,0}||Y_{h_2}||_{s,r,p,0}\\
\label{17.12.22.6}&=&2^{2n+3}\max\{\frac{s}{s-s'},\frac{r}{r-r'}\}
||X_{h_1}||_{s,r,q,\mathbf{a}}||Y_{h_2}||_{s,r,p,\mathbf{a}}\end{eqnarray}
holds true. Finally, \eqref{17.12.22.2} for $\mathbf{a}>0$ is obtained by
 summing $h_1$ and $h_2$.
\end{proof}
The following lemma is the estimate for the momentum majorant norm of the transformed Hamiltonian vector field. It is Lemma 2.17 in \cite{Berti2} with small modification, that is, the definition space and target space of the Hamiltonian vector field $X_H$ are different.
\begin{lem}\label{lem17.12.23.1}Let $r/2\leq r'<r,s/2\leq s'<s$ and $F$ with
\begin{equation}2^{2n+5}e\max\{\frac{s}{s-s'},\frac{r}{r-r'}\}
||X_F||_{s,r,p,\mathbf{a}}<1.\end{equation}
Then the time $1$-Hamiltonian flow
$\Phi_F^1:D(s',r')\rightarrow D(s,r)$
is well defined, analytic, symplectic, and $\forall \ H$ with $||X_H||_{s,r,q,\mathbf{a}}<+\infty$, we have
\begin{equation}\label{17.12.23.2}||X_{H\circ\Phi_F^1}||_{s',r',q,\mathbf{a}}
\leq\frac{||X_H||_{s,r,q,\mathbf{a}}}{1-2^{2n+5}e
\max\{\frac{s}{s-s'},\frac{r}{r-r'}\}||X_F||_{s,r,p,
\mathbf{a}}}.\end{equation}
\end{lem}
\begin{proof}
With the help of Lemma \ref{lem17.12.22.1} which deal with the case that the definition space and target space of one vector field are different, the proof is parallel to Lemma 2.17 in \cite{Berti2}. Thus, we omit the details here.
\end{proof}

\textbf{Acknowledgement.} {Part of this article was written when the first author was visiting School of Mathematics, Sichuan University. She thanks the school for its pleasant working atmosphere and Professor Weinian Zhang for his invaluable help. The second author was supported by National Research
Foundation of China (No. 11671280).}

\end{document}